\documentclass[12pt]{amsart}
\usepackage{amsmath}
\usepackage[psamsfonts]{amssymb}
\usepackage[all]{xy}
\usepackage{mathtools}
\usepackage{graphicx}
\usepackage[T1]{fontenc}
\usepackage[bitstream-charter]{mathdesign}
\usepackage[hyphens]{url}
\usepackage{hyperref}
\hypersetup{colorlinks=false,breaklinks=true}
\makeatletter
\g@addto@macro{\UrlBreaks}{\UrlOrds}
\makeatother
\usepackage{color}
\usepackage{dsfont}
\usepackage{url}
\newtheorem{theorem}{Theorem}[section]
\newtheorem{prop}[theorem]{Proposition}
\newtheorem{lemma}[theorem]{Lemma}
\newtheorem{cor}[theorem]{Corollary}

\newtheorem{conj}[theorem]{Conjecture}

\newtheorem{ex}[theorem]{Example}
\newtheorem{dfn}[theorem]{Definition}
\newtheorem{remark}[theorem]{Remark}

\def\k{{\bf k}}

\def\B{\mathbb{B}}
\def\R{\mathbb{R}}
\def\Z{\mathbb{Z}}
\def\N{\mathbb{N}}
\def\H{\mathbb{H}}

\def\C{\mathbb{C}}

\def\h{\mathfrak h}
\def\f{\mathfrak f}
\def\g{\mathfrak g}
\def\mod{\rm{mod}}

\def\SH{{\rm SH}}
\def\HF{{\rm HF}}
\def\Ham{{\rm Ham}}
\def\Diff{{\rm Diff}}

\newcommand{\ind}{\operatorname{ind}}
\newcommand{\Vol}{\operatorname{Vol}}
\newcommand{\im}{\operatorname{im}}

\begin{document}

\title{Persistence modules, symplectic Banach-Mazur distance and Riemannian metrics}

\author[Vuka\v sin Stojisavljevi\'c ]{Vuka\v sin Stojisavljevi\'c}
\address{School of Mathematical Sciences, Tel Aviv University, Ramat Aviv, Tel Aviv 69978
Israel}
\email{vukasin@mail.tau.ac.il}

\author[Jun Zhang]{Jun Zhang}
\address{Department of Mathematics and Statistics, University of Montreal, C.P. 6128 Succ. Centre-Ville Montreal, QC H3C 3J7, Canada}
\email{jun.zhang.3@umontreal.ca}
\address{School of Mathematical Sciences, Tel Aviv University, Ramat Aviv, Tel Aviv 69978
Israel}
\email{junzhang@mail.tau.ac.il}

\date{\today} 

\maketitle

\begin{abstract} We use persistence modules and their corresponding barcodes to quantitatively distinguish between different fiberwise star-shaped domains in the cotangent bundle of a fixed manifold. The distance between two fiberwise star-shaped domains is measured by a non-linear version of the classical Banach-Mazur distance, called symplectic Banach-Mazur distance and denoted by $d_{SBM}.$ The relevant persistence modules come from filtered symplectic ho\-mo\-lo\-gy and are stable with respect to $d_{SBM}.$ Our main focus is on the space of unit codisc bundles of orientable surfaces of positive genus, equipped with Riemannian metrics. We consider some questions about large-scale geometry of this space and in particular we give a construction of a quasi-isometric embedding of $(\R^n,|\cdot |_\infty)$ into this space for all $n\in \N.$ On the other hand, in the case of domains in $T^*S^2$, we can show that the corresponding metric space has infinite diameter. Finally, we discuss the existence of closed geodesics whose energies can be controlled. \end{abstract}
\tableofcontents

\section{Introduction}

Recently, the technique of persistence modules and barcodes has been successfully used in symplectic and contact topology. For instance, \cite{PS16}, \cite{UZ16}, \cite{Zha16}, \cite{PSS17}, \cite{Ste18}, \cite{KS18} and \cite{She19b} used persistence modules constructed from Floer homology to study questions in Hamiltonian dynamics, while \cite{BHS18}, \cite{LSV18}, \cite{She18}, \cite{Kaw19} and \cite{She19a} applied persistence techniques in the framework of $C^0$-symplectic topology. On the other hand, in \cite{Fraser}, persistence modules defined using generating function homology were considered, while \cite{RS18} used barcodes to deduce displacement energy bounds for Legendrian submanifolds. Persistence modules coming from filtered symplectic homology of star-shaped domains in $\C^{n}$ were considered in \cite{Ush18}. We refer the reader to \cite{PRSZ17} for various applications of persistence modules and barcodes in geometry and analysis. In this paper, we will consider persistence modules coming from filtered symplectic homology in order to study fiberwise star-shaped domains in the cotangent bundle of a fixed manifold in a quantitative fashion.

\subsection{The metrical set-up}

The quantitative perspective which we wish to adopt has its roots in the concept of Banach-Mazur distance, initially appearing in functional analysis with the aim of comparing convex bodies. Let $M$ be a closed, orientable manifold of dimension $n$. Its cotangent bundle $T^*M$ is equipped with a canonical symplectic form $\omega_{can} = d \lambda_{can},$ where $\lambda_{can}$ is the Liouville form, and a canonical vector field $X$ given by $i_X (\omega_{can}) = \lambda_{can}$ called Liouville vector field. We call a domain $U\subset T^*M$ \textit{admissible} if it is a compact, fiberwise star-shaped domain, centered at the zero section $0_M\subset U \subset T^*M,$ whose boundary $\partial U$ is smooth and such that $ X \pitchfork \partial U.$ Restriction of the Liouville form to the boundary of an admissible domain is a contact form, i.e. $(\partial U, \lambda_{can}|_{\partial U})$ is a contact manifold. Denote
\[ \mathcal C_M = \{ \mbox{admissible domains $U$ in $T^*M$} \}.\]

For two admissible domains $U, V \in \mathcal C_M $, an embedding $\phi: U\to V$ satisfying $\phi^*\lambda_{can}-\lambda_{can}=df$ for some smooth function $f:U\rightarrow \R$ is called a {\it Liouville embedding.}  Denote the set of homotopy classes of free loops in $M$ by $\tilde{\pi}_1(M)$. Notice that any $U \in \mathcal C_M$ deformation retracts to the zero section $0_M$ of $T^*M,$ and the projection $\pi:T^*M\rightarrow M$ restricted to $U$ induces a homotopy equivalence $\pi|_U:U\rightarrow M.$ Thus, any Liouville embedding $\phi: U \to V$ between two admissible domains in $T^*M$ induces a map $\phi_*$ on $\tilde{\pi}_1(M)$. Majority of maps which we will consider in this paper will be a special type of Liouville embeddings which are defined as follows. 

\begin{dfn} \label{dfn-pi1-le}
Given two admissible domains $U, V \in \mathcal C_M$, a Liouville embedding $\phi: U \to V$ is {\it $\tilde{\pi}_1$-trivial} if $\phi_* \alpha = \alpha$ for all $\alpha \in {\tilde \pi}_1(M)$. We adopt the notation $U \xhookrightarrow{\phi} V$ for a ${\tilde \pi}_1$-trivial Liouville embedding $\phi: U \to V$.
\end{dfn}
 One readily checks that the composition of two ${\tilde \pi}_1$-trivial Liouville embeddings is again a ${\tilde \pi}_1$-trivial Liouville embedding. The following definition modifies a key definition from \cite{GU17}.

\begin{dfn}  Let $U \subset V$ be two admissible domains in $T^*M$ and $\phi: U \to V$ a Liouville embedding. We call $\phi$ {\it strongly unknotted} if there exists an isotopy $\{\phi_t\}_{t \in [0,1]}$ such that each $\phi_t: U \to V$ is a Liouville embedding and $\phi_0 = i_U$, $\phi_1 = \phi$, $i_U$ being the inclusion $i_U:U\rightarrow V.$ \end{dfn}

Let us illustrate these concepts on an example coming from Riemannian geometry. This example is also going to be the main example considered in this paper.

\begin{ex} \label{Main_Example}
Let $(M,g)$ be a closed, orientable Riemannian manifold with induced norm $\| \cdot \|_g: TM \to \R$ and denote the unit ball at a point $q$ by $B_1(g)_q = \{x \in T_q M \,| \, \|x\|_g \leq 1\}$. The dual norm $\| \cdot \|_{g^*}$ on $T^*M$ is given by $\| \xi_q \|_{g^*} =\max \{\xi_q(x) \, | \, x \in B_1(g)_q\}$ and the unit coball $B^*_1(g^*)_q= \{p \in T^*_q M \,| \, \|p\|_{g^*} \leq 1\}$ defines a convex set in $T^*_qM.$ Denoting the unit codisc bundle (union of unit coballs over all points of the manifold) by $U^*_g M,$ we have that $U^*_g M$ is an admissible domain in $T^*M.$ The boundary $\partial U^*_g M$ is the unit cosphere bundle and the Reeb flow on $(\partial U^*_g M, \lambda_{can}|_{\partial U^*_g M})$ is the cogeodesic flow of $g.$ Now, if for every $q \in M$ and every $x \in T_qM$, it holds $||x||_{g_1} \leq ||x||_{g_2}$, we have that $U^*_{g_1} M\subset U^*_{g_2} M$ and inclusion $i: U^*_{g_1} M \rightarrow  U^*_{g_2} M$ is a ${\tilde \pi}_1$-trivial Liouville embedding. Obviously, this embedding is also strongly unknotted.
\end{ex}

We will now define the distance which we wish to consider.

\begin{dfn} (Ostrover, Polterovich, Gutt, Usher \cite{P15, P17, OstPol, PRSZ17, GU17, Ush18})  \label{dfn-SBM} For $U, V \in \mathcal C_M$, we define {\it symplectic Banach-Mazur distance} $d_{SBM}(U,V)$ by
\[ d_{SBM}(U,V) = \inf\left\{ \ln C \,\bigg| \, \begin{array}{cc} \mbox{$\exists$\,\, $\frac{1}{C}U \xhookrightarrow{\phi} V \xhookrightarrow{\psi} CU$} \,\,(\mbox{and hence} \,\mbox{$\frac{1}{C}V \xhookrightarrow{\psi(C^{-1})} U \xhookrightarrow{\phi(C)} CV$}) \\\mbox{s.t. $\psi \circ \phi$ and $\phi(C) \circ \psi(C^{-1})$ are strongly unknotted} \end{array} \right\}. \]
Here multiplication $C U$ applies on the covector component, i.e. for any $(q, p) \in U$, $C (q,p) = (q, C p)$. Moreover, $\phi(C)$ is defined as $\phi(C)(q,p) = C\phi(q, p/C),$ for $(q,p) \in U$, where again multiplication acts on the covector component and $\psi(C^{-1})$ is defined similarly.
\end{dfn}

In order to study unit codisc bundles of different Riemannian metrics with respect to $d_{SBM}$, we will need an auxiliary distance defined on the space of Riemannian metrics on $M.$ Denote by
\[ \mathcal G_M = \{\mbox{Riemannian metrics $g$ on $M$}\} .\]
Similarly to $d_{SBM}$ on $\mathcal C_M$, we have

\begin{dfn} \label{RBM} For $g_1, g_2 \in \mathcal G_M$, we define {\it Riemannian Banach-Mazur distance} denoted by $d_{RBM}(g_1,g_2)$ as follows,
\[ d_{RBM}(g_1, g_2) = \inf\left\{ \ln C \in [0, \infty) \, \bigg| \, \exists \phi \in {\Diff}_0\,(M)\ \,\,\mbox{s.t.} \,\, \frac{1}{C} g_1 \preceq \phi^*g_2 \preceq C g_1\right\},\]
where $g_1 \preceq g_2$ means that for any $q \in M$ and any $x \in T_qM$, $||x||_{g_1} \leq ||x||_{g_2}$. ${\Diff}_0(M)$ stands for the identity component of ${\Diff}(M).$ \end{dfn}

As we saw in Example~\ref{Main_Example}, every Riemannian metric $g$ defines a domain $ U^*_g M \in \mathcal C_M$ and thus $\mathcal G_M$ can be naturally identified with a subset of $\mathcal C_M.$ With this in mind, $d_{RBM}$ and $d_{SBM}$ are two pseudo-metrics on $\mathcal G_M$ which turn out to be comparable. More precisely, the following inequality is proven in Proposition~\ref{SvsR}
\begin{equation}\label{SRB}
2 \cdot d_{SBM}(U^*_{g_1}M, U^*_{g_2}M) \leq d_{RBM}(g_1, g_2).
\end{equation}

Recall that given a contact manifold $(Y,\mu)$ with Reeb flow $\varphi^{\mu}_t,$ a periodic Reeb orbit $\varphi^{\mu}_t(x),$ $\varphi^{\mu}_T(x)=x$ of period $T$ is called {\it non-degenerate} if $\det (d\varphi^{\mu}_T|_{\ker \mu(x) } - {\mathds 1}_{\ker \mu(x) })\neq 0.$ If all periodic Reeb orbits are non-degenerate, contact manifold $(Y,\mu)$ is called {\it non-degenerate}. In the light of this definition an admissible domain $U$ is called {\it non-degenerate} if $(\partial U,\lambda_{can}|_{\partial U})$ is a non-degenerate contact manifold. A classical tool used to study admissible domains is symplectic homology, denoted by $\SH_{*}(U; \alpha)$, for $U\in \mathcal C_M$ and $\alpha$ a homotopy class of free loops in $M.$ Assuming $U$ is non-degenerate, a filtered version of symplectic homology $\SH^a_{*}(U; \alpha),a>0$ can be viewed as a persistence module which we denote by $\mathbb{SH}_{*,\alpha}(U)$, see Subsection \ref{symplectic-persitence} as well as Subsection \ref{app-pm} for general background on persistence modules. Multiplying the domain by $C>0$ results in the proportional scaling of the filtration, that is  $\SH^{Ca}_{*}(CU; \alpha)=\SH^{a}_{*}(U; \alpha)$ for all $a>0.$ In accordance with Definition \ref{dfn-SBM} we introduce the logarithmic version of $\SH_{*}(U; \alpha),$
$$S^t_{*}(U; \alpha)=\SH^{e^t}_{*}(U; \alpha),~~t\in \R,$$
which satisfies $S^{t+\ln C}_{*}(CU; \alpha)=S^{t}_{*}(U; \alpha).$ The resulting persistence module is denoted by $\mathbb S_{*, \alpha}(U).$ We are able to estimate $d_{RBM}$ and $d_{SBM}$ from below using the associated barcodes, namely the following {\it stability property} holds.

\begin{theorem} \label{TST} For $U,V \in \mathcal C_M$ non-degenerate, denote the barcodes of $\mathbb{S}_{*,\alpha}(U)$ and $\mathbb{S}_{*,\alpha}(V)$ by $\mathbb B_{*,\alpha}(U)$ and $\mathbb B_{*,\alpha}(V)$ respectively. Then
\[ d_{bottle}(\mathbb B_{*,\alpha}(U), \mathbb B_{*,\alpha}(V)) \leq d_{SBM}(U,V). \]
In particular, when $U = U^*_{g_1} M$ and $V = U^*_{g_2}M$, it follows from (\ref{SRB}) that
\[2 \cdot d_{bottle}(\mathbb B_{*,\alpha}(U), \mathbb B_{*,\alpha}(V)) \leq  d_{RBM}(g_1,g_2) .\]
\end{theorem}

Precise definitions and a proof of this theorem are given in Section 3. Different versions of Theorem \ref{TST} for star-shaped domains in $\R^{2n}$ can be found in \cite{PRSZ17} and \cite{Ush18}.

\medskip

We want to emphasize that in \cite{PS16} an analogous result in Hamiltonian Floer theory was proven (see also \cite{UZ16}). In the case of a symplectically aspherical manifold $(M,\omega),$ \cite{PS16} shows that 
$$d_{bottle}(\mathbb{HF}_{*,\alpha}(\phi), \mathbb{HF}_{*,\alpha}(\psi)) \leq d_{\rm Hofer}(\phi,\psi),$$ for $\phi, \psi \in \Ham(M, \omega).$ Here $\mathbb{HF}_{*,\alpha}$ denotes a persistence module coming from filtered Floer homology and $d_{\rm Hofer}$ stands for Hofer's metric. Nowadays it is well-known that this inequality can be used to prove continuity of some famous symplectic invariants, such as spectral invariants or boundary depth, with respect to Hofer's metric.

\begin{remark} ({\it Alternative definition}) One may give a definition of symplectic Banach-Mazur distance different from Definition \ref{dfn-SBM} as follows. 
\begin{dfn} \label{dfn-SBM-2} Let $U, V \in \mathcal C_M$ and
\[ \rho(U,V) = \inf\left\{ \ln C \in [0, \infty) \,\bigg| \, \begin{array}{cc} \mbox{$\exists$\,\, $\frac{1}{C}U \xhookrightarrow{\phi} V \xhookrightarrow{\psi} CU$}  \\\mbox{s.t. $\psi \circ \phi$ is strongly unknotted} \end{array} \right\}. \]
We define $d_{SBM}'(U,V) = \max\{\rho(U,V), \rho(V,U)\}$.\end{dfn}
One may prove that $d'_{SBM}$ defines a pseudo-metric on $\mathcal C_M$ in a similar way to the proof of Proposition \ref{dSBM}. However, in order to prove the stability of $d_{bottle}$ with respect to $d'_{SBM}$, i.e. an analogue of Theorem \ref{TST} , one needs a stronger version of the classical isometry theorem for barcodes which was communicated to us by M. Usher, \cite{Usher}. Quantities $\rho$ and $d'_{SBM}$ can be considered analogous to $\delta_f$ and $d_f$ defined in \cite{Ush18}, as explained in Subsection 1.2 of \cite{Ush18}.
\end{remark}

\begin{remark}
Throughout the paper, we assume that the base manifold $M$ is orientable. This is done in order to simplify considerations regarding the grading in symplectic homology, see Subsection \ref{symplectic-persitence}. It seems likely that, using the results from \cite{Web02}, one may apply similar arguments and obtain analogous results in the non-orientable case. 
\end{remark}

\subsection{Large-scale geometry of the space of Riemannian metrics}

Recall that a map $\Phi: (X_1,d_1) \rightarrow (X_2,d_2)$ between two (pseudo-)metric spaces is called \textit{quasi-isometric embedding} if there exist constants $A\geq 1,B \geq 0$ s.t.
\[\frac{1}{A}d_1(x,y)-B \leq d_2(\Phi(x),\Phi(y)) \leq A d_1(x,y)+B,\]
for all $x,y \in X_1.$

From a general perspective, given a (pseudo-)metric space $(X, d)$, we wish to ask the following questions with the flavor of large-scale geometry.
\begin{itemize}
\item[\textbf{(A)}] What is the diameter of $(X,d)$?
\item[\textbf{(B)}] If ${\rm diam}\,(X,d)=+\infty$, how many unbounded linearly independent directions are there in $X$? More precisely, for which $N$ does there exist a quasi-isometric embedding of $\R^N$ into $(X,d)?$
\end{itemize}

Our goal is to give partial answer to these questions for the space of admissible domains in $T^*M$, i.e. when $(X,d)=(\mathcal C_M,d_{SBM}) .$ In the case of Hofer's metric, i.e. $(X,d) = (\Ham(M, \omega), d_{\rm Hofer})$, these questions have been studied and partially answered using advanced tools from symplectic topology (see, for instance, \cite{Pol98} and \cite{Ush13}). 

Before we state the main results we wish to point out that without imposing additional assumptions on spaces $(\mathcal C_M,d_{SBM})$ and $(\mathcal G_M,d_{SBM})$ it is easy to see that both of their diameters are infinite. This follows from the fact that $d_{SBM}$ satisfies $d_{SBM}(U, CU) = \ln C$ for any $U \in \mathcal C_M$ and $C \geq 1$. Indeed, for any $C \geq 1$ it readily follows that $d_{SBM}(U, CU) \leq \ln C$ simply by taking $\phi$ and $\psi$ in the definition of $d_{SBM}$ to be inclusions. On the other hand, if there would exist some $C'<C$ such that $U/C' \hookrightarrow CU \hookrightarrow C'U$, the second embedding would contradict preservation of volume and hence $d_{SBM}(U,CU) = \ln C$. Thus, in order to make question $\textbf{(A)}$ meaningful we must introduce certain normalizations. To this end we define
\[ \bar{\mathcal C}_M = \left\{ \mbox{admissible domains $U$ in $T^*M$ s.t. $\Vol(U) = {V_n}$}  \right\},\]
where $\Vol(U)=\int\limits_{U} \frac{(d\lambda_{can})^n}{n!}=\int\limits_{U} \frac{(\omega_{can})^n}{n!}$ and $V_n$ denotes the volume of the $n$-dimensional unit ball. Similarly, we define
\[ \bar{\mathcal G}_M = \{\mbox{Riemannian metrics $g$ on $M$ s.t. $\Vol_g(M)=1$ and ${\rm diam}(M,g)\leq 100$}\}. \]
Note that when $U=U^*_g M$, one has $\Vol(U)=V_n \cdot \Vol_g(M)$ and hence we may include $\bar{\mathcal G}_M$ in $\bar{\mathcal C}_M$ via the map $g\rightarrow U^*_g M.$ Slightly abusing the notation we write $\bar{\mathcal G}_M \subset \bar{\mathcal C}_M.$

\begin{remark} \label{res-diam} We also wish to explain the restriction that we put on the diameter of $(M, g).$ Assume that $g_2 \preceq C^2 g_1,$ for a constant $C\geq 1.$ Now, for every smooth curve $\gamma$ in $M$ it holds that $L_{g_2}(\gamma) \leq C \cdot L_{g_1}(\gamma),$ $L_{g_i}$ denoting the length with respect to $g_i,$ and thus ${\rm diam}(M,g_2) \leq C \cdot {\rm diam}(M,g_1).$ If ${\rm diam}(M,g_1)$ is fixed and ${\rm diam}(M,g_2) \rightarrow + \infty,$ we see that $C \rightarrow + \infty$ and since ${\rm diam}(M,g_2)={\rm diam}(M,\phi^* g_2)$ for all diffeomorphisms $\phi,$ we have that $d_{RBM}(g_1,g_2)\rightarrow + \infty.$  This means that if there was no restriction on the diameter of $(M,g)$, the space $(\mathcal G_{M}, d_{RBM})$ would trivially have infinite diameter even if we fix the volume of $M$. \end{remark}

The following theorem is the main result of the paper.

\begin{theorem} \label{Main_Thm}
If $M=S^2$, then there exists a quasi-isometric embedding
\[ \Phi: ([0,\infty),|\cdot|) \to (\bar{\mathcal G}_M, d_{SBM}). \]
If $M=\Sigma$ is a closed, orientable surface whose genus is at least 1, then for every $N\in \N$ there exists a quasi-isometric embedding
\[ \Phi: (\R^N, |\cdot|_{\infty}) \to (\bar{\mathcal G}_{M}, d_{SBM}).\]
Both statements remain true if we replace $d_{SBM}$ by $d_{RBM}.$
\end{theorem}

Since  $\bar{\mathcal G}_M \subset \bar{\mathcal C}_M$ we immediately obtain the following.

\begin{cor} \label{Main_Cor}
If $M=S^2$, then there exists a quasi-isometric embedding
\[ \Phi: ([0,\infty),|\cdot|) \to (\bar{\mathcal C}_M, d_{SBM}). \]
If $M=\Sigma$ is a closed, orientable surface whose genus is at least 1 then for every $N\in \N$ there exists a quasi-isometric embedding
\[ \Phi: (\R^N, |\cdot|_{\infty}) \to (\bar{\mathcal C}_{M}, d_{SBM}).\]
\end{cor}

Corollary \ref{Main_Cor} readily implies that if $M$ is any closed, orientable surface it holds ${\rm diam}(\bar{\mathcal C}_{M}, d_{SBM})=+\infty,$ which answers question $\textbf{(A)}.$ However, regarding question $\textbf{(B)},$ we observe a sharp contrast between cases of a sphere and of higher genus surfaces. Indeed, when $M = \Sigma$ is a closed, orientable surface of positive genus, Corollary \ref{Main_Cor}  proves the existence of many unbounded directions inside $(\bar{\mathcal C}_{M}, d_{SBM})$, namely there exist $N$ unbounded directions inside $(\bar{\mathcal C}_{M}, d_{SBM})$ for any $N\in \N.$ On the other hand when $M=S^2$ it provides only one unbounded direction. This contrast ultimately comes from the fact that $\pi_1(S^2)=0$ while $\pi_1(\Sigma)\neq 0.$ Nevertheless, we pose the following conjecture.

\begin{conj} For every $N \in \N$, there exists a quasi-isometric embedding 
\[ \Phi: (\R^N, |\cdot|_{\infty}) \to (\bar{\mathcal C}_{S^2}, d_{SBM}).\]
\end{conj}

We break down the proof of Theorem \ref{Main_Thm} into the following two propositions.

\begin{prop} \label{emb-thm}
For any $\varepsilon>0$, there exists a map $\Phi: [0, \infty) \to \bar{\mathcal G}_{S^2}$ such that for any $x, y \in [0, \infty)$,
\begin{equation} \label{emb}
|x-y| - \varepsilon \leq 2d_{SBM}(U^*_{\Phi(x)}M,U^*_{\Phi(y)}M) \leq d_{RBM}(\Phi(x), \Phi(y)) \leq 2|x-y|+\varepsilon.
\end{equation}
\end{prop}

\begin{prop} \label{emb-thm2}
Let $\Sigma$ be a closed, orientable surface of genus at least 1. Then for any $N \in \N$ and any $\varepsilon >0$, there exists a map $\Phi: \R^N \to \bar{\mathcal G}_{\Sigma}$ such that for all $\vec{x}, \vec{y} \in \R^N$
\[ \frac{1}{4} \cdot |\vec{x}- \vec{y}|_{\infty} - \varepsilon \leq 2d_{SBM}(U^*_{\Phi(\vec{x})}M,U^*_{\Phi(\vec{y})}M) \leq d_{RBM}(\Phi(\vec{x}), \Phi(\vec{y})) \leq 4N \cdot |\vec{x} - \vec{y}|_{\infty} + \varepsilon .\]
\end{prop}

The lower bounds are the most significant parts of Propositions \ref{emb-thm} and \ref{emb-thm2}. Their proofs use the technique of barcodes and occupy the entire Section \ref{Proofs}. The upper bounds in both theorems are proven simultaneously in Subsection \ref{sec-upper}.

In order to construct quasi-isometric embeddings $\Phi$ as above, we consider geometric models which we call {\it bulked spheres} and {\it multi-bulked surfaces}. A bulked sphere is a surfaces of revolution. Roughly speaking, it is obtained as a connected sum of two spheres through a very narrow ``neck'' as shown in Figure \ref{ex-bsm}. We analyze closed geodesics on a bulked sphere in Section \ref{sec-bulk}. More precisely, in Subsection \ref{Geo_index} we analyze the shortest non-constant closed geodesic, coming from the connecting neck, and its iterates, while in Subsection \ref{Geo_length} we analyze the rest of the closed geodesics. By shrinking the neck we produce the desired direction going to infinity in terms of $d_{SBM}$. On the other hand, a {\it multi-bulked surface} is a closed, orientable surface of genus at least one which has a part that looks like a connected sum of $N+1$ spheres through $N$ ``narrow necks'', see Figure \ref{ex-mb}. By shrinking different necks, we obtain different unbounded directions. The behaviour of closed geodesics in a multi-bulked surface is also discussed in Subsections \ref{Geo_index} and \ref{Geo_length}. Finally, in order to exclude multiple covers of the same loop from our considerations, we work with symplectic homology in the non-trivial class of loops $\alpha.$ This explains the significance of the condition on the genus of $\Sigma,$ since every loop in $S^2$ is contractible.

\begin{remark} A theorem similar to Theorem \ref{Main_Thm} was proven by M. Usher in \cite{Ush18} in the context of star-shaped domains in $\C^n.$ Even though the general set-up is similar, the constructions of the quasi-isometric embeddings as well as the arguments used in the proofs in these two papers are fundamentally different. \end{remark}

\subsection{Applications to the study of closed geodesics}

As another application of our methods, we study existence and robustness (with respect to perturbations of metrics) of closed geodesics on a closed, orientable, Riemannian manifold $(M,g).$ Let $\alpha$ be a homotopy class of free loops in $M$ and denote by $\mathcal{L}_\alpha (M)$ the space of smooth loops $\gamma:S^1=\R / \Z \rightarrow M$ in class $\alpha.$ Define {\it the energy of a loop $\gamma$} by $E_g(\gamma)= \int_0^1 \frac{\| \dot \gamma(t) \|^2_g}{2} dt$. Critical points of $E_g$ are closed geodesics of constant speed and if all of them are non-degenerate, metric $g$ is called {\it bumpy.} The condition of $g$ being bumpy is equivalent to $U^*_g M$ being a non-degenerate admissible domain. Introduce the filtration on $\mathcal{L}_\alpha (M)$ by
$$\mathcal{L}^\lambda_\alpha (M)=\{ \gamma \in \mathcal L_{\alpha}(M)\, | ~  E_g(\gamma) \leq \lambda \}.$$
If $g$ is bumpy, homologies $H_*(\mathcal{L}^\lambda_\alpha (M); \Z_2)$ form a persistence module with parameter $\lambda\in \R,$ which we denote by $\H _{*, \alpha}(M,g)$ (here the homology of the empty set is taken to be zero). Comparison maps
$$\iota_{\lambda , \eta}:H_*(\mathcal{L}^\lambda_\alpha (M); \Z_2) \rightarrow H_*(\mathcal{L}^\eta_\alpha (M); \Z_2) $$
for $\lambda \leq \eta$ are given by the inclusions $\{ \gamma\, |\, E_g(\gamma) \leq \lambda \} \subset \{ \gamma \,|\, E_g(\gamma) \leq \eta \}.$ Now, Morse-Bott theory of $E_g$ implies that the endpoints of bars in the barcode $\mathbb{B}(\mathbb H_{*, \alpha}(M,g))$ are equal to energies of certain closed geodesics. This allows us to use persistence modules to study closed geodesics, namely, in Subsection \ref{Proof_non_symmetric}, we prove the following result.

\begin{theorem}\label{exist-geodesic}
Let $g_1,g_2$ be two bumpy metrics on a closed, orientable manifold $M$ such that $\frac{1}{C_1} g_1 \preceq g_2 \preceq C_2 g_1.$ If there exists a bar $[x,y) \in \mathbb{B}(\mathbb H_{*, \alpha}(M,g_1))$ such that $\frac{y}{x} > C_1 C_2$ then there exist closed geodesics $\gamma_1$ and $\gamma_2$ of $(M,g_2)$ in homotopy class $\alpha$, whose energies satisfy
$$\frac{1}{C_1}x \leq E_{g_2}(\gamma_1) \leq C_2 x, \,\,\,\, \frac{1}{C_1}y \leq E_{g_2}(\gamma_2) \leq C_2 y,$$
and furthermore $[E_{g_2}(\gamma_1), E_{g_2}(\gamma_2)) \in \mathbb{B}(\mathbb H_{*, \alpha}(M,g_2))$.

In the case of an infinite bar $[x,+\infty)\in \mathbb{B}(\mathbb H_{*, \alpha}(M,g_1))$, there exists a closed geodesic $\gamma_1$ of $g_2$ such that 
$$\frac{1}{C_1}x \leq E_{g_2}(\gamma_1)  \leq C_2 x,$$
and we have that $[E_{g_2}(\gamma_1), +\infty) \in \mathbb{B}(\mathbb H_{*, \alpha}(M,g_2)).$
\end{theorem}

In \cite{Ulj16}, for a fixed Finsler metric $F,$ the quantity $l(F)$ was introduced as the length of the shortest non-constant and ``homologically visible'' closed geodesic $\gamma_0$. Assume that $F$ comes from a Riemannian bumpy metric $g$ and denote by $L_g$ the length with respect to $g.$ Since $\gamma_0$ has constant speed we have $E_g(\gamma_0)=\frac{L_g(\gamma_0)^2}{2}=\frac{(l(F))^2}{2}.$ In the language of barcodes $\frac{(l(F))^2}{2}$ is equal to the smallest non-zero endpoint of an infinite bar in $\mathbb{B}(\mathbb H_{*, \alpha}(M,g)),$ where smallest means smallest among all such endpoints for all $\alpha.$ Now Theorem \ref{exist-geodesic} implies the following.

\begin{cor}[Theorem 1.10 in \cite{Ulj16} - Bumpy metric case]\label{existence-visible}
Let $g_1,g_2$ be two bumpy metrics on a closed, orientable manifold $M$ such that $g_2 \preceq g_1.$ Then $l(g_2)\leq l(g_1)$ and in particular there exists a non-constant ``homologically visible'' closed geodesic $\gamma$ of $g_2$ such that $L_{g_2}(\gamma) \leq l(g_1) .$
\end{cor}
\begin{proof}
Since $M$ is compact, there exists $C_1$ such that $\frac{1}{C_1} g_1 \preceq g_2 \preceq g_1.$ For some $\alpha$ we have that $$\left[ \frac{l(g_1)^2}{2},+\infty \right) \in \mathbb{B}(\mathbb H_{*, \alpha}(M,g_1)),$$
and thus Theorem \ref{exist-geodesic} implies that there exists a ``homologically visible'' closed geodesic $\gamma$ of $g_2$ such that
$$\frac{1}{\sqrt{C_1}} l(g_1) \leq L_{g_2}(\gamma) \leq l(g_1),$$
which finishes the proof.
\end{proof}

\begin{remark}
In view of Corollary \ref{existence-visible}, Theorem \ref{exist-geodesic} may be considered a quantitative version of Theorem 1.10 in \cite{Ulj16} (in the case of bumpy metrics on an orientable manifold). Indeed, it provides estimates for the energies (or equivalently lengths) of the closed geodesics in terms of constants $C_1$ and $C_2$ which are used to measure the discrepancy between $g_1$ and $g_2.$ Another benefit of our method is that it allows us to study ``homologically invisible'' closed geodesics, i.e. finite bars in $\mathbb{B}(\mathbb H_{*, \alpha}(M,g)).$   
\end{remark}

Let us illustrate the appearance of finite bars in $\mathbb{B}(\mathbb H_{*, \alpha}(M,g))$ on a concrete example of metrics of revolution on $\mathbb{T}^2.$

\begin{ex} \label{Torus_of_revolution}

Let $A>0, f:[-A,A]\rightarrow (0,+\infty)$ a smooth, even function, strictly increasing on $[-A,0]$ and hence strictly decreasing on $[0,A]$ with unique maximum at 0 and two minima at $\pm A.$ Moreover, assume that $f$ extends $2A$-periodically to a smooth function on $\R$ and let $g$ be a Riemannian metric on $\mathbb{T}^2= \R/ 2A \Z \times \R/ 2\pi\Z$ induced by an embedding into $\R^3$ given by
$$(x,\theta)\rightarrow (x,f(x)\cos \theta, f(x) \sin \theta).$$
Define a change of variable $X(x)=\int_{0}^{x}\sqrt{1+(f'(t))^2} dt, x\in [-A,A]$ and denote by $T=\int_{0}^{A}\sqrt{1+(f'(t))^2}$, so that $X\in[-T,T].$ In $X$ variable we denote $F(X)=f(x(X)).$

In Subsection \ref{geodesic-flow} we give a detailed analysis of the geodesic flow of $g.$ In particular, this analysis shows that parallel circles $X=\pm T$ and $X=0$ (i.e. $x=\pm A$ and $x=0$) are closed geodesics which we denote by $\check{\gamma}$ and $\hat{\gamma}$ respectively, see Figure \ref{torus_rev} below.

\begin{figure}[ht]
\includegraphics[scale=0.5]{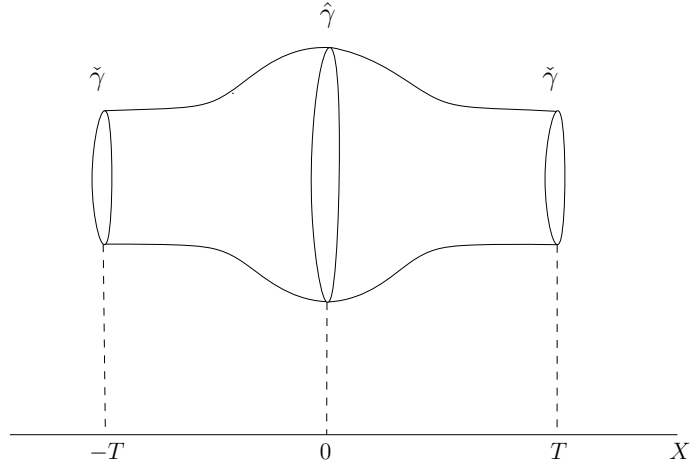}
\caption{Metric of revolution on $\mathbb{T}^2$}
\label{torus_rev}
\end{figure}

Let $\alpha$ be the homotopy class of loops $\check{\gamma}$ and $\hat{\gamma}$ oriented in the direction in which $\theta$-coordinate increases. The function $F$ which we wish to consider is defined\footnote{In principal, one should define $f$ in order to define $g,$ however, in this case one may show that $f$ is implicitly defined by $F$ as explained in Subsection \ref{geodesic-flow}.} on $[-1,1]$ as a $C^0$-small smoothing of $\frac{1}{\sqrt{kX^2+m}}$ for $k,m>0,~\sqrt{k}<m.$ Namely in Subsection \ref{geodesic-flow} we prove the following:
\begin{lemma} \label{No-geodesics}
Let $k,m>0,~\sqrt{k}<m.$ For small enough $\varepsilon>0$ there exists $F_\varepsilon:[-1,1]\rightarrow (0,+\infty)$ as above such that $F_\varepsilon(X)=\frac{1}{\sqrt{kX^2+m}}$ for $X\in [-1+\varepsilon, 1-\varepsilon]$ and there are no closed geodesics of metric $g$ induced by $F_\epsilon$ in class $\alpha$ other than $\check{\gamma}$ and $\hat{\gamma}.$ Moreover, $\check{\gamma}$ and $\hat{\gamma}$ are non-degenerate and $\ind \check{\gamma}=0,\ind \hat{\gamma}=1.$ Furthermore, $F_\varepsilon \xrightarrow{C^0} \frac{1}{\sqrt{kX^2+m}}$ as $\varepsilon \rightarrow 0.$
\end{lemma}

Using Lemma \ref{No-geodesics} we may compute the whole barcode $\mathbb{B}(\mathbb H_{*, \alpha}(\mathbb{T}^2,g)).$ To this end, first note that
\begin{equation}\label{homology_torus}
H_*(\mathcal{L}_\alpha (\mathbb{T}^2); \Z_2)=\left\{ \begin{array}{l}
{\Z_2,~ *=0,2} \\
\\
{\Z_2\oplus \Z_2,~ *=1 }\\
\\
{0,~ {\rm otherwise}}\\
\end{array} \right.
\end{equation}
Indeed, using the group action of $\mathbb{T}^2$ on $\mathcal{L}(\mathbb{T}^2)$ one proves that $\mathcal{L}(\mathbb{T}^2)\cong \mathbb{T}^2 \times \Omega \mathbb{T}^2$, $\Omega \mathbb{T}^2$ being the based loop space of $\mathbb{T}^2,$ and hence $\mathcal{L}_\alpha (\mathbb{T}^2) \simeq \mathbb{T}^2.$ For another computation of $H_*(\mathcal{L}_\alpha (\mathbb{T}^2); \Z_2)$, using symplectic homology, see Section 5 in \cite{BPS03} .

In order to compute $\mathbb{B}(\mathbb H_{*, \alpha}(\mathbb{T}^2,g))$ we use filtered (by $\lambda\in \R$) Morse-Bott chain complex\footnote{Here we omit the auxiliary height function $h$ on $S^1$ from the notation.} $CMB^\lambda_{*,\alpha}(E_g),$ described in detail in Subsection \ref{ss-Mor-B}. Since $\check{\gamma}$ is non-degenerate, it contributes two generators to $CMB^\lambda_{*,\alpha}(E_g)$, one in degree $\ind \check{\gamma}=0,$ the other in degree $\ind \check{\gamma}+1=1,$ both on filtration level $E_{min}=E_g(\check{\gamma})=\frac{2 \pi^2}{k+m}+o(\varepsilon).$ Similarly, $\hat{\gamma}$ contributes two generators to $CMB^\lambda_{*,\alpha}(E_g)$, one in degree 1, the other in degree 2, both on filtration level $E_{max}=E_g(\hat{\gamma})=\frac{2 \pi^2}{m}.$ The general rules for computing the barcode of $HMB^\lambda_{*,\alpha}(E_g)=H_*(\mathcal{L}^\lambda_\alpha (M); \Z_2)$ are the following (see \cite{UZ16} for a much more general version of this procedure):

Each generator corresponds to an endpoint of a bar equal to it's filtration level. Moreover, the generator in degree $d$ corresponds either to a left endpoint of a bar in  $\mathbb{B}(\mathbb H_{d, \alpha}(M,g))$ or to a right endpoint of a bar in $\mathbb{B}(\mathbb H_{d-1, \alpha}(M,g)).$ All infinite bars are of the form $[\cdot,+\infty)$ and they correspond to linearly independent homology classes in $H_*(\mathcal{L}_\alpha (M); \Z_2).$ Hence, the number of infinite bars is equal to the total dimension of $H_*(\mathcal{L}_\alpha (M); \Z_2).$

In our case, in particular, there are 4 infinite bars in $\mathbb{B}(\mathbb H_{*, \alpha}(\mathbb{T}^2,g))$ by (\ref{homology_torus}). Since there are only 4 generators of $CMB^\lambda_{*,\alpha}(E_g)$ each of them contributes an infinite bar or in other words we have
\begin{equation}
\mathbb{B}(\mathbb H_{*, \alpha}(\mathbb{T}^2,g))=\left\{ \begin{array}{l}
{\{[E_{min},+\infty) \},~ *=0} \\
\\
{\{[E_{min},+\infty),~ [E_{max},+\infty) \},~ *=1 }\\
\\
{ \{[E_{max},+\infty) \},~ *=2}\\
\end{array} \right.
\end{equation}

Let us now define for each $n\in \N$ a metric of revolution $g_n$, on $\mathbb{T}^2=\R / 2n \Z \times \R /2\pi \Z$ by stacking $n$ copies of a profile function $F$ next to each other, see Figure \ref{torus_rev_2} below.

\begin{figure}[ht]
\includegraphics[scale=0.5]{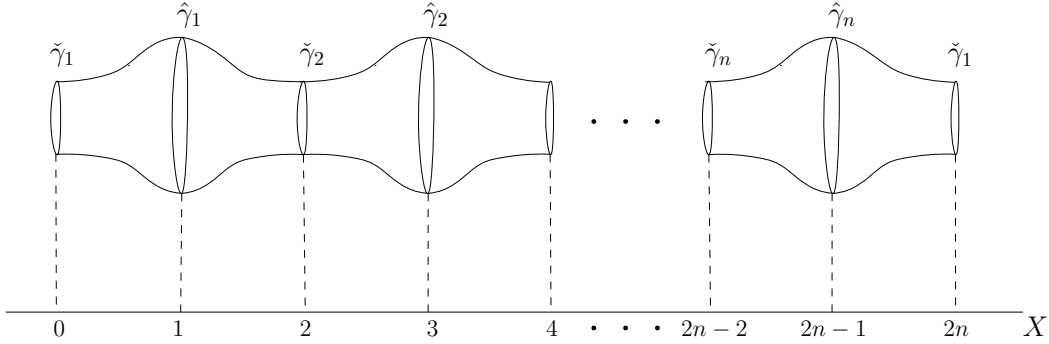}
\caption{Metric of revolution $g_n$ on $\mathbb{T}^2$}
\label{torus_rev_2}
\end{figure}

Same analysis as in the proof of Lemma \ref{No-geodesics} shows that $\check{\gamma}_i, \hat{\gamma}_i, i=1,\ldots n$ are the only closed geodesics of $g_n$ in class $\alpha$ as well as that they are non-degenerate and $\ind \check{\gamma}_i=0, \ind \hat{\gamma}_i=1$ for all $i.$ Similarly to the previous situation, we conclude that in filtered Morse-Bott chain complex $CMB^\lambda_{*,\alpha}(E_{g_n})$ on filtration level $E_{min}$ there are $n$ generators in degree 0 and $n$ generators in degree 1 coming from $\check{\gamma}_i.$ On the other hand, on filtration level $E_{max}$ there are $n$ generators in degree 1 and $n$ generators in degree 2 coming from $\hat{\gamma}_i.$ From (\ref{homology_torus}) we know that one of degree-0 generators contributes a bar $[E_{min},+\infty),$ one of degree-2 generators contributes a bar $[E_{max},+\infty),$ and there are two more infinite bars coming from degree-1 generators. Hence, the remaining $4n-4$ generators correspond to endpoints of finite bars, i.e. there are $2n-2$ finite bars in $\mathbb{B}(\mathbb H_{*, \alpha}(\mathbb{T}^2,g_n))$ , $n-1$ of them in $\mathbb{B}(\mathbb H_{0, \alpha}(\mathbb{T}^2,g_n))$ and $n-1$ of them in $\mathbb{B}(\mathbb H_{1, \alpha}(\mathbb{T}^2,g_n)).$ Since differential $\partial_{CMB}$ strictly lowers filtration, there are no degenerate bars, meaning bars of length zero, in $\mathbb{B}(\mathbb H_{*, \alpha}(\mathbb{T}^2,g_n)).$ Thus all finite bars are equal to $[E_{min},E_{max})$ and it readily follows that

\begin{equation}
\mathbb{B}(\mathbb H_{*, \alpha}(\mathbb{T}^2,g_n))=\left\{ \begin{array}{l}
{\{[E_{min},+\infty),~ [E_{min},E_{max})\times (n-1)  \},~ *=0} \\
\\
{\{[E_{min},+\infty),~ [E_{max},+\infty),~ [E_{min},E_{max})\times (n-1) \},~ *=1 }\\
\\
{ \{[E_{max},+\infty) \},~ *=2}\\
\end{array} \right.
\end{equation}

\end{ex}

\begin{remark}
One may define a ``non-symmetric'' version of $d_{RBM}$ in the following way. We say that $g_1,g_2$ are {\it $(C_1,C_2)$-equivalent} if there exists a diffeomorphism $\phi \in {\Diff}_0(M)$ such that $C_1 g_1 \preceq \phi^* g_2 \preceq C_2 g_1.$ Define {\it equivalence-distance} $d_{EQ}$ on $\mathcal{G}_M$ by
\[ d_{EQ} (g_1, g_2) = \inf \left\{ \ln  \frac{C_2}{C_1}  \,\bigg| \, \begin{array}{cc} \mbox{there exist $C_2 \geq C_1 >0$ such that} \\ \mbox{$(g_1, g_2)$ are $(C_1, C_2)$-equivalent} \end{array} \right\}. \]
One readily checks the following relations hold: $d_{EQ}(g_1,g_2)\geq 0,$ $d_{EQ}(g_1,g_2)=d_{EQ}(g_2,g_1)$ and $d_{EQ}(g_1,g_2)+d_{EQ}(g_2,g_3)\geq d_{EQ}(g_1,g_3).$ On the other hand, if for some $\phi \in {\Diff}_0(M)$ it holds $\phi^*g_2= C g_1$ then $d_{EQ}(g_1,g_2)=0.$ It follows from the definitions that
$$d_{EQ}(g_1,g_2) \leq 2d_{RBM}(g_1,g_2).$$
\end{remark}

Finally, we wish to mention that by taking $C_1=C_2$ in Theorem \ref{exist-geodesic}, we obtain Corollary \ref{exist} given below. This corollary also has a direct proof using stability of barcodes, i.e. Theorem \ref{TST}, which we present in Section \ref{quan-geo}.

\begin{cor} \label{exist}
Let $M$ be a closed, orientable manifold, $\alpha$ a homotopy class of free loops in $M$, $g_1$ a bumpy metric on $M$ and suppose that $[a^2/2, b^2/2) \in \mathbb B(\mathbb H_{*, \alpha} (M, g_1))$ for some $0<a<b$. For any bumpy metric $g_2$ on $M$, such that $0 \leq d_{RBM}(g_1, g_2) < \ln (b/a)$, there exist non-constant closed geodesics $\gamma_1, \gamma_2$ of $ g_2$ in homotopy class $\alpha$ such that
\[ \max\left\{ \left| \ln \left(\frac{E_{g_2}(\gamma_1)}{a^2/2}\right) \right|, \left| \ln \left(\frac{E_{g_2}(\gamma_2)}{b^2/2}\right) \right| \right\} \leq d_{RBM}(g_1,g_2). \]
In the case of an infinite bar $[a^2/2, +\infty) \in \mathbb B(\mathbb H_{*,\alpha}(M, g_1))$, there exists a ``homologically visible'' closed geodesic $\gamma_1$ of $g_2$ which satisfies $$\left| \ln \left(\frac{E_{g_2}(\gamma_1)}{a^2/2}\right) \right| \leq d_{RBM}(g_1,g_2).$$
\end{cor}

\subsection*{Acknowledgments} We are very grateful to Leonid Polterovich for his guidance and for helping us keep this project on the right track. We thank Michael Usher for sharing his preprint \cite{Ush18} with us as well as for helpful discussions. Parts of the project were motivated by a conversation with Kai Cieliebak at the conference -  Quantitative Symplectic Geometry, held in Simons Center for Geometry and Physics in Stony Brook University in May 2017. We are thankful to him for encouraging us to study this interesting topic. Finally, the project greatly benefited from conversations with  Daniel Rosen, Igor Uljarevic and Qingsan Zhu. Both authors are partially supported by the European Research Council Advanced grant 338809. The first author is also partially supported by the ISF grant No. 667/18.

\section{Basic properties of $d_{SBM}$ and $d_{RBM}$}

\subsection{Symplectic Banach-Mazur distance}
Symplectic Banach-Mazur distance $d_{SBM}$ is defined in \cite{PRSZ17} in the setting of general Liouville manifolds and it is proven to be a pseudo-metric. Our case of admissible domains in the cotangent bundle is a special case of this situation. For completeness, we include the proof of the following proposition

\begin{prop} \label{dSBM} $d_{SBM}$ defines a pseudo-metric on $\mathcal C_M$. \end{prop}

We start the proof with a lemma.

\begin{lemma} \label{Lemma_rescale}
Let $\frac{1}{C} U \xhookrightarrow{\phi}  V \xhookrightarrow{\psi} CU$, $U,V \in \mathcal{C}_M$, $C>1$ be ${\tilde \pi}_1$-trivial Liouville embeddings such that $\psi \circ \phi$ is strongly unknotted. Then for any $D>1$ and embeddings
$$\frac{1}{CD} U \xhookrightarrow{\phi(D^{-1})}  V \xhookrightarrow{\psi(D)} CD U,$$
$\psi(D) \circ \phi(D^{-1})$ is strongly unknotted
\end{lemma}

\begin{proof}
For $t\in [0,1]$ consider the following maps

\begin{equation}\label{composition}
 \frac{1}{C (1+(D-1)t)} U \xhookrightarrow{\phi((1+(D-1)t)^{-1})} V \xhookrightarrow{\psi(1+(D-1)t)} C(1+(D-1)t) U.
\end{equation}
Since $D-1>0$ we have that $C(1+(D-1)t) \geq C$ and hence
$$ \frac{1}{CD} U \subset \frac{1}{C(1+(D-1)t)} U,~~ C(1+(D-1)t) U \subset CD U.$$
Composing (\ref{composition}) with these inclusions, we get
$$\frac{1}{CD} U \xhookrightarrow{\psi(1+(D-1)t) \circ \phi((1+(D-1)t)^{-1})} CD U, $$
Denoting 
$$\beta_t= \psi(1+(D-1)t) \circ \phi((1+(D-1)t)^{-1}), ~ \beta_t : \frac{1}{CD} U \hookrightarrow CD U,$$
we get $\beta_0=\psi \circ \phi|_{\frac{1}{CD} U},$ $\beta_1=\psi (D) \circ \phi (D^{-1}).$ Since $\psi \circ \phi$ is strongly unknotted, there exists $\alpha_t:\frac{1}{C} U \hookrightarrow C U$ such that $\alpha_0=i_{\frac{1}{C} U},$ $\alpha_1= \psi \circ \phi.$ Restricting $\alpha_t$ to $\frac{1}{CD} U$ and composing with the inclusion $C U\xhookrightarrow{i} CD U,$ we get that $\alpha'_t=i\circ (\alpha_t |_{\frac{1}{CD}U})$ satisfies $\alpha'_0=i_{\frac{1}{CD}U},$ $\alpha'_1=\beta_0.$ Concatenation of $\alpha'$ and $\beta$ gives the desired isotopy.
\end{proof}

\begin{proof} (Proof of Proposition \ref{dSBM}) It readily follows that $d_{SBM}(U,U) = 0$ by taking $C=1$, $\phi = \psi = {\mathds1}_U$ in the definition of $d_{SBM}.$

To prove symmetry and triangle inequality we need the following two properties which can be proven by direct calculations. If $U \xhookrightarrow{\phi}  V \xhookrightarrow{\psi} W$ and $C,D>0$ then
\begin{itemize}
    \item[$(\ast)$] $(\phi \circ \psi)(C)= \phi (C) \circ \psi (C),$
    \item[$(\ast \ast)$] $(\phi(C))(D)=\phi (CD).$
\end{itemize}
Now, if $U/C \xhookrightarrow{\phi} V \xhookrightarrow{\psi} CU,$ $V/C \xhookrightarrow{\psi(C^{-1})} U \xhookrightarrow{\phi(C)} CV,$ are such that $\psi \circ \phi$ and $\phi(C) \circ \psi(C^{-1})$ are strongly unknotted, $(\ast)$ implies that so are 
$$
\frac{1}{C}V \xhookrightarrow{\psi(C^{-1})} U \xhookrightarrow{\phi(C)} CV,~\frac{1}{C}U \xhookrightarrow{\phi(C)(C^{-1})} V \xhookrightarrow{\psi(C^{-1})(C)} CU.
$$
This proves that $d_{SBM}(U,V)=d_{SBM}(V,U).$

Thus, we are left to prove the triangle inequality. Given $U/C \xhookrightarrow{\phi} V \xhookrightarrow{\psi} CU$ and $V/D \xhookrightarrow{\theta} W\xhookrightarrow{\xi} DV$ with strongly unknotted compositions, we claim that the composition of the following maps
$$ \frac{1}{CD} U \xhookrightarrow{\phi(D^{-1})} \frac{1}{D} V \xhookrightarrow{\theta} W\xhookrightarrow{\xi} DV \xhookrightarrow{\psi(D)} CD U $$
is also strongly unknotted. Indeed, denote by $\alpha_t:\frac{1}{D} V \hookrightarrow DV$ the isotopy such that $\alpha_0=i_{\frac{1}{D}V},$ $\alpha_1= \xi \circ \theta$, given by the unknottedness of $\xi \circ \theta$ and by $\beta_t: \frac{1}{CD} U \hookrightarrow CD U$ the isotopy such that $\beta_0=i_{\frac{1}{CD}U},$ $\beta_1= \psi(D) \circ \phi(D^{-1})$, given by the unknottedness of $\psi \circ \phi$ and Lemma \ref{Lemma_rescale}. Now, the isotopy $\gamma_t:\frac{1}{CD}U \hookrightarrow CD U$ given by

\[ \gamma_t = \left\{ \begin{array}{cc} \beta_{2t} \,\,\,&\mbox{for} \,\, 0 \leq t \leq 1/2 \\ \psi(D) \circ \alpha_{2t-1}  \circ \mathfrak \phi(D^{-1})  \,\,\, \,\,\,&\mbox{for}\,\, 1/2 \leq t \leq 1 \end{array} \right. \]
satisfies $\gamma_0=i_{\frac{1}{CD}U}$, $\gamma_1= \psi (D) \circ \xi \circ \theta \circ \psi (D^{-1})$ which proves the claim.

This way, we proved that the composition of maps
$$\frac{1}{CD} U \xhookrightarrow{\theta \circ \phi(D^{-1})} W \xhookrightarrow{\psi(D) \circ \xi} CD U $$
is strongly unknotted. What is left is to prove that the composition of maps
\begin{equation} \label{reversed}
\frac{1}{CD} W \xhookrightarrow{(\psi(D) \circ \xi)((CD)^{-1})} U \xhookrightarrow{(\theta \circ \phi(D^{-1}))(CD)} CD W
\end{equation}
is strongly unknotted. Using $(\ast)$ and $(\ast \ast)$, we reformulate (\ref{reversed}) as
$$ \frac{1}{CD} W \xhookrightarrow{(\xi(D^{-1}))(C^{-1})} \frac{1}{C} V \xhookrightarrow{\psi(C^{-1})} U\xhookrightarrow{\phi(C)} CV \xhookrightarrow{(\theta(D))(C)} CD W .$$
Now the same construction of the isotopy as the one we used for $\gamma$ applies, only starting from $ W/D \xhookrightarrow{\xi(D^{-1})} V \xhookrightarrow{\theta(D)} D W$ and $W/C \xhookrightarrow{\psi(C^{-1})} V \xhookrightarrow{\phi(C)} C W.$
\end{proof}

\begin{remark} \label{Pseudo-Metric-1}
Note that $d_{SBM}$ is not a genuine metric. Indeed if $U,V \in \mathcal C_M$ are exactly symplectomorphic via a $\tilde{\pi}_1$-trivial map, i.e. there exists a diffeomorphism $\phi: U \to V$ such that 1-form $\phi^* \lambda_{can} - \lambda_{can}$ is exact and $\phi_*= \mathds{1}_{\tilde{\pi}_1(M)}$, we have
\[ U \xhookrightarrow{\phi} V \xhookrightarrow{\phi^{-1}} U \,\,\,\,\mbox{and} \,\,\,\, V \xhookrightarrow{\phi^{-1}} U \xhookrightarrow{\phi} V \]
and thus $d_{SBM}(U, V) = 0$. This is, for example, the case when $V=\phi(U)$ and $\phi\in \Ham_c(T^*M).$ 
\end{remark}

\begin{remark} \label{Scaling-SBM}
Using $(\ast)$ and $(\ast \ast)$ one easily checks that for all $C>0$ and all $U,V\in \mathcal C_M$ it holds that $d_{SBM}(CU,CV)=d_{SBM}(U,V)$ as well as $d_{SBM}(U,CU)=\ln C.$
\end{remark}

\subsection{Riemannian Banach-Mazur distance}
We begin with the following statement

\begin{prop} \label{dRBM} $d_{RBM}$ defines a pseudo-metric on $\mathcal G_M$. \end{prop}

\begin{proof}
By taking $C=1$ and $\phi= {\mathds 1}_M$ one readily concludes that $d_{RBM}(g,g)=0.$ On the other hand, for $\phi \in {\Diff}_0\,(M)$ and $C \geq 1$, we have that $(1/C) g_1 \preceq\phi^* g_2 \preceq C g_1$  if and only if
\[ \frac{1}{C} g_2 \preceq(\phi^{-1})^* g_1 \preceq C g_2.\]
This implies $d_{RBM}(g_1, g_2) = d_{RBM}(g_2, g_1)$.

Finally, for $\phi_1, \phi_2 \in {\Diff}_0\,(M)$ and $C, D \geq 1$, the relations $(1/C) g_1 \preceq \phi^*_1 g_2 \preceq C g_1$ and $(1/D) g_2 \preceq \phi^*_2 g_3 \preceq D g_2$ imply
\[ \frac{1}{CD} g_1 \preceq \frac{1}{D} (\phi_1^* g_2) \preceq \phi_1^* (\phi_2^* g_3) \preceq D (\phi_1^* g_2) \preceq C D g_1. \]
Setting $\phi=\phi_2 \circ \phi_1$ gives $d_{RBM}(g_1, g_3)\leq \ln C + \ln D $ and thus taking infimum over $C$ and $D$ gives $d_{RBM}(g_1, g_3) \leq d_{RBM}(g_1, g_2) + d_{RBM}(g_2, g_3)$.
\end{proof}

\begin{remark} \label{Pseudo-Metric_2}
Similarly to $d_{SBM}$, $d_{RBM}$ is also not a genuine metric. Indeed, if there exists some $\phi \in {\Diff}_0\,(M)$ such that $g_1 = \phi^*g_2$, taking $C=1$ we have
\[  g_1 \preceq \phi^* g_2 = g_1 \preceq g_1. \]
This implies $d_{RBM}(g_1, g_2) =  0$. 
\end{remark}

\begin{remark} \label{Scaling-RBM}
One readily checks that it holds $d_{RBM}(Cg_1,Cg_2)=d_{RBM}(g_1,g_2)$ as well as $d_{RBM}(g_1,Cg_1)=\ln C$ for all $C>0$ and all $g_1,g_2\in \mathcal G_M$.
\end{remark}

As we saw in Example~\ref{Main_Example}, $\mathcal G_{M}$ can be identified with a subset of $\mathcal C_M$ via inclusion $g (\in \mathcal G_M) \mapsto U^*_g M ( \in \mathcal C_M)$. Therefore, it makes sense to compare $d_{SBM}$ with $d_{RBM}$ on $\mathcal G_M$ and we have

\begin{prop} \label{SvsR} Let $M$ be a closed, orientable manifold and $g_1, g_2 \in \mathcal G_M$. Then
\[ 2 \cdot d_{SBM}(U^*_{g_1}M, U^*_{g_2}M) \leq d_{RBM}(g_1, g_2).\]
 \end{prop}

\begin{proof} For any $\varepsilon>0$, there exists some $C^2>1$ and $\phi \in {\Diff}_0\,(M)$ such that $\ln (C^2) \leq d_{RBM}(g_1,g_2) + \varepsilon$ and $(1/C^2) g_1 \preceq \phi^* g_2 \preceq C^2 g_1$. Since $B_1(g_2) \subset B_1(g_1)$ if $g_1 \preceq g_2$ we have that
\begin{equation} \label{u-interleaving}
U^*_{\frac{1}{C^2} g_1} M \subset U^*_{\phi^* g_2} M \subset U^*_{C^2 g_1} M.
\end{equation}
One also readily checks that $U^*_{\frac{1}{C^2} g_1} M = \frac{1}{C} U^*_{g_1} M$ and $U^*_{C^2 g_1} M = C U_{g_1}^* M$.

On the other hand, $\phi \in {\Diff}_0\,(M)$ lifts to a symplectomorphism $\phi^{\#}$ of $T^*M$, given by $\phi^{\#}(p, \xi) = (\phi(p), (\phi_{\phi(p)}^{-1})^* \xi).$ Since $\phi$ is isotopic to $\mathds{1}_M$, the lift $\phi^{\#}$ is isotopic to $\mathds{1}_{T^*M}$ and in particular, $\phi^{\#}$ acts trivially on ${\tilde \pi}_1(M)$. Moreover, one may check that $\phi^{\#}$ is exact as well as that $\phi^{\#}(U^*_{\phi^* g}M)=U^*_gM.$ Therefore (\ref{u-interleaving}) can be rewritten as
\[ \frac{1}{C} U^*_{g_1} M \subset (\phi^{-1})^{\#}(U_{g_2}^* M) \subset C U_{g_1}^*M \]
which implies
\[ \frac{1}{C} U^*_{g_1}M \xhookrightarrow{(\phi^{\#})|_{\frac{1}{C} U^*_{g_1} M}} U_{g_2}^* M \xhookrightarrow{(\phi^{-1})^{\#}|_{U_{g_2}^* M}} C U_{g_1}^* M .\]
The above maps are ${\tilde \pi}_1$-trivial Liouville embeddings and their composition is the inclusion, thus strongly unknotted. We are left to check that the composition
\[ \frac{1}{C} U^*_{g_2}M \xhookrightarrow{(\phi^{-1})^{\#}|_{U_{g_2}^* M} (C^{-1})}  U_{g_1}^* M \xhookrightarrow{(\phi^{\#})|_{\frac{1}{C} U^*_{g_1} M} (C)}  C U_{g_2}^* M .\]
is strongly unknotted. This follows from the fact that $\phi^{\#}(C)=\phi^{\#}$ for all $\phi \in \Diff_0 (M)$ and all $C \geq 1.$

Therefore, by the definition of $d_{SBM}$, we obtain
\[ d_{SBM}(U_{g_1}^*M, U_{g_2}^*M) \leq \ln C \leq \frac{1}{2} d_{RBM}(g_1, g_2) + \frac{\varepsilon}{2}. \]
Since the inequality holds for every $\varepsilon > 0$, the conclusion follows.
\end{proof}

\section{Symplectic homology as a persistence module}\label{section-SympHom}

\subsection{Background on symplectic homology}\label{symplectic-persitence}
Symplectic homology has been developed in the 90's by the work of many people, see \cite{FH94, FHW94, CFH95, CFHW96, Vit99}. There exist different versions of the theory, depending on the class of manifolds and admissible Hamiltonians which are considered. We will use the version developed in \cite{BPS03} and \cite{Web06} (with different signs from \cite{BPS03}).

\medskip

Throughout the paper, all Floer homologies as well as symplectic homology are taken with $\Z_2$-coefficients. As a result all persistence modules will also be persistence modules over $\Z_2.$  
\medskip

We start by briefly recalling the setup of \cite{BPS03} and \cite{Web06}. For a fixed homotopy class $\alpha$ of free loops in $M$, consider the following space
\[ \mathcal L_{\alpha}(T^*M) = \left\{z:S^1\rightarrow T^*M \, | \,\, z= (x,y), \,\, x: S^1 \to M \,\,\mbox{s.t.} \,\,[x] = \alpha \,\,\mbox{and} \,\,y(t) \in T_{x(t)}^*M\right\}.\]
Recall that $(T^*M, \omega_{can}=d\lambda_{can})$ is a symplectic manifold and given a Hamiltonian function $H:\R/\Z \times T^*M \to \R$ we may define its Hamiltonian vector field $X_H$ by $i_{X_H}\omega_{can}=-dH.$ The collection of all Hamiltonian 1-periodic orbits of $H$ in $\mathcal L_{\alpha}(T^*M)$ is denoted by $\mathcal P(H; \alpha)$. Recall also that the symplectic action functional $\mathcal A_H$ is given by
$$\mathcal A_H(z)=\int_0^1 H_t(z(t)) dt - \int_{S^1} z^*\lambda_{can},$$
for any loop $z:S^1=\R/\Z \rightarrow T^*M.$ The action spectrum of $H$ in class $\alpha$ is 
\[ {\rm Spec}(H; \alpha) = \{\mathcal A_H(z) \, | \, z \in \mathcal P(H; \alpha)\}. \]

Since $T^*M$ is not compact, in order to define Floer homology, we need to impose certain restrictions on the Hamiltonian $H.$ The standard assumption in this situation is that $H$ is linear\footnote{Linearity of $H$ in this context can be understood as $H_t(x,y)=\beta \| y \|_{g^*}+\beta'$ for some $\beta,\beta'\in \R$ and a fixed Riemannian metric $g$ on $M.$ More generally, if $(U,\lambda)$ is a Liouville domain, linearity is understood as linearity with respect to the radial coordinate in the completion of $(U,\lambda)$ and the previously described linearity corresponds to the case $(U,\lambda)=(U^*_gM,\lambda_{can}).$} outside a compact subset of $T^*M.$ For simplicity, in this section we only consider compactly supported Hamiltonians, i.e. linear outside of a compact set with slope equal to zero, as this class of Hamiltonians suffices to define symplectic homology. However, in the proof of Theorem \ref{sh-loop-thm} we will need to work with non-zero slopes and we will review the relevant setup in Subsection \ref{isomorphism}.

Now, to a given compactly supported Hamiltonian $H:\R/\Z \times T^*M \to \R$, a class $\alpha$ of free loops in $M$ and real numbers $a<b$ not belonging to ${\rm Spec}(H; \alpha)$ (and also $0\notin [a,b]$ if $\alpha=[pt]$) we associate Hamiltonian Floer homology in action window $[a,b)$ denoted by $\HF^{[a,b)}_*(H, \alpha).$ This is homology of a Floer chain complex generated by elements of $\mathcal P(H; \alpha)$ with action in $[a,b)$ (see \cite{BPS03,Web06} for details). 

Grading on Floer chain complex is defined using the Lagrangian distribution of vertical subspaces $T^vT^*M\subset TT^*M$ given by $T_x^vT^*M=\ker d\pi(x)$ where $\pi:T^*M\rightarrow M$ is projection. Namely, let $z\in \mathcal P(H; \alpha)$ and let $\Phi:S^1\times \R^{2n}\rightarrow z^*(TT^*M)$ be a symplectic trivialization such that $\Phi_t(0\times \R^n)=T_{z(t)}^vT^*M$ for all $t\in S^1.$ The existence of such a trivialization follows from orientability of $z^*(T^vT^*M)$ (which follows from orientability of $M$), see, for example, Lemma 1.2 in \cite{AS06}. Now, if $\phi^H_t$ is the Hamiltonian flow of $H$, $\phi^H_t(z(0))=z(t),$ we have a path of symplectic matrices
$$P(t)=\Phi^{-1}_t \circ d\phi^H_t(z(0)) \circ \Phi_0,~t\in[0,1],$$
and we define $\ind_{\HF}(z)=\ind_{\rm CZ}(P),$ where $\ind_{\rm CZ}$ stands for the Conley-Zehnder index. It is easy to check that $\ind_{\HF}(z)$ does not depend on the choice of $\Phi$ as above, see Lemma 1.3 in \cite{AS06}. Moreover, our conventions for the Conley-Zehnder index are chosen in such a way that isomorphism in Theorem \ref{sh-loop-thm} preserves grading, see \cite{Web02, AS06} and references therein.

\begin{remark}
The definition of Floer homology also includes an auxiliary choice of an almost complex structure. Since $\HF^{[a,b)}_*(H, \alpha)$ does not depend on this choice, we omit it from the notation. We should also mention that one first defines $\HF^{[a,b)}_*(H, \alpha)$ for Hamiltonians $H$ whose periodic orbits are non-degenerate. Floer homology of a general Hamiltonian $H$ as above is then defined, roughly speaking, as $\HF^{[a,b)}_*(H', \alpha)$, where $H'$ is a $C^\infty$-small perturbation of $H$ whose periodic orbits are non-degenerate.
\end{remark}

We now focus on defining symplectic homology of an admissible domain $U \in \mathcal C_M.$ Denote by $\mathcal H_{U}$ the set of all functions on $S^1 \times T^*M$ compactly supported in $S^1 \times {\rm Int}(U)$. Given $a<b$ (with $a,b$ possibly being $\pm \infty$), let 
\[ \mathcal H_{U, a,b} = \{H \in \mathcal H_{U} \,| \, a, b \notin {\rm Spec}(H; \alpha) ~{\rm and} ~ 0 \notin [a,b] ~ {\rm if}~ \alpha=[pt] \} \]
If $b = +\infty$, we denote $\mathcal H_{U, a, +\infty}$ by $\mathcal H_{U,a}$.

Define a partial order on $\mathcal H_{U, a, b}$ as follow. For $H_1, H_2 \in \mathcal H_{U, a,b}$, $H_1 \preceq H_2$  if and only if $H_1(t,z) \geq H_2(t, z)$  for all  $(t,z) \in S^1 \times U.$ If $H_1 \preceq H_2$ then there exists a smooth homotopy $\tau \to H_{\tau}$ from $H_1$ to $H_2$ such that $\partial_\tau H_\tau \leq 0$. We call such a homotopy {\it monotone}. Every monotone homotopy induces a $\Z_2$-linear continuation map 
$$\sigma_{12}: \HF_*^{[a,b)}(H_1, \alpha) \to \HF_*^{[a,b)}(H_2, \alpha).$$
Moreover, the map $\sigma_{12}$ does not depend on the choice of the monotone homotopy. In general, $\sigma_{12}$ may not be an isomorphism. However if there exists a monotone homotopy $\tau \to H_{\tau}$ such that $H_{\tau}\in \mathcal H_{U, a,b} $ for every $\tau$, then $\sigma_{12}$ is an isomorphism, see Proposition 4.5.1 in \cite{BPS03}. Such a monotone homotopy is called {\it action-regular}. For a detailed treatment of maps induced by monotone homotopies, see \cite{BPS03,Web06} and references therein. We will need the following statement (see, for example, Lemma 2.7 in \cite{Web06}).

\begin{lemma} \label{4-lemma-1}
Let $U \in \mathcal C_M$. For any three functions $H_1, H_2$ and $H_3$ in $\mathcal H_{U, a, b}$ with $H_1 \preceq H_2 \preceq H_3$, the induced maps on Hamiltonian Floer homologies in action window $[a,b)$ satisfy $\sigma_{13} = \sigma_{23} \circ \sigma_{12}$. \end{lemma}

Note that Lemma \ref{4-lemma-1} together with the fact that continuation map is independent of the choice of the monotone homotopy implies that if $H_1 \preceq H_2 \preceq H_3$ and $H_1 \preceq H_4 \preceq H_3,$ $H_i\in \mathcal H_{U,a, b}$, the following diagram commutes 
\begin{equation} \label{com-1}
\xymatrix{
\HF_*^{[a,b)}(H_2, \alpha) \ar[r]^-{\sigma_{23}} &\HF_*^{[a,b)}(H_3, \alpha)\\
\HF_*^{[a,b)}(H_1, \alpha) \ar[r]_-{\sigma_{14}} \ar[u]^-{\sigma_{12}} &\HF_*^{[a,b)}(H_4, \alpha). \ar[u]_-{\sigma_{43}}}
\end{equation} 
We will use this in the proof of Theorem \ref{sh-loop-thm}.

Now notice that $(\mathcal H_{U, a,b}, \preceq)$ is a {\it downward directed partially ordered system}, i.e. for any $H_2, H_3 \in \mathcal H_{U, a,b}$, there exists some $H_1 \in \mathcal H_{U, a,b}$ such that $H_1 \preceq H_2$ and $H_1 \preceq H_3$. Lemma \ref{4-lemma-1} implies that Floer homologies $\HF_*^{[a,b)}(H,\alpha),$ $H \in \mathcal H_{U, a,b}$ together with continuation maps $\sigma_{12}:\HF_*^{[a,b)}(H_1,\alpha) \rightarrow \HF_*^{[a,b)}(H_2,\alpha)$ for $H_1 \preceq H_2,$  define an inverse system of $\Z_2$-vector spaces over $(\mathcal H_{U, a, b}, \preceq).$ Thus we can take the inverse limit of such an inverse system, which leads to the definition of symplectic homology. For a general background on inverse system and inverse limit, see subsection 4.6 in \cite{BPS03}. 

As we mentioned in the introduction, if $U\in \mathcal{C}_M$ then $(\partial U,\lambda_{can}|_{\partial U})$ is a contact manifold. We denote by $ {\rm Spec}(\partial U)$ the set of periods of all periodic Reeb orbits of $(\partial U,\lambda_{can}|_{\partial U}).$ Recall that $U$ is called non-degenerate if all the periodic Reeb orbits of $(\partial U,\lambda_{can}|_{\partial U})$ are non-degenerate. In this case for every $T>0$ there are finitely many Reeb orbits of period less than $T$ and in particular $ {\rm Spec}(\partial U)$ is discrete.

\begin{dfn}\label{dfn-symp-pm} For a homotopy class $\alpha$ of free loops in $M$, $U \in \mathcal C_M$ non-degenerate and $a>0, a\notin {\rm Spec}(\partial U)$ define the {\it filtered symplectic homology of $U$} by 
\[  \SH^a_*(U; \alpha)  =  \varprojlim_{H \in \mathcal H_{U, a}} \HF^{[a, \infty)}_*(H, \alpha).\] 
{\it The symplectic persistence module of $U$ in class $\alpha$} is given by the collection of data 
\[ \mathbb{SH}_{*, \alpha}(U)  = \left\{\{\SH^a_*(U; \alpha)\}_{a \in \R_{>0}}; \, \{\iota_{a,b}: \SH^a_*(U; \alpha) \to \SH^b_*(U; \alpha) \}_{a \leq b}\} \right\} \]
with $\iota_{a,b}$ being induced from corresponding filtered Hamiltonian Floer homologies.
\end{dfn}

There are two points which we wish to clarify regarding Definition \ref{dfn-symp-pm}. Firstly, if $0<a\leq b$ and $H\in \mathcal H_{U, a} \cap \mathcal H_{U, b}$ the map $\iota^H_{a,b}:\HF^{[a,+\infty)}_*(H, \alpha)\rightarrow \HF^{[b,+\infty)}_*(H, \alpha)$ is induced by the inclusion of Floer chain complexes\footnote{Here we think of a Floer chain complex ${\rm CF}^{[a,+\infty)}_*(H, \alpha)$ as a quotient ${\rm CF}^{[a,+\infty)}_*(H, \alpha)={\rm CF}^{(-\infty,+\infty)}_*(H, \alpha)/ {\rm CF}^{(-\infty,a)}_*(H, \alpha).$}. This map commutes with continuation maps and hence induces a map $\iota_{a,b}: \SH^a_*(U; \alpha) \to \SH^b_*(U; \alpha)$ for $a,b\notin {\rm Spec}(\partial U).$ 

Secondly, we only defined $\SH^a_*(U; \alpha)$ for $a\notin {\rm Spec}(\partial U).$ However, due to the non-degeneracy of $U$ we may extend the definition of $\SH^a_*(U; \alpha)$ to all $a>0$ by asking for all the bars in the barcode of $\mathbb{SH}_{*, \alpha}(U)$ to have left endpoints closed and right endpoints open. Indeed, one may show that for each $b>0, b\notin {\rm Spec}(\partial U),$ $\SH^b_*(U; \alpha)$ is a finite dimensional $\Z_2$-vector space as well as that if $[b,c]\cap {\rm Spec}(\partial U)=\emptyset$ then $\iota_{b,c}$ is an isomorphism. Since ${\rm Spec}(\partial U)$ is discrete, for every $a\in {\rm Spec}(\partial U)$ there exists $\varepsilon >0$ such that $[a-\varepsilon,a+\varepsilon]$ contains no other points from ${\rm Spec}(\partial U).$ Now, we define $\SH^a_*(U; \alpha)$ by asking for $\iota_{a,a+\varepsilon}$ to be an isomorphism, or, more formally, by setting $\SH^a_*(U; \alpha)=\varprojlim\limits_{t\in (a,a+\varepsilon]} (\SH^t_*(U; \alpha),\iota).$ An interested reader may check that all the bars in the barcode of a symplectic persistence module defined this way have endpoints in ${\rm Spec}(\partial U)$, each point in ${\rm Spec}(\partial U)$ is an endpoint of finitely many bars and all bars have left endpoints closed and right endpoints open.

We are mainly interested in the functorial properties of filtered symplectic homology, which are expressed by the following proposition.

\begin{prop} \label{fp-sh} Let $U, V, W \in \mathcal C_M$ be non-degenerate.
\begin{itemize}
\item[(1)] If $U \xhookrightarrow{\phi} V$ (recall that this means there exists a ${\tilde \pi}_1$-trivial Liouville embedding from $U$ to $V$), then there exists a persistence morphism $\h_{\phi}: \mathbb{SH}_{*, \alpha}(V) \to \mathbb{SH}_{*, \alpha}(U)$. Moreover, if $U \xhookrightarrow{\phi} V \xhookrightarrow{\psi} W$, the following diagram commutes
\[ \xymatrix{
\mathbb{SH}_{*, \alpha}(W) \ar[rr]^{\mathfrak{h}_{\psi}} \ar@/_1.8pc/[rrrr]^{\h_{\psi \circ \phi}} && \mathbb{SH}_{*, \alpha}(V) \ar[rr]^{\mathfrak{h}_{\phi}} && \mathbb{SH}_{*,\alpha}(U)}. \]
\item[(2)] For $C,a>0$, there exists a canonical persistence isomorphism $r_C: \SH_{*}^a(U; \alpha) \xrightarrow{\simeq} \SH_{*}^{Ca}(CU; \alpha).$ These ismorphims satisfy $(r_C)^{-1}=r_{\frac{1}{C}}$ for all $C>0.$ Moreover for $C \geq 1$, we have a commutative diagram
\begin{equation*}
\xymatrix{
\SH_{*}^{Ca}(C U; \alpha)\ar[rd]_{\mathfrak{h}^{Ca}_{i}} && \SH_{*}^{a}(U; \alpha) \ar[ld]^{\iota^{{\mathbb {SH}_{*, \alpha}(U)}}_{a, Ca}}  \ar[ll]^{\simeq}_{r_C}  \\
& \SH_{*}^{Ca}(U; \alpha)}
\end{equation*}
where $\mathfrak h^{Ca}_{i}$ is the persistence morphism induced by inclusion $U \xhookrightarrow{i} CU$ and $\iota^{{\mathbb {SH}_{*, \alpha}(U)}}_{a, Ca}$ is the comparison map of the persistence module $\mathbb {SH}_{*, \alpha}(U)$.

Similarly, for $C\leq 1$, we have the commutative diagram 
\begin{equation*}
\xymatrix{
\SH_{*}^{Ca}(C U; \alpha) && \SH_{*}^{a}(U; \alpha) \ar[ll]^{\simeq}_{r_C}  \\
& \SH_{*}^{Ca}(U; \alpha) \ar[lu]^{\mathfrak{h}^{Ca}_{i}} \ar[ru]_{\iota^{{\mathbb {SH}_{*, \alpha}(U)}}_{Ca, a}}}
\end{equation*}
where $\mathfrak h^{Ca}_{i}$ is the persistence morphism induced by inclusion $CU \xhookrightarrow{i} U$ and $\iota^{{\mathbb {SH}_{*, \alpha}(U)}}_{Ca, a}$ is the comparison map of the persistence module $\mathbb {SH}_{*, \alpha}(U).$
\item[(3)] If $\phi: U \rightarrow V$ is a ${\tilde \pi}_1$-trivial Liouville embedding, then for any positive $C$ and $a$, it holds $\mathfrak h_{\phi(C)}^{Ca} \circ r_C = r_C \circ \mathfrak h_{\phi}^a.$ In other words, the following diagram commutes:
\[ \xymatrix{ 
\SH^a_*(V; \alpha) \ar[rr]^-{\mathfrak h_{\phi}^a} \ar[d]_-{r_C}^-{\simeq} && \SH^a_*(U; \alpha) \ar[d]^-{r_C}_-{\simeq} \\
\SH^{Ca}_*(CV; \alpha) \ar[rr]_-{\mathfrak h_{\phi(C)}^{Ca}} && \SH^{Ca}_*(CU; \alpha)}\]
\item[(4)] Let $U \subset V$ and suppose a Liouville embedding $\phi: U \to V$ is isotopic to inclusion $i_U$ through Liouville embeddings, i.e. strongly unknotted. Then $\h_{\phi} = \h_{i_U}$.
\end{itemize}
\end{prop}

The proof of Proposition \ref{fp-sh} can be derived from Definition \ref{dfn-symp-pm} and is left to an interested reader. The proof is analogous to the proof in the case of star-shaped domains which is treated in \cite{PRSZ17}. The main difference between the two cases is in the way the grading is defined. Indeed, in the case of star-shaped domains this is done using a symplectic trivialization of the tangent bundle over a disc capping the orbit, while, as explained above, we use the vertical Lagrangian distribution. Hence, we should prove that in our case $\h_f$ preserves grading. We do this in the lemma that follows.

\begin{lemma}
Let $M$ be a closed, orientable manifold, $U,V \in \mathcal{C}_M$ non-degenerate domains and let $f:U\rightarrow V$ be a ${\tilde \pi}_1$-trivial Liouville embedding. Then $\h_f$ preserves grading.
\end{lemma}
\begin{proof}
For $H\in \mathcal H_{U}$, denote by $f_* H\in \mathcal H_{V}$ the extension by zero of $H\circ f^{-1}.$ Recall from \cite{PRSZ17} that $\h_f$ is induced by the map $f_*$ between Floer chain complexes of $H$ and $f_*H.$ This map sends a periodic orbit $z\in \mathcal P(H; \alpha)$ to a periodic orbit $f\circ z \in \mathcal P(f_*H; \alpha).$ Hence, we need to prove that $\ind_{\HF}(z)=\ind_{\HF}(f\circ z)$ where both indices are calculated using the vertical Lagrangian distribution. To this end, let $\Phi:S^1\times \R^{2n}\rightarrow z^*(TT^*M), \Psi: S^1\times \R^{2n}\rightarrow (f\circ z)^*(TT^*M)$ be symplectic trivializations such that for all $t\in S^1$ it holds
$$\Phi_t(0\times \R^n)=T_{z(t)}^vT^*M,~~\Psi_t(0\times \R^n)=T_{f(z(t))}^vT^*M.$$
By definition 
$$\ind_{\HF}(z)=\ind_{\rm CZ} (\Phi^{-1}_t \circ d\phi^H_t(z(0)) \circ \Phi_0),$$
as well as 
$$\ind_{\HF}(f\circ z)=\ind_{\rm CZ} (\Psi^{-1}_t \circ d\phi^{f_*H}_t(f(z(0))) \circ \Psi_0).$$
Since $f$ is a symplectic embedding, we have that $ d\phi^H_t = (df)^{-1} \circ d\phi^{f_*H}_t \circ df$ and thus
\begin{align*}
\ind_{\HF}(z) & = \ind_{\rm CZ}(\Phi^{-1}_t \circ (df)^{-1} \circ d\phi^{f_*H}_t \circ df \circ \Phi_0) \\
& = \ind_{\rm CZ}(\Phi^{-1}_t \circ (df)^{-1} \circ \Psi_t \circ \Psi_t^{-1} \circ d\phi^{f_*H}_t \circ \Psi_0 \circ \Psi_0^{-1} \circ df \circ \Phi_0)\\
& = \ind_{\rm CZ}(\theta(t)\circ \Phi^{-1}_0 \circ (df)^{-1} \circ \Psi_0 \circ \Psi_t^{-1} \circ d\phi^{f_*H}_t \circ \Psi_0 \circ \Psi_0^{-1} \circ df \circ \Phi_0),
\end{align*}
where $\theta(t)=\Phi_t^{-1}\circ (df)^{-1} \circ \Psi_t \circ \Psi_0^{-1} \circ df \circ \Phi_0$ is a loop of symplectic matrices. One readily checks that $\theta(0)=\theta(1)=\mathbb{1}$ and hence, using loop and naturality properties of the Conley-Zehnder index, we have that 
\begin{align*}
\ind_{\HF}(z) & = \ind_{\rm CZ}(\Phi^{-1}_0 \circ (df)^{-1} \circ \Psi_0 \circ \Psi_t^{-1} \circ d\phi^{f_*H}_t \circ \Psi_0 \circ \Psi_0^{-1} \circ df \circ \Phi_0) +2\mu(\theta) \\
& = \ind_{\rm CZ} (\Psi^{-1}_t \circ d\phi^{f_*H}_t \circ \Psi_0) +2\mu(\theta) \\
& = \ind_{\HF}(f\circ z) + 2\mu(\theta),
\end{align*}
where $\mu$ denotes the Maslov index. Thus, our goal is to show that $\mu(\theta)=0.$

Fix a Lagrangian subspace $V_0\subset \R^{2n}$ given by $V_0=\Phi_0^{-1}((df)^{-1}(T^v_{f(z(0))}T^*M)).$ Loop $\theta$ induces a loop of Lagrangian subspaces $\Lambda(t)=\theta(t)V_0$ and from properties of $\Phi$ and $\Psi$ it follows that $\Lambda(t)=\Phi^{-1}_t((df)^{-1}(T^v_{f(z(t))}T^*M)).$ Let $G:[0,1]\times S^1\rightarrow U$ be a homotopy such that $G_0(t)=z(t),~G_1(t)=\bar{z}(t)$ for $\bar{z}:S^1\rightarrow M \subset U$ (such $G$ exists because $U$ is fiberwise star-shaped). Let $\widetilde{\Phi}:[0,1]\times S^1 \times \R^{2n} \rightarrow G^*(TT^*M)$ be a symplectic trivialization such that for all $t\in S^1,~ \widetilde{\Phi}_{0,t}=\Phi_t$ and for all $s\in [0,1], t\in S^1$ it holds\footnote{The existence of such a trivialization $\widetilde{\Phi}$ follows from an argument similar to the proof of Lemma 1.7 in \cite{AS06}. In \cite{AS06}, $\widetilde{\Phi}_{1,t}$ is predetermined and hence $\widetilde{\Phi}_{s,t}(0\times \R^n)=T^v_{G(s,t)}T^*M$ only holds for $s=0,1.$ One may notice that this weaker condition would also be sufficient for our purposes.} $\widetilde{\Phi}_{s,t}(0\times \R^n)=T^v_{G(s,t)}T^*M.$ Denote by $\bar{\Phi}_t=\widetilde{\Phi}_{1,t}:S^1\times \R^{2n}\rightarrow \bar{z}^*(TT^*M).$ Now, $\widetilde{\Phi}_{s,t}^{-1}((df)^{-1}(T^v_{f(G(s,t))}T^*M))$ provides a homotopy between loops $\Lambda(t)$ and $\bar{\Lambda}(t):=\bar{\Phi}_t^{-1}((df)^{-1}(T^v_{f(\bar{z}(t))}T^*M))$ of Lagrangian subsapaces of $\R^{2n}.$ Thus,
\begin{align*}
\mu(\theta)=\mu(\Lambda)=\mu(\bar{\Lambda})=\mu(\bar{\Phi}_t^{-1}((df)^{-1}(T^v_{f(\bar{z}(t))}T^*M)), 0\times \R^n) \\ =\mu(\bar{\Phi}_t^{-1}((df)^{-1}(T^v_{f(\bar{z}(t))}T^*M)), \bar{\Phi}_t^{-1}(T^v_{\bar{z}(t)}T^*M)),  
\end{align*}
where $\mu(\cdot,\cdot)$ denotes the relative Maslov index for the pair of Lagrangian paths, see \cite{RS93}.

Now, notice that $(df)^{-1}(T^vT^*M)$ and $T^vT^*M$ are Lagrangian subbundles of the symplectic vector bundle $TT^*M$ over $M$, and define their Maslov class, $\mu_{T^vT^*M,(df)^{-1}(T^vT^*M)}\in H^1(M; \Z)$, see \cite{AudHolo} for the definition and properties of the Maslov class of a pair of Lagrangian subbundles. Moreover, it holds
\begin{equation}
\mu(\bar{\Phi}_t^{-1}((df)^{-1}(T^v_{f(\bar{z}(t))}T^*M)), \bar{\Phi}_t^{-1}(T^v_{\bar{z}(t)}T^*M))=\mu_{T^vT^*M,(df)^{-1}(T^vT^*M)}([\bar{z}]).
\end{equation}
Since the bundles $T^vT^*M$ and $TM$ are fiberwise transversal Lagrangian subbundels of $TT^*M$, we have that $\mu(TM,T^vT^*M)=0$ and hence
\begin{equation}
\mu_{T^vT^*M,(df)^{-1}(T^vT^*M)}=\mu_{TM,T^vT^*M}+\mu_{T^vT^*M,(df)^{-1}(T^vT^*M)}=\mu_{TM,(df)^{-1}(T^vT^*M)}.    
\end{equation}
Usind $df$ to identify $TT^*M$ and $f^*(TT^*M)$ as symplectic vector bundles over $M,$ we have that
\begin{equation}
\mu_{TM,(df)^{-1}(T^vT^*M)}=\mu_{TM,f^*(T^vT^*M)}=\mu_f,
\end{equation}
where $\mu_f$ denotes the Maslov classs of a Lagrangian immersion $f:M\rightarrow T^*M.$ Since $f$ is actually an exact Lagrangian embedding, it follows from \cite[Appendix E]{Kra13} that $\mu_f=0$ and we have $\mu(\theta)=\mu_f([\bar{z}])=0,$ which proves that $\ind_{\HF}(z)=\inf_{\HF}(f\circ z).$
\end{proof}

\medskip

As explained in the introduction, in order to use standard (additive) parametrization of persistence modules, we also consider a logarithmic version of $\mathbb{SH}_{*, \alpha}(U)$. 

\begin{dfn} \label{dfn-log-sm} For $t\in \R,$ let $S^t_{*}(U; \alpha) = \SH^{e^t}_{*}(U; \alpha)$. Define a logarithmic version of the symplectic persistence module associated to $U \in \mathcal C_M$ as
\[ \mathbb S_{*, \alpha}(U) = \left\{\{S^t_{*, \alpha}(U)\}_{t \in \R}; \{\iota_{s,t}=\iota^{\SH}_{e^s,e^t}: S^s_{*}(U; \alpha) \to S^t_{*}(U; \alpha)\}_{s\leq t} \right\}.\]
\end{dfn}

\subsection{Proof of Theorem \ref{TST}}
\begin{proof} The second inequality directly follows from the first one and Proposition \ref{SvsR}. Thus we will prove the first inequality. By Definition \ref{dfn-SBM}, for any $\varepsilon>0$, there exists $C \geq 1$ with $\ln C \leq d_{SBM}(U,V) + \varepsilon$ such that 
\begin{itemize}
\item[$(\ast)$] $ U/C \xhookrightarrow{\phi} V \xhookrightarrow{\psi} CU$ and $\psi \circ \phi$ is strongly unknotted;
\item[$(\ast\ast)$] $ V/C \xhookrightarrow{\psi(C^{-1})} U\xhookrightarrow{\phi(C)} CV$ and $\phi(C) \circ \psi(C^{-1})$ is strongly unknotted.
\end{itemize}

Then, (1) and (4) in Proposition \ref{fp-sh} together with $(\ast)$ imply for any positive $a$, $\mathfrak h^a_{\phi} \circ \mathfrak h^a_{\psi} = \mathfrak h^a_{\psi \circ \phi} = \mathfrak h^a_{i}$ where $i$ is the inclusion $i: U/C \to CU.$ Moreover, (2) in Proposition \ref{fp-sh} implies that the following diagram commutes
\begin{equation} \label{sh-com}
\xymatrixcolsep{4pc} \xymatrix{
\SH^a_{*}(CU; \alpha) \ar[r]^{\mathfrak h^a_{\psi}} \ar[rdd]_{\mathfrak h^a_{i'}} \ar@/_2pc/[rr]_-{\mathfrak h_{i}^a} & \SH^a_*(V; \alpha)\ar[r]^-{\mathfrak h^a_{\phi}}&  \SH_{*}^a(U/C; \alpha)  \ar[dd]_{\simeq}^{r_C} \\
& & & \\
\SH_{*}^{\frac{a}{C}}(U; \alpha) \ar[r]_{\iota^{\mathbb{SH}_{*,\alpha}(U)}_{a/C, a}} \ar[uu]_-{\simeq}^-{r_C} & \SH_{*}^a(U; \alpha) \ar[r]_{\iota^{\mathbb{SH}_{*,\alpha}(U)}_{a, Ca}} \ar[ruu]_{\mathfrak h^a_{i''}} & \SH_{*}^{Ca}(U; \alpha)}
\end{equation}
where $\mathfrak h_{i'}$ is induced by the inclusion $i': U \to CU$ and $\mathfrak h_{i''}$ is induced by the inclusion $i'': U/C \to U$. Now set 
\begin{itemize}
    \item{} $\Psi : = \mathfrak h_{\psi} \circ r_C$ where $\Psi^{a/C} = \mathfrak h_{\psi}^a \circ r_C: \SH_*^{\frac{a}{C}}(U; \alpha) \to \SH_*^a(V; \alpha)$;
    \item{} $\Phi: = r_C \circ \mathfrak h_{\phi}$ where $\Phi^a = r_C \circ \mathfrak h^a_{\phi}: \SH_*^a(V; \alpha) \to \SH^{Ca}(U; \alpha)$.
\end{itemize} 
For any positive $a$, (\ref{sh-com}) implies 
\begin{align*}
    \Phi^a \circ \Psi^{a/C} & = (r_C \circ \mathfrak h_{\phi}^a) \circ (h_{\psi}^{a} \circ r_C) \\
    & = r_C \circ \mathfrak h_{i''}^a \circ \mathfrak h_{i'}^a \circ r_C \\
    & = \iota^{\mathbb{SH}_{*,\alpha}(U)}_{a, Ca} \circ \iota^{\mathbb{SH}_{*,\alpha}(U)}_{a/C, a} = \iota^{\mathbb{SH}_{*,\alpha}(U)}_{a/C, Ca}.
\end{align*}

Similarly to (\ref{sh-com}) , (1), (2) and (4) in Proposition \ref{fp-sh} together with ($\ast\ast$) give a commutative diagram 
\begin{equation} \label{sh-com-2}
\xymatrixcolsep{4pc} \xymatrix{
\SH^a_{*}(CV; \alpha) \ar[r]^{\mathfrak h^a_{\phi(C)}} \ar[rdd]_{\mathfrak h^a_{j'}} \ar@/_2pc/[rr]_-{\mathfrak h_{j}^a} & \SH^a_*(U; \alpha)\ar[r]^-{\mathfrak h^a_{\psi(C^{-1})}}&  \SH_{*}^a(V/C; \alpha)  \ar[dd]_{\simeq}^{r_C} \\
& & & \\
\SH_{*}^{\frac{a}{C}}(V; \alpha) \ar[r]_{\iota^{\mathbb{SH}_{*,\alpha}(V)}_{a/C, a}} \ar[uu]_-{\simeq}^-{r_C} & \SH_{*}^a(V; \alpha) \ar[r]_{\iota^{\mathbb{SH}_{*,\alpha}(V)}_{a, Ca}} \ar[ruu]_{\mathfrak h^a_{j''}} & \SH_{*}^{Ca}(V; \alpha)}
\end{equation}
where $\mathfrak h_{j}$, $\mathfrak h_{j'}$ and $\mathfrak h_{j''}$ are induced by the inclusions $j: V/C \to CV$, $j': V \to CV$ and $j'': V/C \to V$ respectively. Moreover, applied to $\Psi$ and $\Phi$ which we defined above, (3) in Proposition \ref{fp-sh} gives 
\[ \Phi^{a/C} =r_C \circ \mathfrak h_{\phi}^{a/C} =  \mathfrak h_{\phi(C)}^a \circ r_C  \,\,\,\,\mbox{and}\,\,\,\, \Psi^a = h_{\psi}^{Ca} \circ r_C = r_C \circ h_{\psi(C^{-1})}^{a}.\]
Then commutative diagram (\ref{sh-com-2}) implies 
\begin{align*}
\Psi^a \circ \Phi^{a/C} & = (h_{\psi}^{Ca} \circ r_C) \circ (r_C \circ \mathfrak h_{\phi}^{a/C}) \\
& = r_C \circ h_{\psi(C^{-1})}^{a} \circ \mathfrak h_{\phi(C)}^a \circ r_C\\
& = r_C \circ \mathfrak h^a_{j''} \circ \mathfrak h^a_{j'} \circ r_C \\
& = \iota^{\mathbb{SH}_{*,\alpha}(V)}_{a, Ca} \circ \iota^{\mathbb{SH}_{*,\alpha}(V)}_{a/C, a} = \iota^{\mathbb{SH}_{*,\alpha}(V)}_{a/C, Ca}. 
\end{align*}
Therefore, passing to the logarithmic version of symplectic persistence modules defined in Definition \ref{dfn-log-sm}, the existence of the pair $(\Phi, \Psi)$ implies that $\mathbb S_{*, \alpha}(U)$ and $\mathbb S_{*, \alpha}(V)$ are $(\ln C)$-interleaved. Hence, by Theorem \ref{iso-thm}
\[ d_{bottle}(\mathbb B_{*,\alpha}(U), \mathbb B_{*,\alpha}(V)) = d_{inter}(\mathbb {S}_{*,\alpha}(U), \mathbb {S}_{*,\alpha}(V)) \leq \ln C \leq d_{SBM}(U,V) + \varepsilon.\]
We draw the conclusion by letting $\varepsilon \rightarrow 0.$ \end{proof}

\section{Filtered homology of the free loop space}

In this section we review basic notions about the homology of the free loop space filtered by energy and show how this filtered homology relates to symplectic homology of the unit codisc bundle.

\subsection{Morse-Bott perspective} \label{ss-Mor-B}

Let $(M,g)$ be a closed, orientable, Riemannian manifold, $\alpha$ a homotopy class of free loops in $M$ and $\mathcal L_\alpha(M)$ the space of smooth loops in $M$ in class $\alpha$. Recall that {\it the energy functional} $E_g: \mathcal L_{\alpha}(M) \to \R$ is defined as $E_g(\gamma) = \int_0^1 \frac{||\dot\gamma||^2_g}{2} dt$ for any $\gamma \in \mathcal L_{\alpha} (M)$. This functional is never Morse, but rather Morse-Bott in a generic situation. In this subsection we briefly review some basic notions of Morse-Bott homology in the context of $E_g$. Our exposition mostly follows Section 4 in \cite{AS09}, which is based on \cite{AM} and \cite{Fra}. For other treatments of this topic, see \cite{Kli78, Oan14}.

Let $f:W\rightarrow \R$ be a smooth function on a Hilbert manifold $W$ and assume that $ {\rm Crit} (f)$ consists of a disjoint union of closed submanifolds of $W.$ Hessian, ${\rm Hess}(f)_p$, at a point $p\in {\rm Crit} (f)$ is a bilinear form on $T_pW$, and we have that $T_p {\rm Crit}(f) \subset \ker ({\rm Hess}(f)_p).$ Let $N\subset {\rm Crit}(f)$ be a connected component and $p\in N$ a critical point. 

We define {\it nullity} of $p$ to be equal to $\dim (\ker ({\rm Hess}(f)_p)) - 1$ and {\it index} of $p$ to be the maximal dimension of a subspace of $T_pW$ on which  ${\rm Hess}(f)_p$ is negative definite. Both index and nullity are constant along $N$ and hence we may define index and nullity of $N$ as index and nullity of any point in $N.$ $N$ is said to be a {\it non-degenerate critical submanifold of} $f$ if $\ker ({\rm Hess}(f)_p) = T_p {\rm Crit}(f)$ for all $p\in N$ or equivalently if nullity of $N$ equals $\dim N -1.$

We consider $E_g$ as a functional on the space $W^{1,2}(S^1,M)\supset \mathcal L_{\alpha} (M)$ which is a Hilbert manifold. Critical points of $E_g$ are closed geodesics (this includes constant loops too). By a closed geodesic, we mean a closed curve $\gamma: S^1 \to M$ such that $\nabla_{\dot{\gamma}} \dot{\gamma} = 0$. In particular, $||\dot{\gamma}||_g^2 = \rm{constant}$. Constant geodesics form a critical submanifold diffeomorphic to $M.$ This critical submanifold is always non-degenerate and has index equal to 0 (see Proposition 2.4.6 in \cite{Kli78}).  On the other hand, any non-constant closed geodesic appears in an $S^1$-family corresponding to reparameterizations. More precisely, if $\gamma$ is a non-constant closed geodesic of constant speed, so is $s\cdot \gamma, s\in S^1$ given by $(s\cdot \gamma)(t)=\gamma(t+s).$ We say that a non-constant closed geodesic $\gamma$ is non-degenerate if $S^1\cdot \gamma$ is a non-degenerate critical submanifold, i.e. nullity of $S^1\cdot \gamma$ is zero.

\begin{dfn}[\cite{Abraham}]\label{bumpy}
A metric $g$ is called {\it bumpy} if all of its closed geodesics are non-degenerate. One may check that this definition is equivalent to $(U^*_gM,\lambda_{can})$ being a non-degenerate domain.
\end{dfn}

\begin{remark}
A generic Riemannian metric is bumpy, see \cite{Abraham,Anosov} for a precise statement.
\end{remark}

\begin{remark} In certain cases index and nullity of a closed geodesic $\gamma$ can be computed in a more direct way by analyzing Poincare return map and Jacobi vector fields along $\gamma.$ We will make this precise in Section \ref{sec-bulk} and use it to carry out calculations for the bulked spheres and multi-bulked surfaces. \end{remark}

 If ${\rm Crit}(E_g)$ consists only of non-degenerate critical submanifolds $E_g$ is called {\it Morse-Bott.} For a bumpy $g,$ $E_g$ is Morse-Bott and one may use it to define Morse-Bott homology. There are different approaches to constructing Morse-Bott homology (see \cite{Hurt} and references therein for finite dimensional cases) and we focus on the one described in \cite{Fra} which uses moduli spaces of flow lines with cascades. Let $g$ by a bumpy metric and pick an auxiliary Morse function $h$ on ${\rm Crit} (E_g)$, meaning Morse on each connected component of ${\rm Crit}(E_g).$ If $x\in {\rm Crit}(h),$ it follows that $x\in N,$ where $N\subset {\rm Crit}(E_g)$ is a connected critical submanifold of $E_g$ and we define {\it total index} of $x$ as
$$\ind_{E_g,h}(x)=\ind_{E_g}(N)+\ind_h(x),$$
where $\ind_{E_g}(N)$ denotes the index of $N$ as a critical submanifold and $\ind_h(x)$ denotes the standard Morse index. Slightly abusing the notation, throughout the paper we will write just $\ind$ when it is clear what is the index in question. Morse-Bott $k$-th chain group in homotopy class $\alpha$ is defined as
$$CMB_{k,\alpha}(E_g,h) = Span_{\Z_2} \big( \{ x\in {\rm Crit}(h) ~|~ [x]=\alpha, \ind_{E_g,h}(x)=k \} \big). $$
In order to define the differential, we introduce moduli spaces of flow lines with cascades. Fix two regular metrics\footnote{Regular means such that transversality is achieved in the definition of all the moduli spaces which appear. Such choice of metrics is generic.}, one on $W^{1,2}(S^1,M),$ the other one on ${\rm Crit} (E_g)$ and denote by $\nabla E_g$ and $\nabla h$ the gradient vector fields corresponding to these metrics. For $x,y \in {\rm Crit}(h)$ let
$$\mathcal{M}^{cas}_0(x,y)=\{ u: \R \rightarrow W^{1,2}(S^1,M) ~ | ~ \dot{u}= - \nabla h (u),~ u(-\infty)=x,~ u(+\infty)=y \}.$$
Note that $\mathcal{M}^{cas}_0(x,y)$ can only be non-empty if $x$ and $y$ belong to the same connected component of ${\rm Crit}(E_g).$ For $k\geq 1$ define $\mathcal{M}^{cas}_k(x,y)$ as the set of pairs $({\bf u}, {\bf t})$ where ${\bf u} = (u_1, \ldots , u_k)$ is a $k$-tuple of negative gradient flow lines $u_i:\R \rightarrow W^{1,2}(S^1,M),$ $$ \dot{u_i}=-\nabla E_g(u_i),$$
and ${\bf t}=(t_1, \ldots , t_{k-1})$  a $(k-1)$-tuple of non-negative numbers $t_i\geq 0$ such that
\begin{enumerate}
    \item $u_1(-\infty)\in W^u(x)$, $u_k(+\infty) \in W^s(y)$, where $W^u(x)$ and $W^s(y)$ denote respectively unstable and stable manifolds of $x$ and $y$ with respect to the flow of $-\nabla h.$
    \item For every $1\leq i \leq k-1$ there exists a negative gradient flow line $v_i:\R \rightarrow {\rm Crit}(E_g)$, $\dot{v_i}=-\nabla h(v_i)$ such that $$v_i(0)=u_i(+\infty),~ v_i(t_i)=u_{i+1}(-\infty).$$ 
\end{enumerate}

Now, $\R$ acts freely on each of $u_i$ by translations and thus $\R^k$ acts freely on $\mathcal{M}^{cas}_k(x,y)$ and we denote $$\bar{\mathcal{M}}^{cas}_k(x,y)= \mathcal{M}^{cas}_k(x,y) / \R^k \,\,\,\,\mbox{and}\,\,\,\, \bar{\mathcal{M}}^{cas}(x,y)= \bigcup_{k\geq 0} \bar{\mathcal{M}}^{cas}_k(x,y).$$
Regularity of the choice of metrics implies that $\bar{\mathcal{M}}^{cas}(x,y)$ is a smooth manifold of dimension $\ind_{E_g,h}(x)- \ind_{E_g,h}(y)-1 .$ When $\ind_{E_g,h}(x)= \ind_{E_g,h}(y)+1$ this manifold is zero-dimensional and compact, i.e. it is a finite set of points, and we denote by $n(x,y)$ the number of points in $\bar{\mathcal{M}}^{cas}(x,y)$ modulo 2. The Morse-Bott differential  $\partial: CMB_{k,\alpha}(E_g,h) \rightarrow CMB_{k-1,\alpha}(E_g,h),$ is given by
$$\partial x=\sum\limits_{y, \ind y=k-1} n(x,y)y.$$
It satisfies $\partial^2=0$ and the resulting Morse-Bott homology does not depend on the regular choices of two metrics, $h$ or $E_g.$ In fact, we have that
$$HMB_{k,\alpha}(E_g,h) \cong H_k( \mathcal  L_{\alpha}(M); \Z_2).$$
For our purposes it is essential to consider Morse-Bott chain complex together with the filtration by energy, i.e., we define
$$CMB^\lambda_{k,\alpha}(E_g,h) = Span_{\Z_2} \big( \{ x\in {\rm Crit}(h) ~|~ [x]=\alpha, \ind_{E_g,h}(x)=k, E_g(x) \leq \lambda \} \big). $$
Since $E_g$ decreases along the flow lines of $-\nabla E_g,$ $\partial$ restricts to $CMB^\lambda_{k,\alpha}(E_g,h)$ and we may define filtered Morse-Bott homology $HMB^\lambda_{k,\alpha}(E_g,h).$ In this case it holds
\begin{equation}\label{Filtered_Morse-Bott}
HMB^{\lambda}_{k,\alpha}(E_g,h) \cong H_k( \mathcal L^\lambda_{\alpha}(M) ; \Z_2),    
\end{equation}
and this isomorphism commutes with the maps induced from inclusions of sublevel sets $\{ E_g \leq \lambda \}.$

If $g$ is bumpy, then for all $\lambda\geq 0$ there are finitely many critical submanifolds in the sublevel $\{ E_g \leq \lambda \}$ and hence $CMB^\lambda_{k,\alpha}(E_g,h)$ is finitely generated (see Theorem 3.5 in \cite{Oan14} and references therein). This, together with (\ref{Filtered_Morse-Bott}) implies that the collection of data
\begin{equation}\label{homology}
\mathbb H_{*, \alpha}(M, g) = \left\{\begin{array}{l} \{H_*(\mathcal L_{\alpha}^{\lambda}(M,g); \Z_2)\}_{\lambda \in \R_{>0}};\\ \{\iota_{\lambda,\eta}: H_*(\mathcal L_{\alpha}^{\lambda}(M,g); \Z_2) \to H_*(\mathcal L_{\alpha}^{\eta}(M); \Z_2)\}_{\lambda \leq \eta}\end{array}\right\}
\end{equation}
forms a persistence module with $\Z_2$-coefficients, where $\iota_{\lambda, \eta}$ are induced by inclusion $\mathcal L_{\alpha}^{\lambda}(M,g) \hookrightarrow \mathcal L_{\alpha}^{\eta}(M,g)$ when $\lambda \leq \eta.$ Moreover, since the endpoints of bars in the barcode $\mathbb{B}(\mathbb H_{*, \alpha}(M, g))$ come from generators of $CMB^\lambda_{k,\alpha}(E_g,h),$ the barcode $\mathbb{B}(\mathbb H_{*, \alpha}(M, g))$ has only finitely many endpoints of bars below every fixed $\lambda\geq 0.$ These endpoints are equal to the energies of certain closed geodesics.

\begin{remark}\label{Below}
For isomorphism (\ref{Filtered_Morse-Bott}) to hold it is enough that all closed geodesics of energy not greater than $\lambda$ are non-degenerate. Indeed, one may apply the same considerations as above directly to $\mathcal L^\lambda_{\alpha}(M).$ 
\end{remark}

Let us sum up the important features of the above construction. Firstly, every non-constant, non-degenerate closed geodesic gives rise to a critical submanifold of $E_g$ diffeomorphic to $S^1.$ There exists a function on $S^1$ which has exactly 2 critical points of Morse index 0 and 1 (for example the standard height function). By picking the auxiliary function $h$ to be equal to such a function on each of the $S^1$-critical submanifolds, we obtain that to each non-constant, non-degenerate closed geodesic $\gamma$ correspond two critical points of $h$ whose total indices are equal to $\ind \gamma$ and $\ind \gamma+1.$ In other words, $\gamma$ produces two generators of the chain complex $CMB_{*,\alpha}(E_g,h),$ one in degree $\ind \gamma$ and the other one in degree $\ind \gamma + 1.$ On the other hand, critical submanifold of constant geodesics is diffeomorphic to $M$ and has index equal to 0. Hence it gives rise to critical points of $h$ whose total indices are equal to their Morse indices with respect to $h.$ In other words, if we view $h$ as a function on $M,$ each critical points of Morse index $k$ produces a generator of $CMB_{k,pt}(E_g,h).$ Finally the differential counts certain broken trajectories in $\mathcal L_\alpha M.$ Each broken trajectory can be viewed as a tuple of maps from a cylinder to $M$ connecting different closed geodesics. 

\subsection{The isomorphism with symplectic homology}\label{isomorphism}
In this subsection, we will elaborate a result which enables us to transfer computations from symplectic homology to the homology of the loop space. It states that, under certain parametrizations, filtered versions of these homologies are isomorphic as persistence modules, see Theorem \ref{sh-loop-thm}. We will use this result to describe the barcode of the symplectic persistence module associated to the unit cotangent bundle of metrics coming from our main geometric constructions, see Sections \ref{Proofs} and \ref{sec-bulk}.

The isomorphism between symplectic homology or Floer homology of the cotangent bundle and the homology of the loop space first appeared in \cite{Vit99}. Other versions of this isomorphism, constructed using different methods, have appeared in \cite{SW06,Web06,AS06,AS14, Abo15}. The version which fits our conventions is the one from \cite{Web06} and we give a short exposition of it below. Let us recall some notions first.

Given a closed, orientable, Riemannian manifold $(M, g)$ and a homotopy class $\alpha$ of free loops, denote the loop space of $M$ in class $\alpha$ by $\mathcal L_{\alpha} (M)$. We define the {\it length spectrum} of $g$ in class $\alpha,$ denoted by $\Lambda_\alpha,$ to be the set of lengths of all closed geodesics in class $\alpha.$ Recall that if $g$ is bumpy there are finitely many closed geodesics below any fixed energy level and hence $\Lambda_\alpha$ is discrete.

\begin{remark}\label{endpoints}
One may check that all the endpoints of all bars in $\mathbb{B}(\mathbb{SH}_*(U^*_g M; \alpha))$ belong to $\Lambda_\alpha.$ Indeed, it is enough to check that 
$$\iota_{a,b}:\SH_*^a(U^*_g M; \alpha) \rightarrow \SH_*^b(U^*_g M; \alpha)$$
is an isomorphism if $0<a\leq b$ are such that $[a,b]\cap \Lambda_\alpha=\emptyset.$ To prove this one considers a radially symmetric Hamiltonian $H(\xi)=h(\| \xi \|_{g^*})$ for $h:[0,+\infty) \rightarrow \R.$ If $h$ is a decreasing function such that $h|_{[0,1-\varepsilon]}=C,$ and $h|_{[1,+\infty)}=0,$ taking $C$ large enough (namely $C>b$) and $\varepsilon$ small enough one sees that there are no periodic orbits of $H$ in the action window $[a,b].$ This implies that 
$$\iota_{a,b}: \HF^{[a, \infty)}_*(H, \alpha) \rightarrow \HF^{[b, \infty)}_*(H, \alpha)$$
is an isomorphism and by taking $C\rightarrow +\infty,\varepsilon \rightarrow 0$ we get the desired conclusion. For more details see \cite{Web06}.
\end{remark}

Recall that
\[ \mathcal L_{\alpha}^{\lambda} (M,g) = \{ \gamma \in \mathcal L_{\alpha}(M) \,| \, E_g(\gamma) \leq \lambda\} .\]
We have expained in (\ref{homology}), that if $g$ is a bumpy metric then $\{H_*(\mathcal L_{\alpha}^{\lambda}(M,g); \Z_2)\}_{\lambda \in \R_{>0}}$ form a persistence module $\mathbb H_{*, \alpha}(M, g)$ such that $\mathbb{B}(\mathbb H_{*, \alpha}(M, g))$ has finitely many endpoints of bars below every fixed $\lambda.$ Moreover the endpoints of bars in $\mathbb{B}(\mathbb H_{*, \alpha}(M, g))$ are equal to energies of certain closed geodesics. We are now ready to state the result.

\begin{theorem} \label{sh-loop-thm} (\cite{Vit99,SW06,Web06,AS06,AS14,Abo15}) Let $(M, g)$ be a closed, orientable, Riemannian manifold with bumpy metric $g$ and $\alpha$ a homotopy class of free loops in $M$. There exists a family of isomorphisms
$$\Phi_a: \SH_*^a(U_g^*M ;\alpha) \to {H}_*(\mathcal L_{\alpha}^{a^2/2} (M); \Z_2),$$
for $a>0,$ which commute with comparison maps. In other words, under suitable parameterizations, $\Phi$ is a persistence isomorphism.
\end{theorem}

For every $a \in \R_{>0} \backslash \Lambda_{\alpha}$, the isomorphism $\Phi_a$ as in Theorem \ref{sh-loop-thm} has been constructed in Theorem 3.1 in \cite{Web06}. Our goal is to show that this $\Phi$ commutes with comparison maps and then to extend it to a persistence isomorphism for all $a\in \R_{>0}.$ For reader's convenience, let us review the construction of $\Phi_a$ from \cite{Web06} first.

According to Definition \ref{dfn-symp-pm}, in order to study symplectic homology, we need to understand the associated Hamiltonian Floer homologies. In Subsection \ref{symplectic-persitence} we explained how to associate Floer homology $\HF_*^{[a,b)}(H, \alpha)$ to $H\in \mathcal H_{U, a,b}.$ In an analogous fashion, instead of considering compactly supported Hamiltonians, one may look at the sets 
\[ \mathcal K_{U^*_g M} = \{H:S^1\times T^*M\rightarrow \R \,| \, \exists \beta>0,\beta'\in \R ~ {\rm s.t.}~ H_t(\xi)=-\beta\| \xi \|_{g^*}+\beta' ~ {\rm for} ~ \| \xi \|_{g^*}\geq 1 \} ,\]
and 
\[ \mathcal K_{U^*_g M,a,b} = \{H\in K_{U^*_g M} \,| \, a, b \notin {\rm Spec}(H; \alpha) ~ {\rm and}~ {\rm either} ~ \beta \notin \Lambda_\alpha ~ {\rm or} ~ \beta' \notin [a,b] \} .\]
Hamiltonians compactly supported inside $U^*_g M$ constitute the case $\beta=\beta'=0$ and thus $ \mathcal H_{U^*_g M, a,b} \subset \mathcal K_{U^*_g M, a,b}.$ As in the compactly supported case if $H_1 \preceq H_2$ one may define a continuation map associated to a monotone homotopy (manifestly if $H_1 \preceq H_2$ slopes satisfy $-\beta_{H_1}\geq -\beta_{H_2}$). Analogously, a monotone homotopy $H_{\tau}$ such that $H_{\tau}\in \mathcal K_{U^*_g M, a,b} $ for every $\tau$ is called {\it action-regular}. Continuation maps will have the same properties as before, in particular the existence of an action-regular monotone homotopy will imply that the corresponding continuation map is an isomorphism, see \cite{Web06} for more details.

We are now ready to define $\Phi_a.$ Fix $a \in \R_{>0} \backslash \Lambda_{\alpha}$ and take a radially symmetric Hamiltonian $H \in \mathcal H_{U^*_gM}$, $H(\xi)=h(\| \xi \|_{g^*})$ with $h:[0,+\infty)\rightarrow \R$ such that $h|_{[0,1-\varepsilon)}=C>a,$ $h$ is decreasing and $h|_{[1,+\infty)}=0.$ Moreover consider a monotone homotopy, shown in Figure \ref{htp1}, from $H$ to a new Hamiltonian $\tilde H_a$  which we obtain by making the ``tail'' of $H$ linear with slope $-a$.

\begin{figure}[ht]
\includegraphics[scale=0.5]{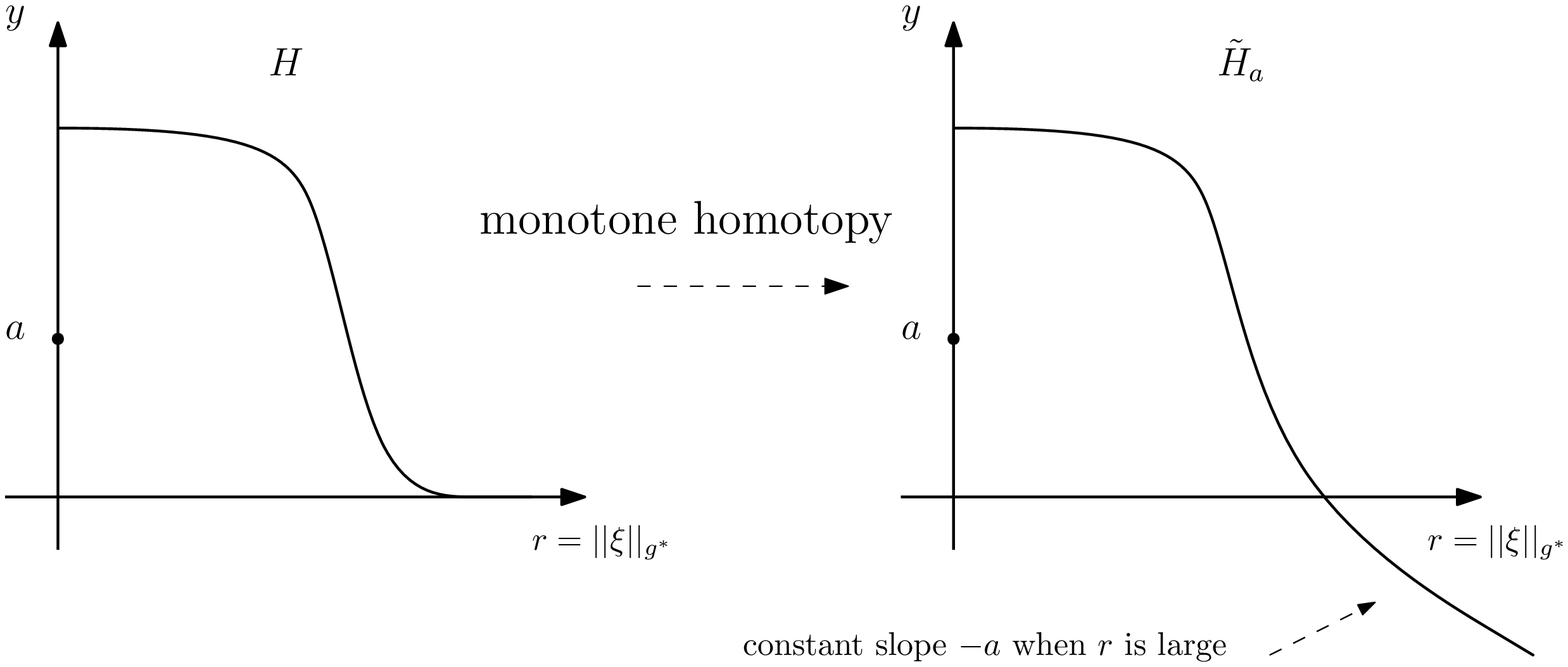}
\caption{Monotone homotopy 1}
\label{htp1}
\end{figure}

More precisely, one sees that there are exactly two points $r_1, r_2 \in [0,1]$ for which $h'(r_i)=-a$ for $i =1, 2$. If we label them by $r_1\leq r_2$ then $h(r_1)\approx C$ and $h(r_2) \approx 0.$ Up to a small smoothing at $\|\xi \|_{g^*}= r_2,$ Hamiltonian $\tilde H_a$ is equal to $H$ for $\|\xi \|_{g^*}\in [0,r_2]$ and is linear with slope $-a$ for $\|\xi \|_{g^*}\in [r_2,+\infty).$

The monotone homotopy from $H$ to $\tilde{H}_a$ gives the isomorphism 
\begin{equation} \label{iso-1}
c_a: \HF_*^{[a,\infty)}(H, \alpha)  \xrightarrow{\simeq}  \HF_*^{[a,\infty)}(\tilde H_a, \alpha). 
\end{equation}
because no Hamiltonian 1-periodic orbit with action in the action window $[a, \infty)$ appears during the homotopy, i.e. it is action-regular.

Moreover, we construct the third Hamiltonian, denoted by $H_a$, in the similar fashion to the construction of $\tilde H_a.$ Namely, up to a small smoothing at $\|\xi \|_{g^*}= r_1,$ $H_a$ coincides with $H$ on the set $\|\xi \|_{g^*}\in [0,r_1]$ and is linear with slope $-a$ for $\|\xi \|_{g^*}\in [r_1,+\infty).$

Now, there exists another action-regular monotone homotopy from $H_a$ to $\tilde H_a,$ see Figure \ref{htp2}, which provides another isomorphism 
\begin{equation} \label{iso-2}
s_a: \HF_*^{[a,\infty)}(H_a, \alpha) \xrightarrow{\simeq} \HF_*^{[a,\infty)}(\tilde{H}_a, \alpha)   
\end{equation}
 
\begin{figure}[ht]
\includegraphics[scale=0.5]{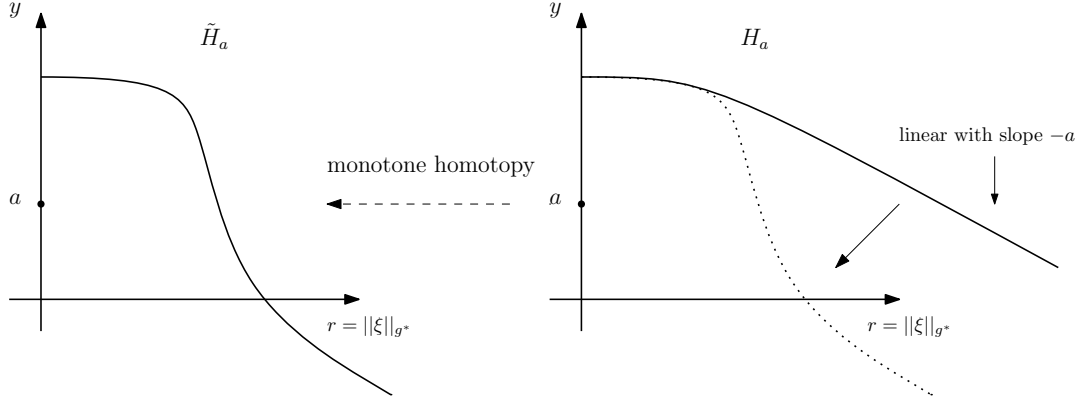}
\caption{Monotone homotopy 2}
\label{htp2}
\end{figure}

Following \cite{Web06}, for a radially symmetric Hamiltonian function $H = h(r),$ $r = ||\xi||_{g^*}$ with $h''(r)\leq 0$, define for any $\lambda \in \R_{\geq 0}$,
\begin{equation} \label{y-int} 
C(H, \lambda) = \lambda r_* + h(r_*) \,\,\,\,\mbox{where}\,\,\,\, h'(r_*) = -\lambda,
\end{equation} 
if such $r_*$ exists. Observe that $C(H, \lambda)$ is the $y$-intercept of the line passing through point $(r_*, h(r_*))$ with slope $-\lambda$. Since the Hamiltonian function $H_a$ in Figure \ref{htp2} is concave with respect to $r$, value $C(H_a, \lambda)$ is well-defined for all $\lambda \in [0,a]$. Now, the advantage of considering Hamiltonian $H_a$ is that it does not have any Hamiltonian 1-periodic orbit of action less than $a$. On the other hand, the maximal action of the Hamiltonian 1-periodic orbit of $H_a$ is less than $C(H_a, a)$. 
Therefore, one gets the following isomorphisms 
\begin{equation} \label{iso-3}
\HF_*^{(-\infty, C(H_a, a))}(H_a, \alpha) \xrightarrow[\simeq]{i_{H_a}^a} \HF_*^{(-\infty,\infty)}(H_a, \alpha)  \xrightarrow[\simeq]{\pi_{H_a}^a}\HF_*^{[a, \infty)}(H_a, \alpha)
\end{equation}
from the equality on the chain level. For more details regarding all the constructions see Section 3 in \cite{Web06}.

Finally, Theorem 2.9 in \cite{Web06} claims that there exists an isomorphism
\begin{equation} \label{iso-4}
\psi_{H_a}^a: \HF_*^{(-\infty, C(H_a, a))}(H_a, \alpha) \xrightarrow{\simeq} H_*(\mathcal L_{\alpha}^{a^2/2} (M); \Z_2).
\end{equation}
Map $\psi_{H_a}^a$ essentially comes from the main result in \cite{SW06} which compares the symplectic action functional with a certain energy functional on the loop space. Combining all the above defined isomorphisms together, one obtains the following isomorphism,  
\begin{equation} \label{web-iso}
\Phi_{H,a}: \HF_*^{[a,\infty)}(H, \alpha)  \to H_*(\mathcal L_{\alpha}^{a^2/2} (M); \Z_2)
\end{equation}
where $\Phi_{H,a} = \psi_{H_a}^a \circ (i_{H_a}^a)^{-1} \circ (\pi_{H_a}^a)^{-1}\circ s_a^{-1} \circ c_a$. The desired isomorphism $\Phi_a$ is then given by $\Phi_a = \varprojlim_{H \in \mathcal H_{U_g^*M}} \Phi_{H, a}$.

\begin{proof} (Proof of Theorem \ref{sh-loop-thm}) It follows from the definitions that both persistence modules $\mathbb{SH}_{*, \alpha}(U_g^*M)$ and $\mathbb H_{*, \alpha}(M, g)$ are such that all the bars in their barcodes have left endpoints closed and right endpoints open. Moreover, by Remark \ref{endpoints}, the endpoints of bars in the barcode of $\mathbb{SH}_{*, \alpha}(U_g^*M)$ belong to $\Lambda_\alpha$ and since $g$ is bumpy, for every fixed $\lambda >0,$ $\Lambda_\alpha\cap [0,\lambda]$ is finite. Thus, it is enough to prove Theorem \ref{sh-loop-thm} for $a \in \R_{>0} \backslash \Lambda_{\alpha}$ and afterwards extend $\Phi_a$ to $a\in \Lambda_\alpha$ by continuity.

From the definition of symplectic homology, it readily follows that for any $a<b$, there exists a single $H \in \mathcal H_{U^*_gM, \{a,b\}}$ such that $\iota_{a,b}^{\mathbb{SH}}: \SH_*^a(U_g^*M ;\alpha) \to \SH_*^b(U_g^*M ;\alpha)$ can be seen as 
\[ \iota_{a,b}^{\HF}: \HF_*^{[a, \infty)}(H, \alpha) \to \HF_*^{[b, \infty)}(H, \alpha). \]
The example of such $H$ which we consider is a radially symmetric Hamiltonian, shown in Figure \ref{H}, such that $\max{H} \geq b$ (thus also $\max{H} \geq a$), $H$ is equal to $\max{H}$ for $\|\xi\|_{g^*}=r \leq 1- \varepsilon$ with some small $\varepsilon>0$ and is decreasing in $r.$ 
\begin{figure}[ht]
\includegraphics[scale=0.5]{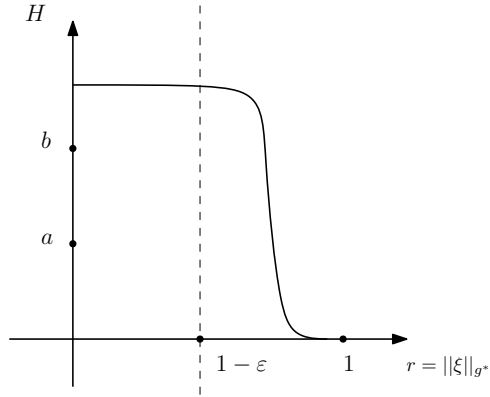}
\caption{A radially symmetric Hamiltonian which computes symplectic homology}
\label{H}
\end{figure}

Using this $H$, we can carry out monotone homotopies as described above and shown in Figure \ref{htp1} and Figure \ref{htp2} for both slopes $a$ and $b.$ This way, we obtain new Hamiltonian functions $\tilde H_a$, $\tilde H_b$ as well as $H_a, H_b$, see Figure \ref{HH}.
\begin{figure}[ht]
\includegraphics[scale=0.5]{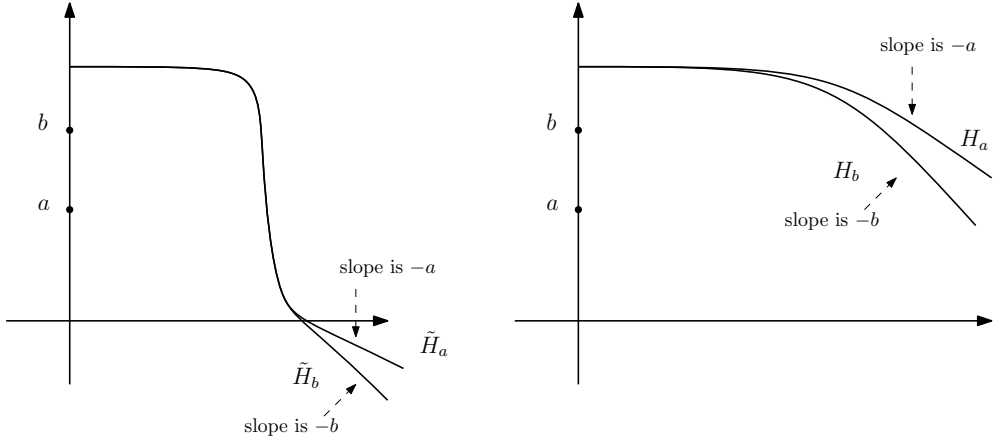}
\caption{New Hamiltonian functions coming from monotone homotopies}
\label{HH}
\end{figure}
We claim that the following diagram commutes. 
\[ \xymatrix{
& \HF_*^{[a,\infty)}(H, \alpha) \ar[ld]_-{c_a} \ar[d]^-{c'_b} \ar[rr]^-{\iota_{a,b}^{\HF}} & & \HF_*^{[b,\infty)}(H, \alpha) \ar[d]^-{c_b} \\
\HF_*^{[a,\infty)}(\tilde{H}_a, \alpha) \ar[r]^-{\tilde{c}} \ar@/^/[d]^-{s_a^{-1}} &  \HF_*^{[a,\infty)}(\tilde{H}_b, \alpha) \ar[rr]^-{inc_*} & & \HF_*^{[b,\infty)}(\tilde{H}_b, \alpha)\ar@/^/[d]^-{s_b^{-1}}\\
\HF_*^{[a,\infty)}(H_a, \alpha) \ar[u]^-{s_a} \ar[r]^-{c} &  \HF_*^{[a,\infty)}(H_b, \alpha) \ar[rr]^-{inc_*} \ar[u]^-{s_b'}& & \HF_*^{[b,\infty)}(H_b, \alpha) \ar[u]^-{s_b}.} \] 
Commutativity comes from the following arguments: 
\begin{itemize}
\item{} In the upper-left triangle, $c'_b$ is induced from a monotone homotopy from $H$ to $\tilde{H}_b$ and $\tilde{c}$ is induced from a monotone homotopy from $\tilde{H}_a$ to $\tilde{H}_b$. Because $H \preceq \tilde{H}_a \preceq \tilde{H}_b$, $c'_b = \tilde{c} \circ c_a$ comes from Lemma \ref{4-lemma-1}. 
\item{} In the lower-left rectangle, $s_b'$ is induced from a monotone homotopy from $H_b$ to $\tilde{H}_b$ and $c$ is induced from a monotone homotopy from $H_a$ to $H_b$. Because $H_a \preceq \tilde H_a \preceq \tilde H_b$ and $H_a \preceq H_b \preceq \tilde H_b$, from (\ref{com-1}) we get $\tilde{c} \circ s_a = s'_b \circ c$, which implies $\tilde{c} = s'_b \circ c \circ s_a^{-1}$ where $s_a^{-1}$ is the inverse of $s_a$ ($s_a$ is an isomorphism by (\ref{iso-2})).
\item{} The upper-right rectangle trivially commutes because we may take monotone homotopy inducing $c'_b$ to be the same as the monotone homotopy inducing $c_b$ and hence the maps count the same Floer trajectories. 
\item{} The lower-right rectangle trivially commutes by the same reason as above, which implies $inc_* = s^{-1}_b \circ inc_* \circ s'_b$. 
\end{itemize}
Finally, we also claim that the following diagram commutes. 
\[ \xymatrix{
\HF_*^{[a, \infty)}(H_a, \alpha) \ar[r]^c \ar@/^/[d]^-{(\pi_{H_a}^a)^{-1}} & \HF_*^{[a, \infty)}(H_b, \alpha) \ar[r]^{inc_*} & \HF_*^{[b, \infty)}(H_b, \alpha) \ar@/^/[d]^-{(\pi_{H_b}^b)^{-1}}\\
\HF_*^{(-\infty, \infty)}(H_a, \alpha) \ar[r]^{c} \ar[u]^-{\pi_{H_a}^a} \ar@/^/[d]^-{(i_{H_a}^a)^{-1}} & \HF_*^{(-\infty, \infty)}(H_b, \alpha) \ar[r]^-{\mathds{1}} \ar[u]^-{\pi_{H_b}^a} & \HF_*^{(-\infty, \infty)}(H_b, \alpha) \ar[u]^-{\pi_{H_b}^b} \ar@/^/[d]^-{(i_{H_b}^b)^{-1}}\\
\HF_*^{(-\infty, C(H_a, a))}(H_a, \alpha) \ar[r]^-{c} \ar[u]^-{i_{H_a}^a} \ar[rd]_-{\psi_{H_a}^a} & \HF_*^{(-\infty, C(H_b, a))}(H_b, \alpha) \ar[r]^-{inc_*} \ar[u]^-{i_{H_b}^a} \ar[d]^-{\psi_{H_b}^a} & \HF_*^{(-\infty, C(H_b, b))}(H_b, \alpha) \ar[u]^-{i_{H_b}^b} \ar[d]^-{\psi_{H_b}^b}\\
& H_*(\mathcal L_{\alpha}^{a^2/2}(M)) \ar[r]^-{\iota_{a^2/2, b^2/2}^{\mathbb H}} &  H_*(\mathcal L_{\alpha}^{b^2/2}(M)).} \]
The only non-trivial commutativity is of the lower-left triangle and the lower-right rectangle. The former comes from the second proposition of Theorem 2.9 in \cite{Web06} while the latter comes from the third proposition of Theorem 2.9 in \cite{Web06}. Notice that maps $\pi_{H_a}^a, i_{H_a}^a$ and $\pi_{H_b}^b, i_{H_b}^b$ are all isomorphisms by (\ref{iso-3}), but $i_{H_b}^a$ is not an isomorphism. Denote by $\iota_{a^2/2, b^2/2}^{\mathbb H}$ the persistence comparison map from filtration level $a^2/2$ to filtration level $b^2/2$ of persistence module $\mathbb H_{*, \alpha}(M,g)$. Using the definition of $\Phi_{H,a}$ given by (\ref{web-iso}) and the two commutative diagrams above we obtain $\iota_{a^2/2, b^2/2}^{\mathbb H} \circ \Phi_{H,a} = \Phi_{H,b} \circ \iota_{a,b}^{\HF}$, which finished the proof. \end{proof}

\section{Proofs of Proposition \ref{emb-thm} and Proposition \ref{emb-thm2} (lower bounds)}\label{Proofs}

In this section, we prove lower bounds in Propositions \ref{emb-thm} and \ref{emb-thm2}. To this end, we will describe two classes of Riemannian metrics which realize quasi-isometric embeddings in Propositions \ref{emb-thm} and \ref{emb-thm2}. The first class of metrics will be defined on $S^2$ and metrics in this class will be called {\it bulked sphere metrics} on $S^2$. The other class will be defined on a closed, orientable surface $\Sigma$ of genus at least 1, and metrics in this class will be called {\it multi-bulked metrics} on $\Sigma$. The way we construct these metrics enables us to precisely analyze closed geodesics and prove that they have various nice properties, see Propositions \ref{bsm-p} and \ref{mb-p}. Then, using Theorem \ref{sh-loop-thm}, we are able to describe parts of the barcodes of the corresponding symplectic persistence modules. Finally, the lower bounds in both Proposition \ref{emb-thm} and Proposition \ref{emb-thm2} comes from the stability property - Theorem \ref{TST} and a combinatorial result - Lemma \ref{opt-matching}, which we will now prove. 

\subsection{A combinatorial lemma}
The following combinatorial lemma says that a particular shape of barcodes can help us get a lower bound on the bottleneck distance. 

\begin{lemma} \label{opt-matching}
Let $\B_1$ and $\B_2$ be two barcodes. Let $a_1 \geq ... \geq a_n$ be the $n$ smallest left endpoints of bars in $\B_1$ and denote by $[a_1, C_{a_1}), ..., [a_n, C_{a_n}) \in \B_1$ the corresponding bars. Similarly let $b_1 \geq ... \geq b_n$ be the $n$ smallest left endpoints of bars in $\B_2$ with corresponding bars $[b_1, C_{b_1}), ..., [b_n, C_{b_n}) \in \B_2.$ Assume that 
$$ \min\{C_{a_1}, ..., C_{a_n}, C_{b_1}, ..., C_{b_n}\} > \max\{a_1, b_1\}.$$
Then it holds 
$$\frac{1}{2} |\vec{a} - \vec{b}|_{\infty} \leq d_{bottle}(\B_1, \B_2),$$
where $\vec{a} = (a_1, ..., a_n)$ and $\vec{b} = (b_1, ..., b_n).$ The statement remains true if some of the $C_{a_i}$ or $C_{b_j}$ are equal to $+\infty.$ 
\end{lemma}

\begin{proof} Let $k$ be such that $|\vec{a} - \vec{b}|_{\infty}=|a_k - b_k|$ and assume without loss of generality that $a_k\leq b_k.$ Further assume that there exists a $\delta$-matching $\sigma: \B_1 \rightarrow \B_2.$ It is enough to prove that $2 \delta \geq b_k - a_k=|\vec{a} - \vec{b}|_{\infty}.$ We split the proof in three cases.

\medskip

\noindent {$\bullet$ Case 1 - One of the bars $[a_k,C_{a_k}), \dots , [a_n,C_{a_n}) $ is erased.}

\medskip

Denote by $l$ the index of the erased bar. Since $a_l\leq a_k$ and $C_{a_l} > b_1$ we have 
$$2 \delta \geq C_{a_l} - a_l \geq C_{a_l} - a_k > b_1 - a_k \geq b_k - a_k.$$
{$\bullet$ Case 2 - None of the bars $[a_k,C_{a_k}), \dots , [a_n,C_{a_n}) $ are erased, but at least one of them is matched with a bar different from $[b_k,C_{b_k}), \dots , [b_n,C_{b_n}).$}

\medskip 

Let $l$, where $k\leq l \leq n$, be such that $[a_l,C_{a_l})$ is not matched with any of the bars $[b_k,C_{b_k}), \dots , [b_n,C_{b_n})$ and let $\sigma ( [a_l,C_{a_l}) ) = [x,y).$ By the assumption of the theorem, we have that $x\geq b_k$ and hence
$$\delta\geq x - a_l \geq b_k - a_l \geq b_k - a_k.$$
{$\bullet$ Case 3 - Bars $[a_k,C_{a_k}), \dots , [a_n,C_{a_n})$ are all matched with bars $[b_k,C_{b_k}), \dots , [b_n,C_{b_n}).$}

\medskip

Let $l, ~ k\leq l \leq n$ be such that $\sigma ([a_l,C_{a_l}))=[b_k,C_{b_k}).$ We have
$$\delta \geq b_k - a_l \geq b_k - a_k,$$
and the proof is finished. 
\end{proof}

\subsection{Proof of Proposition \ref{emb-thm}} \label{ss-p-12}

We start with the definition of a bulked sphere.

\begin{dfn} \label{dfn-bsm} A {\it bulked sphere $S\subset \R^3$} is a surface of revolution obtained by rotating a {\it profile function} $r: [-L, L] \to [0, \infty)$ around axis $l$ as shown in Figure \ref{pf1}.

\begin{figure}[ht]
\includegraphics[scale=0.55]{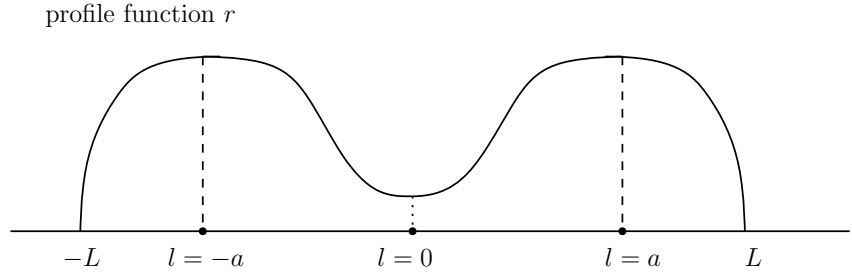}
\caption{Profile function of a bulked sphere $S$}
\label{pf1}
\end{figure}

\noindent We ask for $r$ to satisfy the following properties.  
\begin{itemize}
\item{} $r(l)$ is a smooth even function on $(-L, L)$ and $r(l) = 0$ exactly at $l = L \,\mbox{and}\, -L$.
\item{} $r(l)$ has only three critical points at $l = -a, 0, a$ and $r$ attains global maximum at $l=a, -a$ and local minimum at $l =0$.
\item{} $r''(0) > 0.$
\end{itemize}
Figure \ref{ex-bsm} shows a general picture of a bulked sphere. A {\it bulked sphere metric $g$} is a metric on $S^2$ induced from the standard metric on $\R^3.$ 
\begin{figure}[ht]
\includegraphics[scale=0.5]{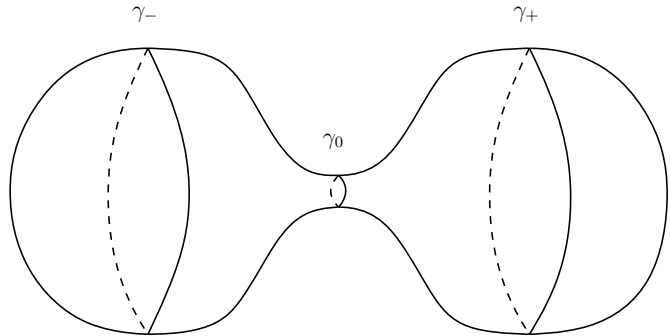}
\caption{A general picture of a bulked sphere}
\label{ex-bsm}
\end{figure}
\end{dfn}

A parallel circle is a geodesic if and only if it passes through a local extremum. In other words, we have three non-constant geodesic parallel circles of a bulked sphere metric, which we denote by $\gamma_-$, $\gamma_0$ and $\gamma_+$ as shown in Figure \ref{ex-bsm}.

\begin{lemma} \label{lm-short-geo}
For $m\in \N$, denote by $\gamma_0^m$ the $m$-times iteration of a closed geodesic $\gamma_0$ and by $\gamma_0^{-m}$ the $m$-times iteration of $\gamma_0$ in the opposite direction. For every $m\in \N$, $\gamma^{\pm m}_0$ are non-degenerate and ${\rm Ind}(\gamma^{\pm m}_0) = 0$.
\end{lemma}

The proof of Lemma \ref{lm-short-geo} comes from a direct computation which we carry out in Subsection \ref{Geo_index}. The following proposition is crucial for our proof of Proposition \ref{emb-thm} (lower bound). 

\begin{prop} \label{bsm-p}
Given any $0<\varepsilon_0<1$, there exists a positive $\delta_0<<1$ such that for every $x \in [0, \infty)$, there exists a bulked sphere metric $g_x \in {\mathcal G}_{S^2}$ satisfying the following properties.
\begin{itemize}
\item[(1)] Closed geodesic $\gamma_0$ has energy $E_{g_x}(\gamma_0) = \frac{\delta_0^2}{2} e^{-2x}$.
\item[(2)] Any closed geodesic $\gamma$ of $(S^2, g_x)$ different from $\gamma_0^{\pm m}$, $m\in \N$ has energy $E_{g_x}(\gamma) > \frac{\delta_0^2}{2}$. 
\item[(3)] There exists a constant $R_x \in \left[ \sqrt{\frac{1}{1+ \varepsilon_0}}, \sqrt{\frac{1}{1 - \varepsilon_0}} \right]$ such that $R_x \cdot g_x \in \bar{\mathcal G}_{S^2}$. 
\end{itemize}
\end{prop}
Part (2) of Proposition \ref{bsm-p} is proven in Subsection \ref{Geo_length}. Roughly speaking, it comes from a fact that every closed geodesic $\gamma$ different from $\gamma_0^{\pm m}$ has to exit the ``narrow neck'' and enter the two ``spherical regions'', i.e. regions where $l \notin [-a,a].$ By making these regions sufficiently large we get that the length of $\gamma$ must be large compared to the length of $\gamma_0.$ Finally, in order to prove (1) and (3) in Proposition \ref{bsm-p}, we need an explicit parametrization of $S$, see Subsection \ref{precise-par} in the Appendix. 

\begin{remark}\label{rmk-bumpy} Metrics $g_x$ in Proposition \ref{bsm-p} are not bumpy\footnote{They may be thought of as "bumpy below energy level $\frac{\delta_0^2}{2}$".} due to the existence of a rotational symmetry. However, they can be perturbed, by a $C^\infty$-small perturbation, to a bumpy metric which still satisfies all the properties from Lemma \ref{lm-short-geo} and Proposition \ref{bsm-p} (up to a small difference in logarithms of energies), see \cite{Anosov}. Since $C^\infty$-small perturbations create small differences in $d_{SBM},$ we ignore this point in the proof that follows, for the sake of clarity.
\end{remark}

We are now ready to give a proof of the lower bound in Proposition \ref{emb-thm}.

\begin{proof} (proof of Proposition \ref{emb-thm} (lower bound))
Define $\tilde{\Phi}: [0, \infty) \to {\mathcal G}_{S^2}$ as $\tilde{\Phi}(x)=g_x$ where $g_x$ is the metric given by Propostion \ref{bsm-p}.

Recall that
\[ \mathcal L_{pt}^{\lambda} (S^2,g_x) = \{ \gamma \in \mathcal L_{pt}(S^2) \,| \, E_{g_x}(\gamma) \leq \lambda\}. \]
and also that $\mathbb H_{*, pt}(S^2,g_x)$ denotes the persistence module given by $\mathbb H^\lambda_{*, pt}(S^2,g_x)=H_*(\mathcal L_{pt}^{\lambda} (S^2,g_x); \Z_2),$ comparison maps being induced by inclusions of sublevel sets. Our goal is to describe the barcode $\mathbb B(\mathbb H_{*, pt}(S^2,g_x)).$ By Proposition \ref{bsm-p} all closed geodesics of energy $\leq \frac{\delta_0^2}{2}$ are iterations of $\gamma_0$ and they are all non-degenerate. Thus, we may use Morse-Bott techniques described in Subsection \ref{ss-Mor-B}, namely the identity (\ref{Filtered_Morse-Bott}), see also Remark \ref{Below}.

As explained in Subsection \ref{ss-Mor-B}, constant geodesics will produce two generators $p_0\in CMB_{0,pt}(E_g,h)$ and $p_2\in CMB_{2,pt}(E_g,h)$ corresponding to two critical points of a height function on $S^2$. On the other hand, by Lemma \ref{lm-short-geo}, every $\gamma_0^{\pm m}$ satisfies ${\rm Ind}(\gamma^{\pm m}_0) = 0$ and hence every $\gamma_0^{\pm m}$ produces two generators $p_{\pm m}^0\in CMB_{0,pt}(E_g,h)$ and $p_{\pm m}^1 \in CMB_{1,pt}(E_g,h).$ These two generators correspond to minimum and maximum of a height function on $S^1$-critical submanifold $S^1 \cdot \gamma_0^{\pm m}.$

Furthermore $E_{g_x}(p_0)=E_{g_x}(p_2)=0$ while
\[ E_{g_x}(p_{\pm m}^0) = E_{g_x}(p_{\pm m}^1) = m E_{g_x}(\gamma_0) = m \cdot \frac{\delta^2_0}{2} e^{-2x}. \]
The boundary operator does not increase energy and thus we have that
$$\partial p^1_1=n(p^1_1,p_0)p_0+n(p^1_1,p^0_1)p^0_1+n(p^1_1,p^0_{-1})p^0_{-1},$$
where $n(p^1_1,p_0)$ equals the number of flow lines with cascades connecting $p^1_1$ to $p_0,$ and same for $n(p^1_1,p^0_1)$, $n(p^1_1,p^0_{-1})$, see Subsection \ref{ss-Mor-B}. Since $p^1_1$ and $p^0_1$ belong to the same $S^1$-critical submanifold we have that $n(p^1_1,p^0_1)=2 = 0~ \mod ~ 2.$ On the other hand, $p^1_1$ and $p^0_{-1}$ have the same energy, but belong to different $S^1$-critical submanifolds, which implies that there are no flow lines with cascades connecting them, i.e. $n(p^1_1,p^0_{-1})=0.$ Finally, as the global minimum, $p_0$ represents the homology class of a point which is not zero, i.e. $\partial p^1_1\neq p_0$ and we conclude that $\partial p^1_1=0.$ The same argument shows that $\partial p^1_{-1}=0.$ Thus, we may schematically present boundary relations with the following diagram.
\[ \begin{xymatrix}{
{\mbox{index $1$}} & & p^1_{-1},p_1^1 \ar[ld]|{NOT} \ar[d]|{NOT} & p^1_{-2},p_2^1& \,\,\,\ldots  \\
{\mbox{index $0$}} & p_0 & p^0_{-1},p_1^0 & p^0_{-2},p_2^0 & \ldots \\
{\mbox{energy}} & \lambda_0 = 0 & \lambda_1 = \frac{\delta^2_0}{2} e^{-2x} & \lambda_2 = \delta^2_0 e^{-2x} & \ldots.} \end{xymatrix} \]
Since $\gamma^{\pm m}_0$ do not produce any critical points of index 2, (2) in Proposition \ref{bsm-p} guarantees that $[p^1_{-1}], [p^1_1]\in H_1(\mathcal L_{pt}^{\lambda} (S^2,g_x); \Z_2)$ are non-zero for all $\lambda \leq \frac{\delta_0^2}{2}.$ In other words $\mathbb B(\mathbb H_{1, pt}(S^2,g_x))$ contains a bar $[E_{g_x}(\gamma_0),C_x)$ with $C_x \geq \frac{\delta_0^2}{2}$ (in fact it contains two such bars). Moreover $E_{g_x}(\gamma_0)$ is the smallest left endpoint in $\mathbb B(\mathbb H_{1, pt}(S^2,g_x)).$

Recall that $\B_{1,pt}(U_{g_x}^* S^2)$ denotes the barcode of a symplectic persistence module with logarithmic parametrization in degree one and homotopy class of a point. Theorem \ref{sh-loop-thm} implies that 
$$\left[ \ln \sqrt{2 E_{g_x}(\gamma_0)}, \ln \sqrt{2 C_x} \right) = \left[ \ln \delta_0 - x, \ln \sqrt{2 C_x} \right) \in \B_{1,pt}(U_{g_x}^* S^2).$$
By (2) in Proposition \ref{bsm-p} we also have that $\ln \sqrt{2 C_x}  \geq \ln \delta_0 - y$ for any $y\geq 0.$ Hence, for any $x,y\in [0,\infty)$ Lemma \ref{opt-matching} gives
$$\frac{1}{2} |x-y| \leq d_{bottle}(\B_{1,pt}(U_{g_x}^* S^2), \B_{1,pt}(U_{g_y}^* S^2)),$$
which together with Theorem \ref{TST} implies $\frac{1}{2} |x-y| \leq d_{SBM}(U_{g_x}^* S^2, U_{g_y}^* S^2)$.

Now define the desired embedding $\Phi: [0, \infty) \to \bar{\mathcal G}_{S^2}$ by 
\[ \Phi(x) = R_x \cdot \tilde{\Phi}(x) = R_x \cdot g_x ,\]
where $R_x$ is the rescaling factor given by (3) in Proposition \ref{bsm-p}. From Remark \ref{Scaling-SBM} it follows that 
\begin{align*}
    d_{SBM}(U_{\Phi(x)}^*S^2, U_{\Phi(y)}^* S^2) & = d_{SBM}(\sqrt{R_x} U_{g_x}^* S^2, \sqrt{R_y} U_{g_y}^* S^2)\\
    & = d_{SBM}(\sqrt{R_x/R_y} U_{g_x}^* S^2, U_{g_y}^* S^2) \\
    & \geq d_{SBM}(U_{g_x}^* S^2, U_{g_y}^* S^2) - d_{SBM}(U_{g_x}^*S^2, \sqrt{R_x/R_y} U_{g_x}^* S^2)\\
    & = d_{SBM}(U_{g_x}^* S^2, U_{g_y}^* S^2) - \frac{1}{2} |\ln R_x - \ln R_y|\\
    & \geq \frac{1}{2} |x-y| - \frac{1}{2} |\ln R_x - \ln R_y|.
\end{align*}
For any $\varepsilon>0$, take $\varepsilon_0 = \frac{e^{2\varepsilon} -1}{e^{2\varepsilon}+1}$ in Proposition \ref{bsm-p}. Then the range of $R_x$ given by (3) in Proposition \ref{bsm-p} implies $|\ln R_x - \ln R_y| \in [0, \varepsilon]$. Thus, we get the desired lower bound. 
\end{proof}

\subsection{Proof of Proposition \ref{emb-thm2}} \label{ss-p-13}
Let us give the definition of a multi-bulked surface first. Let $\Sigma$ be a closed, orientable surface of genus at least 1.

We call a subset of $\R^3$ {\it a cylindrical segment} if it can be obtained as an open surface of revolution with a constant profile function $r:(L_-,L_+)\rightarrow \R$ on some interval $(L_-,L_+).$

For $N\geq 1$ {\it an open chain of $N-1$ spheres}, denoted by $O(N),$ is an open surface of revolution with a smooth profile function $r:(L_-,L_+)\rightarrow \R$ which satisfies the following properties:
\begin{itemize}
\item{} $r(\ell)$ has $N$ local minima $a_1,\ldots , a_N$ and $N-1$ local maxima $b_1, \ldots , b_{N-1}.$
\item{} $r''(a_i) > 0$ for all $i=1, \ldots , N.$
\end{itemize}

Profile function of an open chain of $N-1$ spheres is illustrated in the Figure \ref{mprof} below.

\begin{figure}[ht]
\includegraphics[scale=0.5]{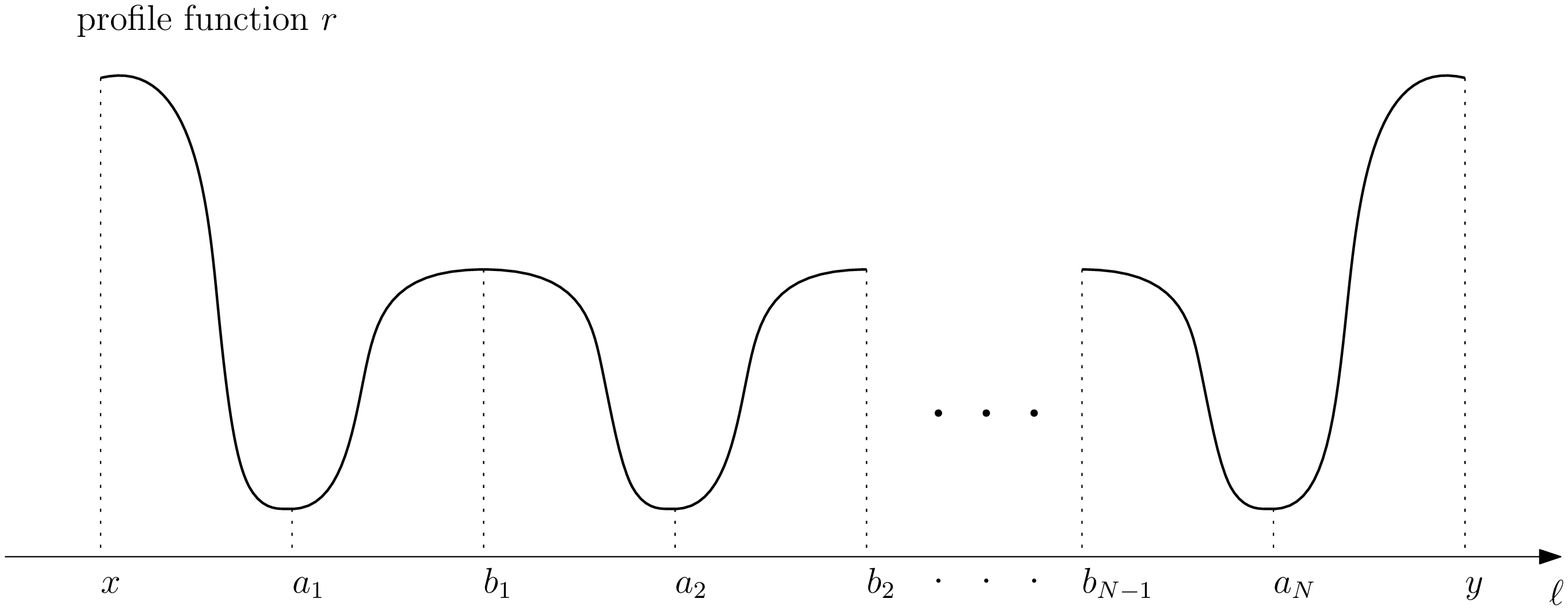}
\caption{Profile function of an open chain of $N-1$ spheres}
\label{mprof}
\end{figure}

\begin{dfn} \label{dfn-mbm}
Fix an embedding $\phi: \Sigma \rightarrow \R^3$ such that $\im \phi$ contains a cylindrical segment. A {\it multi-bulked surface} $S\subset \R^3$ is obtained by cutting out the cylindrical segment from $\im \phi$ and inserting $O(N).$ A general picture of a multi-bulked surface is shown in Figure \ref{ex-mb}. A {\it multi-bulked metric $g$} is a metric on $\Sigma$ induced by the standard metric on $\R^3.$  If we want to emphasize the role of $N$, we will also use terms {\it an $N$-bulked surface} and {\it an $N$-bulked metric}. 
\end{dfn}

\begin{figure}[ht]
\includegraphics[scale=0.5]{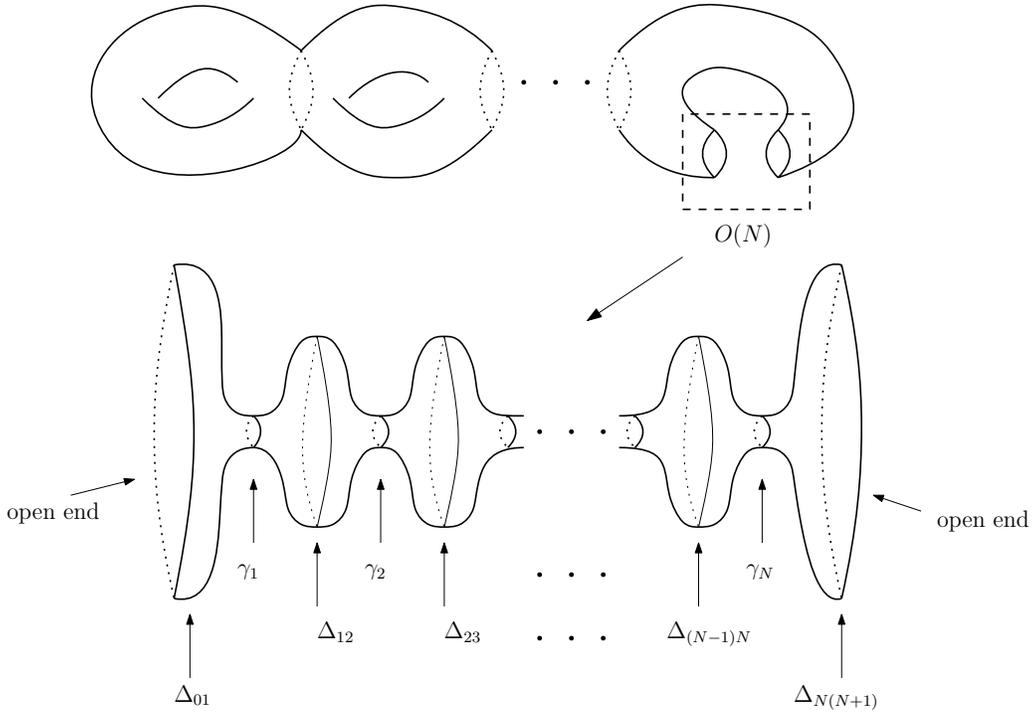}
\caption{A general picture of a multi-bulked surface}
\label{ex-mb}
\end{figure}

Denote the short simple closed geodesics coming from the ``narrow necks'' in $O(N)$ by $\gamma_1, .., \gamma_N$ from left to right as in Figure \ref{ex-mb}. All of $\gamma_i$ belong to the same free homotopy class, which we denote by $\alpha.$ Denote the long simple closed geodesics from the ``spherical parts'' in $O(N)$ by $\Delta_{12}, ..., \Delta_{(N-1) N}$ from left to right and the boundary curves of $O(N)$ by $\Delta_{01}$ and $\Delta_{N (N+1)}$ as in Figure \ref{ex-mb}. Moreover, we put the following requirements. 
\begin{itemize}
    \item{} Lengths of $\Delta_{i (i+1)}$ satisfy $L_{g}(\Delta_{12}) = ... = L_{g}(\Delta_{(N-1)N})$.
    \item{} Energies of $\gamma_i$ satisfy $E_g(\gamma_1) \leq ... \leq E_g(\gamma_N)$. 
\end{itemize}

For any $N \in \N$, let
\[ \mathcal T(N) = \left\{\vec{x} = (x_1, ..., x_{N}) \in [0, \infty)^{N}\,| \, x_1 \leq  x_2 \leq ... \leq x_{N} \right\}.\]

Similarly to Lemma \ref{lm-short-geo}, we have the following result. 
\begin{lemma} \label{lm-short-geo-2}
Each closed geodesic $\gamma_i$ is non-degenerate and ${\rm Ind}(\gamma_i) = 0$ for $i \in \{1, ..., N\}$.
\end{lemma}

Lemma \ref{lm-short-geo-2} is proven in Subsection \ref{Geo_index}. Similarly to Proposition \ref{bsm-p}, we have the following result.

\begin{prop} \label{mb-p} Let $\Sigma$ be a closed, orientable surface of genus at least 1. For any $N \in \N$ and $0<\varepsilon_0<1$, there exists a positive $\delta_0 <<1$ such that for any $\vec{x} = (x_1, ..., x_{N}) \in \mathcal T(N)$, there exists an $N$-bulked metric $g_{\vec{x}} \in {\mathcal G}_{\Sigma}$ satisfying the following properties. 
\begin{itemize}
    \item[(1)] Each closed geodesic $\gamma_i$ has energy $E_{g_{\vec{x}}}(\gamma_i) = \frac{\delta_0^2}{2} e^{-2x_i}$ for $i \in \{1, ..., N\}$.
    \item[(2)] Any closed geodesic $\gamma$ on $(\Sigma, g_{\vec{x}})$ different from $\gamma_1, ..., \gamma_{N}$ and their iterates has energy $E_{g_{\vec{x}}}(\gamma) > \frac{\delta_0^2}{2}$.
    \item[(3)] Every cylinder connecting $\gamma_i$ and $\gamma_j$ for $i \neq j$ must pass through a loop with energy greater than $\frac{\delta_0^2}{2}$. 
    \item[(4)] There exists some constant $R_{\vec{x}} \in \left[ \sqrt{\frac{1}{1+ \varepsilon_0}}, \sqrt{\frac{1}{1 - \varepsilon_0}} \right]$ such that $R_{\vec{x}} \cdot g_{\vec{x}} \in \bar{\mathcal G}_{\Sigma}$.
\end{itemize}
\end{prop}
Properties (1), (2) and (3) in Proposition \ref{mb-p} can be confirmed by the same argument as (1), (2) and (4) in Proposition \ref{bsm-p}. Property (3) in Proposition \ref{mb-p} essentially comes from the fact that curves $\Delta_{i (i+1)}$ are very long compared to $\gamma_j.$ 

\medskip

The quasi-isometric embedding of $(\R^N, |\cdot|_{\infty})$ into $\bar{\mathcal G}_{\Sigma}$ which we construct to prove Proposition \ref{emb-thm2} will be realized as a composition of two quasi-isometric embeddings according to the following scheme 
$$(\R^N, |\cdot|_{\infty}) \xrightarrow{Q} (\mathcal T(2N), |\cdot|_{\infty}) \xrightarrow{\Psi} \bar{\mathcal G}_{\Sigma}.$$ To this end, in Subsection \ref{red-emb-space} we prove the following lemma. 

\begin{lemma} \label{red-par}
Fix $N \in \N$. There exists a map $Q: (\R^N, |\cdot|_{\infty}) \to (\mathcal T(2N), |\cdot|_{\infty})$ such that for any $\vec{x}, \vec{y} \in (\R^N, |\cdot|_{\infty})$, 
\[ \frac{1}{4} |\vec{x} - \vec{y}|_{\infty} \leq |Q(\vec{x}) - Q(\vec{y})|_{\infty} \leq  (2N) \cdot |\vec{x} - \vec{y}|_{\infty}. \]
\end{lemma}

\begin{remark} Similarly to metrics $g_x$ in Proposition \ref{bsm-p}, metrics $g_{\vec{x}}$ in Proposition \ref{mb-p} may not be bumpy. As exaplained in Remark \ref{rmk-bumpy}, they can be perturbed by a $d_{SBM}$-small perturbation to a bumpy metric which still satisfies all the properties from Lemma \ref{lm-short-geo-2} and Proposition \ref{mb-p} (up to a small difference in logarithms of energies). Again, we ignore this point in the proof that follows, for the sake of clarity.
\end{remark}

We are now in a position to give a proof of the lower bound in Proposition \ref{emb-thm2}.

\begin{proof} (Proof of Proposition \ref{emb-thm2} (lower bound)) Define a map $\tilde{\Psi}: \mathcal T(2N) \to {\mathcal G}_\Sigma$ as $\tilde{\Psi}(\vec{x}) = g_{\vec{x}},$
where $\vec{x} = (x_1, ..., x_{2N}) \in \mathcal T(2N)$ and $g_{\vec{x}}$ is a $2N$-bulked metric given by Proposition \ref{mb-p}. The short geodesics in the $O(2N)$ part are labelled from the longest to the shortest by $\gamma_1,\ldots , \gamma_{2N}$. Denote the homotopy class of these geodesics by $\alpha = [\gamma_1] = ...= [ \gamma_{2N}]$.

For this $\alpha$, reversed loops $\gamma_i^{-1}$ as well as iterations $\gamma_i^{\pm m}$ for $m\geq 2$ are all not in $\alpha.$ Constant loops are also not in $\alpha$ and thus (3) in Proposition \ref{mb-p} implies that the only closed geodesics in class $\alpha$ with energy less or equal to $\frac{\delta_0^2}{2}$ are $\gamma_i$, $i=1,\ldots , 2N.$ Lemma \ref{lm-short-geo-2} guarantees that all $\gamma_i$ are non-degenerate and thus we may use Morse-Bott techniques introduced in Subsection \ref{ss-Mor-B}, namely the identity (\ref{Filtered_Morse-Bott}), see also Remark \ref{Below}.

As in the proof of Proposition \ref{emb-thm} each $\gamma_i,$ $i=1,\ldots , 2N$ produces two generators of the Morse-Bott chain complex, $p^0_i\in CMB_{0,\alpha}(E_{g_{\vec{x}}},h)$ and $p^1_i\in CMB_{1,\alpha}(E_{g_{\vec{x}}},h).$ Moreover these these are the only generators of $CMB^\lambda_{*,\alpha}(E_{g_{\vec{x}}},h)$ for $\lambda \leq \frac{\delta_0}{2}.$

In terms of the boundary operator we have that for all $i=1,\ldots , 2N$ it holds $\partial p^0_i=0$ as well as $n(p^1_i, p^0_i)=0$ because $p^1_i$ and $p^0_i$ belong to the same $S^1$-critical submanifold. We claim that also $n(p^1_i, p^0_j)=0$ when $i\neq j.$ Indeed, assume that there exists a flow line with cascades $(u_1, \ldots, u_k, t_1, \ldots , t_{k-1})$ connecting $p^1_i$ and $p^0_j.$ Since $\gamma_i \neq \gamma_j,$ we must have $k\geq 1$ and one of the flow lines $u_l$ would have to start at a critical submanifold $S^1 \cdot \gamma_{i_1}$ and end at a critical submanifold $S^1 \cdot \gamma_{i_2}$ with $i_1 \neq i_2.$ However, this would mean that $\im u_l \subset \Sigma$ defines a cylinder which connects $\gamma_{i_1}$
and $\gamma_{i_2}$ and which passes only through loops of energy no greater than $\lambda \leq \frac{\delta_0^2}{2}.$ Existence of such a cylinder is ruled out by (3) in Proposition \ref{mb-p} and hence $n(p^1_i, p^0_j)=0$ for all $i,j.$ This means that for $\lambda \leq \frac{\delta_0^2}{2},$ $\partial = 0$ on $CBM^\lambda_{*,\alpha}(E_{g_{\vec{x}}},h).$

Using (\ref{Filtered_Morse-Bott}) we conclude that $\mathbb B(\mathbb H_{1, \alpha}(\Sigma, g_{\vec{x}}))$ contains bars $[E_{g_{\vec{x}}}(\gamma_i), C_i(\vec{x}))$ for $i=1,\ldots , 2N$, with $C_i(\vec{x}) \geq \frac{\delta_0^2}{2}$ (possibly $C_i(\vec{x}) = \infty$). Moreover, $E_{g_{\vec{x}}}(\gamma_i)$ are the $2N$ smallest left endpoints of bars in $\mathbb B(\mathbb H_{1, \alpha}(\Sigma, g_{\vec{x}})).$

Recall that $\B_{1,\alpha}(U_{g_{\vec x}}^* \Sigma)$ denotes the barcode of a symplectic persistence module with logarithmic parametrization in degree one and homotopy class $\alpha$. Theorem \ref{sh-loop-thm} implies that, for any $i \in \{1, ..., 2N\}$,  
\[ \left[ \ln\sqrt{2E_{g_{\vec{x}}}(\gamma_i)}, \ln \sqrt{2C_i(\vec{x})} \right) = \left[ \ln \delta_0 - x_i, \ln \sqrt{2C_i(\vec{x})} \right) \in \B_{1, \alpha}(U_{g_{\vec{x}}^*} \Sigma). \]
Similar conclusion holds for any $\vec{y} \in \mathcal T(2N)$. Moreover, $\ln \sqrt{2C_i(\vec{x})} \geq \ln \delta_0 - y_j$ for any $y_j\in [0,\infty)$. Hence, Lemma \ref{opt-matching} implies 
\[ \frac{1}{2} |\vec{x} - \vec{y}|_{\infty} \leq d_{bottle}(\B_{1, \alpha}(U_{g_{\vec{x}}}^* \Sigma), \B_{1, \alpha}(U_{g_{\vec{y}}}^* \Sigma)).\]
Theorem \ref{TST} then yields $\frac{1}{2}|\vec{x} - \vec{y}|_{\infty} \leq d_{SBM}(U_{g_{\vec{x}}}^* \Sigma, U_{g_{\vec{y}}}^* \Sigma)$.

Now Lemma \ref{red-par} provides an embedding $\tilde \Phi: = \tilde{\Psi} \circ Q: \R^N \to \mathcal G_{\Sigma}$, which satisfies
$$\frac{1}{8} |\vec{x}-\vec{y}|_{\infty} \leq \frac{1}{2}|{Q}(\vec{x}) - {Q}(\vec{y})|_{\infty} \leq d_{SBM}(U_{\tilde{\Phi}(\vec{x})}^* \Sigma, U_{\tilde{\Phi}(\vec{y})}^* \Sigma),$$
for any $\vec{x}, \vec{y} \in \R^N.$

Finally, $\Phi: \R^N \to \bar{\mathcal G}_{\Sigma}$ is defined by setting $\Phi(\vec{x}) = R_{Q(\vec{x})} \cdot \tilde{\Phi}(\vec{x}) =  R_{Q(\vec{x})} \cdot g_{Q(\vec{x})},$ where $R_{Q(\vec{x})}$ is the rescaling factor given by (4) in Proposition \ref{mb-p}, associated to vector $Q(\vec{x}) \in \mathcal T(2N)$. The same argument as in the proof of Proposition \ref{emb-thm} (lower bound) implies 
\[ \frac{1}{8} |\vec{x} - \vec{y}|_{\infty} - \frac{1}{2}| \ln R_{Q(\vec{x})} - \ln R_{Q(\vec{y})}| \leq d_{SBM}(U_{{\Phi}(\vec{x})}^* \Sigma, U_{{\Phi}(\vec{y})}^* \Sigma). \]
For any $ \varepsilon>0$, take $\varepsilon_0 = \frac{e^{2\varepsilon} -1}{e^{2\varepsilon}+1}$ in Proposition \ref{mb-p}. Then (4) in Proposition \ref{mb-p} implies the desired lower bound. 
\end{proof}

\section{Bulked sphere and multi-bulked surface} \label{sec-bulk}

\subsection{Analyzing short geodesics}\label{Geo_index}
The goal of this subsection is to prove Lemmas \ref{lm-short-geo} and \ref{lm-short-geo-2}. All considerations in this subsection are local and hence apply equally to both propositions. Let us focus on $\gamma_0$ on a bulked sphere $S.$ 

\begin{lemma} \label{hyper} The geodesic $\gamma_0$ on a bulked sphere is hyperbolic. 
\end{lemma}

We start with some necessary background. Let $(M,g)$ be an $n$-dimensional Riemannian manifold. Recall that a vector field $J$ along the geodesic path $\gamma:[0,1]\rightarrow M$ is called {\it Jacobi field} if it satisfies the Jacobi equation
\begin{equation}\label{jac}
\nabla_{\dot{\gamma}} \nabla_{\dot{\gamma}} J + R(J,\dot{\gamma})\dot{\gamma}=0,    
\end{equation}
where $R(\cdot,\cdot)$ stands for the curvature tensor associated to $g.$ Jacobi fields are tangent to the space of geodesic paths with free endpoints. When $\gamma$ is a closed geodesic, they can be used to calculate index and nullity of $\gamma.$ To this end, first notice that Jacobi field is uniquely determined by two initial conditions $J(0)$ and $\nabla_{\dot{\gamma}}J(0).$ Moreover, we may choose these two vectors freely, which means that the space of Jacobi fields is $2n$-dimensional. The two initial conditions $J_0(0)=\dot{\gamma}(0)$, $ \nabla_{\dot{\gamma}}J_0(0)= 0$ and $\bar{J}_0(0)=0$, $ \nabla_{\dot{\gamma}}\bar{J}_0(0)=\dot{\gamma}(0)$ yield Jacobi fields $J_0(t)=\dot{\gamma}(t)$ and $\bar{J}_0(t)=t\dot{\gamma}(t)$ which are tangent to $\gamma.$ Let 
$$E(t)=(T\gamma(t))^{\perp}\oplus(T\gamma(t))^{\perp}\subset T_{\gamma(t)}M\oplus T_{\gamma(t)}M$$
be the $(2n-2)$-dimensional vector bundle along $\gamma,$ where $(T\gamma(t))^{\perp}$ denotes the orthogonal space to $\dot{\gamma}(t)$ inside $T_{\gamma(t)}M.$ It is easy to check that if $J(0)\perp \dot{\gamma}(0)$ and $\nabla_{\dot{\gamma}(0)}J(0)\perp \dot{\gamma}(0)$ then $J(t)\perp \dot{\gamma}(t)$ and $\nabla_{\dot{\gamma}(t)}J(t)\perp \dot{\gamma}(t)$ for all $t\in [0,1].$ This means that we may define a family of maps
\[ P(t):E(0)\rightarrow E(t) \]
by $P(t)(v,w)= \left(J(t), \nabla_{{\dot\gamma}} J(t) \right)$ where $J$ is the Jacobi field with the initial condition $\left(J(0),\nabla_{{\dot\gamma}} J(0)\right)= (v, w)$. In particula if $\gamma$ is closed, i.e. $\gamma(t+1)=\gamma(t),$ we have that
\[ P(1):E(0)\rightarrow E(0)\]
and this map is called the {\it linearized Poincare map}.

\begin{dfn} Closed geodesic $\gamma$ is called {\it hyperbolic} if no eigenvalue of the linearized Poincare map has norm equal to 1. \end{dfn}

Taking advantage of the geometry of a bulked sphere, the proof of Lemma \ref{hyper} comes from a direct computation of eigenvalues of the linearized Poincare map.

\begin{proof} (Proof of Lemma \ref{hyper})
Suppose that our bulked sphere $S$ comes from rotating a profile function $r$ around the $x$-axis and denote the radius of the circle $\gamma_0(t)$ by $r(0): = r.$ Then
\[ \gamma_0(t) = (0, r\cos(2\pi t), r\sin(2 \pi t)), ~t\in [0,1] \]
and its velocity is given by
\[ {\dot{\gamma}}_0(t) = (0, - 2 \pi r \sin(2 \pi t), 2 \pi r\cos(2 \pi t)). \]
Gaussian curvature $K_G$ along $\gamma_0(t)$ is constant and can be expressed using the formula for the Gaussian curvature of the surface of revolution. More precisely, we have
$$K_G=- \frac{r''(0)}{r},$$
which is negative by the third property in the definition of a bulked sphere, namely $r''(0)>0,$ see Definiton \ref{dfn-bsm}.

In order to calculate the linearized Poincare map, we are only interested in the Jacobi fields orthogonal to $\dot{\gamma}_0(t)$. Let $J = J(t)$ be such a Jacobi field, $J(t) \perp {\dot\gamma_0}(t)$ for all $t\in [0,1].$ Since $\dim S =2,$ $J(t)$ and ${\dot \gamma}_0(t)$ span the tangent planes $T_{\gamma(t)}S$. On the other hand, the curvature tensor satisfies $\left< R(J, \dot{\gamma}_0)\dot{\gamma}_0, {\dot\gamma}_0\right> =0$ and hence $R(J(t), \dot{\gamma}_0(t))\dot{\gamma}_0(t)$ is proportional to $J(t).$ We calculate
\begin{align*}
\left< R(J(t), \dot{\gamma}_0(t))\dot{\gamma}_0(t), J(t) \right> & = |\dot{\gamma}_0(t)|^2 |J(t)|^2 \left<R(e_2, e_1)e_1, e_2 \right> \,\,\,\,\mbox{with $\{e_1,e_2\}$ orthonormal}\\
& = |\dot{\gamma}_0(t)|^2 |J(t)|^2 K_G \\
& = (2 \pi r)^2 |J(t)|^2 \cdot \frac{- r''(0)}{r} \\
& = - 4 \pi^2 r r''(0) |J(t)|^2 \\
& = - 4 \pi^2 r r''(0) \left< J(t), J(t) \right>.
\end{align*}
Denoting $K = - 4 \pi^2 r r''(0)$, (\ref{jac}) is simplified as
\begin{equation} \label{jac2}
\nabla_{\dot{\gamma}_0} \nabla_{\dot{\gamma}_0} J + K \cdot J = 0.
\end{equation}
Note that $K$ is always negative because $r''(0)>0.$

Now, since $S$ is a surface of revolution with axis of rotation being the $x$-axis, and since $r'(0)=0,$ the tangent space to $S$ at $\gamma_0(t)$ is generated by $\dot{\gamma}_0(t)$ and $(1,0,0).$ This means that a Jacobi field orthogonal to $\dot{\gamma}_0$ has the form $J(t) = (J_1(t), 0, 0).$ It follows that
\begin{equation}\label{first-cd}
 \nabla_{\dot{\gamma}_0} J(t) = ({\dot J}_1(t), 0, 0),  
\end{equation}
as well as
\begin{equation}\label{second-cd}
\nabla_{\dot{\gamma}_0} \nabla_{\dot{\gamma}_0} J(t) = ({\ddot J}_1(t), 0, 0),  
\end{equation}
and (\ref{jac2}) becomes a second order equations
\begin{equation} \label{jac3}
{\ddot J}_1(t) + K \cdot J_1(t) = 0.
\end{equation}
The two solutions of this equation are vector fields $J_{+}(t)=(e^{\sqrt{-K}t},0,0)$ and $J_{-}(t)=(e^{-\sqrt{-K}t},0,0).$ Moreover, initial vectors 
$$(J_{+}(0), (\nabla_{\dot{\gamma}_0} J_{+})(0)) = ((1,0,0), (\sqrt{-K},0,0))$$
and
$$(J_{-}(0), (\nabla_{\dot{\gamma}_0} J_{-})(0)) = ((1,0,0), (-\sqrt{-K},0,0))$$
are linearly independent and hence generate  $E(0)=(T\gamma_0(0))^{\perp}\oplus(T\gamma_0(0))^{\perp}.$ In order to compute the eigenvalues of the linearized Poincare map $P:E(0)\rightarrow E(1)$ it is enough to notice that from (\ref{first-cd}) we have
$$(J_{+}(1), (\nabla_{\dot{\gamma}_0} J_{+})(1)) = ((e^{\sqrt{-K}},0,0), (\sqrt{-K}e^{\sqrt{-K}},0,0))$$
as well as
$$(J_{-}(1), (\nabla_{\dot{\gamma}_0} J_{-})(1)) = ((e^{-\sqrt{-K}},0,0), (-\sqrt{-K}e^{-\sqrt{-K}},0,0)).$$
Thus $((1,0,0), (\sqrt{-K},0,0))$ and $((1,0,0), (-\sqrt{-K},0,0))$ are eigenvectors of $P$ with eigenvalues  $\lambda_1 = e^{\sqrt{-K}}$ and $\lambda_2 = e^{-\sqrt{-K}}.$ Since $K \neq 0$, neither one of these has norm one, which means that $\gamma_0$ is hyperbolic by definition.\end{proof} 

Recall that a closed geodesics is non-degenerate if its nullity is zero. The following lemma is a direct consequence of the second variation formula, see, for example, Corollary 2.5.6 in \cite{Kli95}.

\begin{lemma}\label{Lemma_Nulity} Nullity of a closed geodesic $\gamma$ is equal to the dimension of the space of periodic Jacobi fields along $\gamma$ minus one. In particular, $\gamma$ is non-degenerate, that is nullity of $\gamma$ is $0$, if and only if there are no periodic Jacobi fields along $\gamma$ which are orthogonal to $\dot{\gamma}$.
\end{lemma}

Note that ``minus one'' in Lemma \ref{Lemma_Nulity} comes from the need to exclude the tangent Jacobi field $J_0(t)=\dot{\gamma}(t).$

When a closed geodesic is hyperbolic, its index as well as the indices of all its iterations are particularly easy to compute. Let us recall some related formulas. When $\gamma$ is hyperbolic, we have a splitting
\[ E = E_s \oplus E_u \]
such that $P(t)|_{E_s}$ is contracting and $P(t)|_{E_u}$ is expanding as $t$ goes from $0$ to $1$. Now, for each $t_* \in [0,1]$, define a number $\iota(t_*)$ to be the dimension of the subspace of Jacobi fields $J(t)$ along $\gamma$ such that $(J(t), (\nabla_{{\dot\gamma}} J(t))) \in E_s(t)$, for all $t\in [0,1]$ and $J(t_*) = 0.$ The number of points $t_*$ for which $\iota(t_*)>0$ is finite and the following holds.

\begin{lemma}[Proposition 5, page 4 of \cite{Kli74}]\label{Lemma_Index}
If $\gamma$ is hyperbolic, then 
\[{\rm Ind}(\gamma) = \sum_{t_* \in [0,1)} \iota(t_*) .\]
\end{lemma}

One may regard this lemma as an analogue of the well-known Morse index theorem for a geodesic segment. A general result about the index of a closed geodesic (not necessarily hyperbolic) is worked out in \cite{Kli73}. Finally, 

\begin{lemma}[Corollary 3.2.15 in \cite{Kli95}]\label{Lemma_Index_Iterations}
For a hyperbolic $\gamma$ it holds:
\[ {\rm Ind}(\gamma^m) = m \cdot {\rm Ind}(\gamma). \]
\end{lemma}

We are now ready to give the desired proofs. 
\begin{proof}(Proof of Lemmas \ref{lm-short-geo} and \ref{lm-short-geo-2}) From the computations in the proof of Lemma \ref{hyper}, we know that the space of Jacobi fields orthogonal to $\dot{\gamma}_0$ is generated by the fields $J_{+}(t)=(e^{\sqrt{-K}t},0,0)$ and $J_{-}(t)=(e^{-\sqrt{-K}t},0,0).$ Since $e^{\sqrt{-K}t}\rightarrow +\infty$ and $e^{-\sqrt{-K}t}\rightarrow 0$ when $t\rightarrow + \infty,$ no linear combination of $J_{+}$ and $J_{-}$ can be periodic. Thus, by Lemma \ref{Lemma_Nulity}, we know that $\gamma_0^m$ are non-degenerate for all $m\in \N$.

On the other hand, by (\ref{first-cd}) we have that for $t\in [0,1]$
$$(J_{+}(t), (\nabla_{\dot{\gamma}_0} J_{+})(t)) = ((e^{\sqrt{-K}t},0,0), (\sqrt{-K}e^{\sqrt{-K}t},0,0))$$
as well as
$$(J_{-}(t), (\nabla_{\dot{\gamma}_0} J_{-})(t)) = ((e^{-\sqrt{-K}t},0,0), (-\sqrt{-K}e^{-\sqrt{-K}t},0,0)).$$
In other words contracting and expanding spaces in the splitting $E(t)= E_s(t) \oplus E_u(t)$ are generated by $(J_{-}(t), (\nabla_{\dot{\gamma}_0} J_{-})(t))$ and $(J_{+}(t), (\nabla_{\dot{\gamma}_0} J_{+})(t))$ respectively. Since for all $t\in [0,1]$ it holds $J_{-}(t)\neq 0,$ Lemma \ref{Lemma_Index} implies that ${\rm Ind}(\gamma_0) = 0$ and thus by Lemma \ref{Lemma_Index_Iterations} ${\rm Ind}(\gamma_0^m) = m \cdot {\rm Ind}(\gamma_0) = 0.$ Finally, note that the direction of $\gamma_0$ played no role in this subsection, i.e. all statements apply equally to $\gamma_0^{-m}.$ This completes the proof of Lemma \ref{lm-short-geo}. Since all the considerations are local, the proof of Lemma \ref{lm-short-geo-2} follows in the same fashion. \end{proof}
 
\begin{remark}
One may also prove that ${\rm Ind}(\gamma_0^m)=0$ by a direct computation using Lemma \ref{Lemma_Index}, without realying on Lemma \ref{Lemma_Index_Iterations}.
\end{remark} 
 
\subsection{Analyzing long geodesics}\label{Geo_length} In this subsection, we will prove (2) in Proposition \ref{bsm-p} as well as (2) and (3) in Proposition \ref{mb-p}. To this end, let us describe closed geodesics on a bulked sphere and a multi-bulked surface. We start with a bulked sphere.

Assume that our bulked sphere $S\subset \R^3$ is obtained by rotating a profile function $r$ around an axis $l$ as described in Subsection \ref{ss-p-12}. For a point $p\in S$ we denote by $l(p)$ he coordinate of $p$ on the $l$-axis and by $r(p)$ the value of the profile funtction at $l(p),$ i.e. $r(p)$ is the distance from $p$ to the axis $l.$

Firstly, we notice that parallel circles given by $l=const$ are geodesics if and only if $r'(l)=0.$ This means that $\gamma_-,$ $\gamma_0$ and $\gamma_+$ are the only geodesic parallel circles. In order to describe geodesics which are not parrallel circles we evoke the well-known Clairaut's relation, see, for example, Proposition 4.4 in \cite{Shi}.

\begin{theorem} \label{thm-cla} Suppose that $S(r)$ is a surface of revolution obtained by rotating a profile function $r$ around a fixed axis. Then any geodesic on $S(r)$ satisfies the equation 
\begin{equation} \label{clairaut}
r \cos(\phi) = \mbox{constant}
\end{equation}
where $\phi$ is the angle between the geodesic and the parallel circles. Conversely, any constant speed curve satisfying (\ref{clairaut}) which is not a parallel circle is a geodesic. \end{theorem} 

Using notations from Subsection \ref{ss-p-12}, we call the part of the bulked sphere $S$ where $l\in (-a,a)$ {\it the neck} of $S$ and the part where $|l|\geq a$ the {\it spherical regions} of $S.$ The next lemma claims that a closed geodesic different from $\gamma_0^{\pm m}, m\in \N,$ can not be entirely contained in the neck.

\begin{lemma}\label{geodesics1}
Assume that $\gamma:\R / \Z \rightarrow S$ is a closed geodesics different form $\gamma_0^{\pm m}, m\in \N$ such that $l(\gamma_{t_0})\in (-a,a)$ for some $t_0\in \R / \Z.$ Then $\gamma$ intersects either $\gamma_-$ or $\gamma_+.$ 

\end{lemma}
\begin{proof}
Since $\gamma \neq \gamma_0^{\pm m},$ we have that $\gamma$ is not a parallel circle and thus for some $T\in \R / \Z,$ $\dot{\gamma}(T)$ is transverse to the parallel circle $P=\{ p\in S ~|~ l(p)=a_0 \}$ for some $0\leq a_0 <a.$ We may assume that $T=0$ as well as that $\dot{\gamma}(0)$ points away from $\gamma_0$ and towards $\gamma_+.$ We can make this assumption because if $\dot{\gamma}(0)$ points towards $\gamma_0$ we may look at $\gamma^{-1}(t)=\gamma(-t)$ which defines the same curve as $\gamma$ only with reversed direction. We may also assume that the angle $\phi(\gamma(0))$ between the parallel circle $P$ and $\dot{\gamma}(0)$ satisfies $\phi(\gamma(0))\in (0,\frac{\pi}{2}),$ see Figure \ref{geo-bulk}, the case $\phi(\gamma(0))\in (\frac{\pi}{2},\pi)$ is treated in the same manner.

\begin{figure}[ht]
\begin{center}
\includegraphics[scale=0.4]{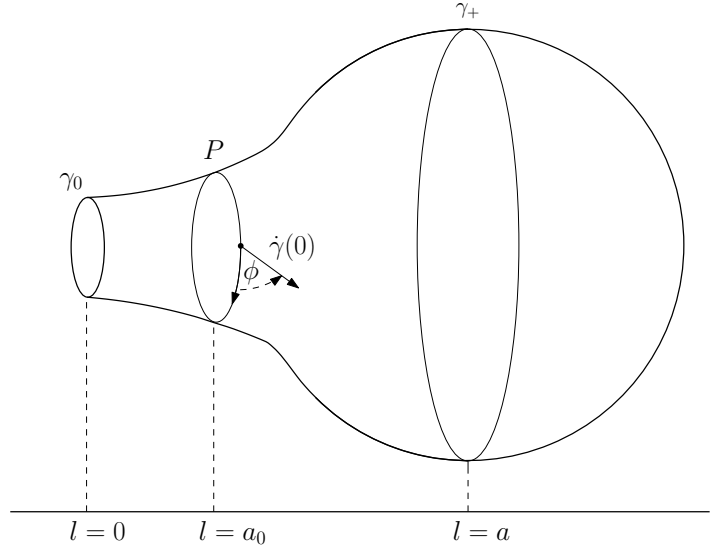}
\caption{Geodesic $\gamma$ intersecting parallel circle $P$}\label{geo-bulk}
\end{center}
\end{figure}

Clairaut's relation implies that $r(\gamma(t)) \cos(\phi(\gamma(t)))=C_0>0.$ For a small $\varepsilon>0$ it holds $l(\gamma(-\varepsilon))<a_0$, $l(\gamma(\varepsilon))>a_0$ and since $\gamma$ is closed, we have that for some $\tau>0,$ $l(\gamma(\tau))<a_0.$ This means that $\gamma$ eventually exits the region $\{ l>a \}$ and hence it must intersect $P$ with a negative angle $-\phi (\gamma(0))$. Formally, there exists $\tau_1>0$ such that $\gamma(\tau_1)\in P$ and $\phi (\gamma(\tau_1))=-\phi (\gamma(0))<0.$ It follows that there exists $0<\tau_0 <\tau_1$ such that $\phi(\tau_0)=0,$ and Clairaut's relation implies
$$C_0=r(\gamma(0)) \cos (\phi(\gamma(0)))=r(\gamma(\tau_0)) \cos (\phi(\gamma(\tau_0)))=r(\gamma(\tau_0)).$$
Thus $r(\gamma(\tau_0))<r(\gamma(0))$ and since $r$ increases on the interval $[0,a]$ we have that $l(\gamma(\tau_0))>a,$ which proves the claim.
\end{proof}

Now, if $S$ is a multi-bulked surface defined in Subsection \ref{ss-p-13}, we call the part of $S$ between $\Delta_{(i-1)i}$ and $\Delta_{i(i+1)}$ {\it the neck of} $\gamma_i.$ Using the same argument as in the proof of Lemma \ref{geodesics1} we may prove the following.

\begin{lemma}\label{geodesics2}
Let $\gamma$ be a closed geodesic on a multi-bulked surface $S$ which enters the neck of $\gamma_i.$ Then $\gamma$ intersects either $\Delta_{(i-1)i}$ or $\Delta_{i(i+1)}.$
\end{lemma}

\begin{remark}
One may also deduce Lemmas \ref{geodesics1} and \ref{geodesics2} form the analysis of the geodesic flow similar to the one presented in Subsection \ref{geodesic-flow}.
\end{remark} 

We are now in a position to give a proof of (2) in Proposition \ref{bsm-p} as well as (2) and (3) in Proposition \ref{mb-p}. However, before we proceed with the arguments, we wish to explain the general logic which these proofs follow.

In the case of a multi-bulked surface, firstly we fix the genus of the surface and the number of necks $N$ (in the case of the bulked sphere these are automatically fixed).

Secondly, we fix an embedding $\phi:\Sigma\rightarrow \R^3,$ of the surface and a cylindrical segment inside $\im \phi$ which we wish to replace by an open chain of $N-1$ spheres $O(N)$ as described in Subsection \ref{ss-p-13}. After inserting $O(N)$ we obtain the multi-bulked surface $S\subset \R^3.$ 

Finally the major part of $S$ remains fixed as we vary $g_{\vec{x}},$ for $\vec{x} \in \mathcal T(N)$ (or $g_x$, for $x\in [0,\infty)$ in the bulked sphere case). In fact, for different $\vec{x} \in \mathcal T(N),$ metrics $g_{\vec{x}}$ only differ in very small neighbourhoods of $\gamma_1,\ldots , \gamma_N$ (or in a small neighbourhood of $\gamma_0$ in the bulked sphere case). Moreover, we have the freedom to define $g_{\vec{x}}$ in these neighbourhoods in such a way that the energies of $\gamma_1,\ldots , \gamma_N$ (or the energy of $\gamma_0$) are equal to any sufficiently small numbers, see Subsection \ref{precise-par}.

Now, proving the existence of $\delta_0$ as in (2) in Proposition \ref{bsm-p} and (2), (3) in Proposition \ref{mb-p} actually means providing $\delta_0$ which only depends on the fixed part of $S.$ In other words, $\delta_0$ should not depend on the small change that we make in the neighbourhoods of short geodesics $\gamma_1,\ldots, \gamma_N$ (or $\gamma_0$). Given such $\delta_0,$ we may define $g_{\vec{x}}$ (or $g_x$) in the neighbourhoods of $\gamma_i$ in such a way that (1) in Propositions \ref{bsm-p} and \ref{mb-p} are satisfied, see Subsection \ref{precise-par}.  

\begin{proof} (Proof of (2) in Proposition \ref{bsm-p} and (2), (3) in Proposition \ref{mb-p})

We will prove properties (2) and (3) in Proposition \ref{mb-p}. Property (2) in Proposition \ref{bsm-p} is proven in the same way as (2) in Proposition \ref{mb-p}. Let us start by giving a lower bound as in (3).

Assume that a cylinder $u:\R \times S^1 \rightarrow M$ connects $\gamma_i$ and $\gamma_j$ for $i<j,$ that is $u(-\infty,t)=\gamma_i(t),u(+\infty,t)=\gamma_j(t).$ Since $\gamma_i$ and $\gamma_j$ belong to different connected components of $\Sigma \setminus (\Delta_{0 1} \cup \Delta_{i (i+1)})$, we have that $\im(u)$ must intersect either $\Delta_{0 1}$ or $\Delta_{i (i+1)}.$ Assume first that it intersects $\Delta_{i (i+1)}$ and let $s_0\in \R$ be such that curves $u_{s_0}=u(s_0,\cdot):S^1\rightarrow M$ and $\Delta_{i (i+1)}$ intersect. Take $p=u(s_0,t_0)\in u_{s_0} \cap \Delta_{i (i+1)}$ and let $B(p;\rho)$ be a disc of radius $\rho$ around $p,$ with respect to the distance induced by $g_{\vec{x}}.$ If we take $\rho$ to be smaller than the injectivity radius at $p,$ $B(p;\rho)$ is embedded. Since curve $u_{s_0}$ belongs to a non-trivial homotopy class $\alpha,$ it is not completely contained in $B(p;\rho),$ i.e. there exists $t_1$ such that $u(s_0,t_1) \notin B(p;\rho).$ Hence, two arcs form $t_0$ to $t_1$ on $S^1$ are mapped into two paths $u\big( \wideparen{t_0t_1} \big)$ and $u\big( \wideparen{t_1t_0} \big)$ which connect the center $p$ of the disc $B(p;\rho)$ with the outside of the disc, see Figure \ref{cylinder_picture} below. This implies that $L_{g_{\vec{x}}}(u_{s_0})\geq 2\rho.$

\begin{figure}[ht]
\begin{center}
\includegraphics[scale=0.45]{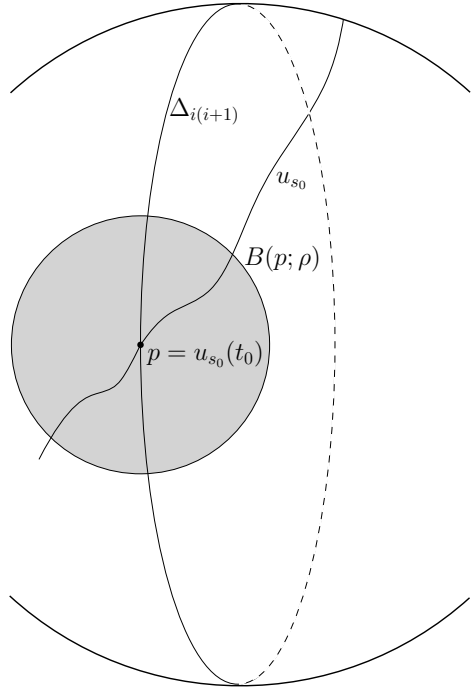}
\caption{Disc centered at a point $p\in u_{s_0} \cap \Delta_{i (i+1)}.$}\label{cylinder_picture}
\end{center}
\end{figure}

The exact same reasoning applies if we assume that $\im (u)$ intersects $\Delta_{0 1}$ at some point $p',$ in which case we get that $L_{g_{\vec{x}}}(u_{s_0})\geq 2\rho',$ where $\rho'$ is smaller than the injectivity radius at $p'.$ Note also that our metric $g_{\vec{x}}$ has $S^1$-symmetry near each $\Delta_{k (k+1)},$ $0\leq k \leq N$ because it was defined as an induced metric on a surface of revolution. This means that $\rho$ and $\rho'$ may be chosen to be the same for all $p\in \Delta_{i (i+1)}$ and all $p'\in \Delta_{0 1}.$ Moreover, the neighbourhoods of $\Delta_{k (k+1)}$ for $1\leq k \leq N-1$ are all isometric and hence $\rho$ and $\rho'$ can be chosen independently of $i$ and $j.$ By taking $\bar{\delta}_0=\min \{ \rho, \rho'\}$, we have that $L_{g_{\vec{x}}}(u_{s_0})\geq \bar{\delta}_0.$ Finally Cauchy-Schwarz inequality yields $$E_{g_{\vec{x}}}(u_{s_0}) \geq \frac{ L_{g_{\vec{x}}}(u_{s_0})^2}{2} \geq \frac{\bar{\delta}_0^2}{2},$$
which gives a lower bound as in (3).

\medskip

In order to give a lower bound as in (2), first notice that by Lemma \ref{geodesics2} every closed geodesics $\gamma$ in class $\alpha$, different than $\gamma_1,\ldots, \gamma_N$, either intersects $\Delta_{i(i+1)}$ for some $i=0, \ldots, N$ or it is entirely contained in $S \setminus O(N).$ In the first case we get a lower bound
$$E_{g_{\vec{x}}}(\gamma) \geq \frac{\bar{\delta}_0^2}{2},$$
with the same $\bar{\delta}_0$ as above by applying the exact same argument to $\gamma$ that we applied to $u_{s_0}.$

In the second case we have that $\gamma$ is also a closed geodesic on $\im \phi \subset \R^3,$ where $\phi:\Sigma \rightarrow \R^3$ is the embedding we fixed in order to define a multi-bulked surface. This means that 
$$E_{g_{\vec{x}}}(\gamma) \geq E_{\min},$$
with $E_{\min}$ being the minimal energy of a closed geodesic in class $\alpha$ on $\im \phi.$ Finally, we finish the proof by taking $\delta_0<\min \{ \bar{\delta}_0, \sqrt{2 E_{\min}} \}$.
\end{proof}

\subsection{Upper bounds in Proposition \ref{emb-thm} and Proposition \ref{emb-thm2}}  \label{sec-upper}
In this subsection, we will explain how to prove the upper bounds in Propositions \ref{emb-thm} and \ref{emb-thm2}.

Recall that metrics $g_x$, $x\in [0,\infty)$ which we used to prove lower bound in Proposition \ref{emb-thm} come from bulked spheres which are surfaces of revolution, see Proposition \ref{bsm-p}. On the other hand, metrics $g_{\vec{x}},$ $\vec{x}\in \mathcal{T}(N)$ which we used to prove lower bound in Proposition \ref{emb-thm2}, come from multi-bulked surfaces, which contain $O(N)$-part which is a surface of revolution, see Proposition \ref{mb-p}. Moreover, for different $\vec{x}\in \mathcal{T}(N),$ metric $g_{\vec{x}}$ only differ in the $O(N)$-part. Thus in order to compare different (multi)-bulked metrics, we will first explain how to compare metrics which come from surfaces of revolution in general.

Let $I\subset \R$ be an interval and $r:I \rightarrow [0,\infty)$ a smooth function. Using $r$ as a profile function, we define a surface of revolution $S(r)\subset \R^3$ and, by pulling back the standard metric from $\R^3,$ we define a metric $g^r$ on $I\times S^1.$ We claim the following.

\begin{lemma}\label{Lemma_matrix}
Let $r_1,r_2:I \rightarrow [0,\infty)$ be two profile functions and fix $C>0.$ Then $g^{r_1} \preceq C g^{r_2}$ if and only if
$$ (r_1(l))^2\leq C (r_2(l))^2, ~~ 1+(r_1'(l))^2\leq C ( 1+(r_2'(l))^2) $$
for all $l\in I.$
\end{lemma}
\begin{proof}
Introduce local coordinates $(l,\theta)\in I\times S^1.$ If $r:I\rightarrow [0,\infty)$ is a smooth profile function, a simple computation shows that the matrix of $g^r$ expressed in coordinates $(\partial_l ,\partial_\theta )$ satisfies
\begin{equation*}
    [g^r]_{(l,\theta)}= \begin{pmatrix} 1+(r'(l))^2 & 0 \\ 0 & (r(l))^2 \end{pmatrix}.
\end{equation*}
By definition $g^{r_1}\preceq C g^{r_2}$ is equivalent to $\|\cdot \|_{g^{r_1}}\leq \|\cdot \|_{Cg^{r_2}}$ at all points $(l,\theta)\in I \times S^1$, and hence the claim follows.
\end{proof}

From Lemma \ref{Lemma_matrix} we obtain the following corollary.

\begin{cor}\label{ratios}
Let $r_1,r_2:I \rightarrow [0,\infty)$ be two profile functions and denote by\footnote{Here we use the convention that $\frac{0}{0}=1.$} 
$$ C=\max_{l\in I} \max \left\{ \frac{r_1(l)}{r_2(l)},\frac{r_2(l)}{r_1(l)}, \frac{r_1'(l)}{r_2'(l)}, \frac{r_2'(l)}{r_1'(l)} \right\}.$$
If $g^{r_1},g^{r_2}$ are the induced Riemannian metrics on $I\times S^1$ it holds
$$\frac{1}{C^2} g^{r_1} \preceq g^{r_2} \preceq C^2 g^{r_1}.$$
\end{cor}
\begin{proof}
By Lemma \ref{Lemma_matrix} $g^{r_2} \preceq C^2 g^{r_1}$ is equivalent to 
$$ (r_1(l))^2\leq C^2 (r_2(l))^2, ~~ 1+(r_1'(l))^2\leq C^2 ( 1+(r_2'(l))^2).$$
The first inequality follows directly from the definition of $C.$ Since $C\geq 1,$ the second inequality follows from
$$1+(r_1'(l))^2 \leq 1+C^2(r_2'(l))^2 \leq C^2+C^2(r_2'(l))^2.$$
Inequality $\frac{1}{C^2} g^{r_1} \preceq g^{r_2}$ is proven by the same argument.
\end{proof}

Finally, we have the following proposition.
\begin{prop}\label{Prop-upper}
Bulked sphere metrics $g_x$, $x\in [0,\infty)$, whose existence is guaranteed by Proposition \ref{bsm-p}, can be defined in such a way that their profile functions satisfy 
$$\max_{l\in I} \max \left\{ \frac{r_x(l)}{r_y(l)},\frac{r_y(l)}{r_x(l)}, \frac{r_x'(l)}{r_y'(l)}, \frac{r_y'(l)}{r_x'(l)} \right\} = e^{|x-y|}$$
for all $x,y\in [0,\infty).$ Similarly, multi-bulked metrics $g_{\vec{x}},$ $\vec{x}\in \mathcal T(N)$ in Proposition \ref{mb-p} can be defined in such a way that the profile functions of the corresponding $O(N)$-parts satisfy 
$$\max_{l\in I} \max \left\{ \frac{r_{\vec{x}}(l)}{r_{\vec{y}}(l)},\frac{r_{\vec{y}}(l)}{r_{\vec{x}}(l)}, \frac{r_{\vec{x}}'(l)}{r_{\vec{y}}'(l)}, \frac{r_{\vec{y}}'(l)}{r_{\vec{x}}'(l)} \right\} = e^{|\vec{x}-\vec{y}|_\infty}$$
for all $\vec{x},\vec{y}\in \mathcal T(N).$
\end{prop}

In order to prove Proposition \ref{Prop-upper} one must specify precisely the profile functions which are used to define bulked spheres and multi-bulked surfaces. This is done in Subsection \ref{precise-par}.

\medskip

We are now ready to give a proof of the upper bounds.

\begin{proof} (Upper bounds in Propositions \ref{emb-thm} and \ref{emb-thm2}) We will only prove the upper bounds in terms of $d_{RBM}.$ The upper bounds in terms of $d_{SBM}$ then follow from Proposition \ref{SvsR}. 

For any $x \in [0, \infty)$, let $g_x$ be the bulked sphere metric given by Proposition \ref{bsm-p}. In order to prove the upper bound in Proposition \ref{emb-thm}, notice that Corollary \ref{ratios} and Proposition \ref{Prop-upper} imply that, for $x,y\in [0,\infty),$ it holds
$$e^{-2|x-y|} g_x \preceq g_y \preceq e^{2|x-y|} g_x .$$
Taking $\phi=\mathds{1}_M$ in the definition of $d_{RBM}$ we get
$$d_{RBM}(g_x,g_y)\leq 2|x-y|.$$
Recall that the embedding $\Phi: [0, \infty) \to \bar{\mathcal G}_{S^2}$ is defined by $\Phi(x) = R_x \cdot g_x$ where $R_x$ is the rescaling factor from (3) in Proposition \ref{bsm-p}. Now, Remark \ref{Scaling-RBM} implies 
\begin{align*}
    d_{RBM}(\Phi(x), \Phi(y)) & = d_{RBM}(R_x \cdot g_x, R_y \cdot g_y) \\
    & = d_{RBM}((R_x/R_y) \cdot g_x, g_y) \\
    & \leq d_{RBM}(g_x, (R_x/R_y) \cdot g_x) + d_{RBM}(g_x, g_y) \\
    & \leq |\ln R_x - \ln R_y| + 2|x-y|. 
\end{align*}
For any given $\varepsilon>0$, using (3) in Proposition \ref{bsm-p} with $\varepsilon_0 = \frac{e^{2\varepsilon} -1}{e^{2\varepsilon}+1}$ yields the desired upper bound in Proposition \ref{emb-thm}. 

To prove the upper bound in Proposition \ref{emb-thm2}, recall that $g_{\vec{x}}$ is the multi-bulked metric given by Proposition \ref{mb-p}. Moreover, metrics $g_{\vec{x}}$ for different $\vec{x}\in \mathcal{T}(N)$ only differ in the $O(N)$-part. Hence, Corollary \ref{ratios} and Proposition \ref{Prop-upper} imply
$$ e^{-2|\vec{x}-\vec{y}|_\infty} g_{\vec{x}} \preceq g_{\vec{y}} \preceq e^{2|\vec{x}-\vec{y}|_\infty} g_{\vec{x}},$$
for any $\vec{x},\vec{y}\in \mathcal{T}(N).$ Taking $\phi=\mathds{1}_M$ in the definition of $d_{RBM}$ gives
\begin{equation}\label{T-RBM}
d_{RBM}(g_{\vec{x}}, g_{\vec{y}}) \leq 2 |\vec{x}-\vec{y}|_\infty.    
\end{equation}

Map $\Phi:\R^N \rightarrow \bar{\mathcal G}_{M}$ in the proof of Proposition \ref{emb-thm2} is defined by $\Phi(\vec{x})=R_{Q(\vec{x})} \cdot g_{Q(\vec{x})}$ where $Q:\R^N \rightarrow \mathcal{T}(2N)$ is the quasi-isometric embedding given by Lemma \ref{red-par} and $R_{Q(\vec{x})}$ is the rescaling factor from (4) in Proposition \ref{mb-p}.  The same argument as above together with Lemma \ref{red-par} implies, for any $\vec{x}, \vec{y} \in \R^N$, 
$$ d_{RBM}(\Phi(\vec{x}), \Phi(\vec{y})) \leq |\ln R_{Q(\vec{x})} - \ln R_{Q(\vec{y})}| + 4N\cdot |\vec{x}- \vec{y}|_{\infty}.$$
For any $\varepsilon >0$, using (4) in Proposition \ref{mb-p} with $\varepsilon_0 = \frac{e^{2\varepsilon} -1}{e^{2\varepsilon}+1}$ yields the desired upper bound in Proposition \ref{emb-thm2}. 
 \end{proof}

\section{Quantitative existence of closed geodesics} \label{quan-geo}

The goal of this section is to prove Theorem \ref{exist-geodesic} and Corollary \ref{exist}. Since Corollary \ref{exist} immediately follows from Theorem \ref{TST}, we give its proof first.

\begin{proof} (Proof of Corollary \ref{exist})
We will prove the claim in the case of a finite bar $[a^2/2, b^2/2),$ the case of an infinite ray is treated in the same fashion. Using the isomorphism of persistence modules provided by Theorem \ref{sh-loop-thm}, we conclude that the barcode $\B_{*,\alpha}(U^*_{g_1}M)$ of $\mathbb{S}_{*,\alpha}(U^*_{g_1}M)$, contains the bar $[\ln a, \ln b).$ Theorem \ref{TST} and the assumptions give
\begin{equation}\label{Inequality_RBM}
d_{bottle}(\mathbb B_{*,\alpha}(U^*_{g_1}M), \mathbb B_{*,\alpha}(U^*_{g_2} M)) \leq \frac{1}{2} d_{RBM}(g_1,g_2)<\frac{1}{2}(\ln b - \ln a).
\end{equation}
Denoting $D=d_{bottle}(\mathbb B_{*,\alpha}(U^*_{g_1}M), \mathbb B_{*,\alpha}(U^*_{g_2}M))$ we have that for every $0<\varepsilon < \frac{1}{2}(\ln b - \ln a) - D$ there exists a $(D+\varepsilon)$-matching between $\mathbb B_{*,\alpha}(U^*_{g_1} M)$ and $\mathbb B_{*,\alpha}(U^*_{g_2} M).$ Since $D+\varepsilon < \frac{1}{2}(\ln b - \ln a)$, the bar $[\ln a, \ln b) \in \mathbb B_{*,\alpha}(U^*_{g_1} M)$ is not erased in this matching but rather has a genuine match $[c_\varepsilon,d_\varepsilon)\in \mathbb B_{*,\alpha}(U^*_{g_2} M)$ such that
\[ \max \left\{ \left| \ln a - c_\varepsilon \right|, \left| \ln b -d_\varepsilon   \right| \right\} \leq D + \varepsilon.\]
Since $\mathbb{S}_{*,\alpha}(U^*_{g_2}M)$ is pointwise finite dimensional, the fact that a bar $[c_\varepsilon,d_\varepsilon)$ as above exists for all $0<\varepsilon < \frac{1}{2}(\ln b - \ln a) - D$ implies that there exists $[c_0,d_0)\in \mathbb B_{*,\alpha}(U^*_{g_2} M)$ such that
\begin{equation}\label{sharp}
\max \left\{ \left| \ln a - c_0 \right|, \left| \ln b -d_0  \right| \right\} \leq D.
\end{equation} 
Indeed, if this was not the case, by shrinking $\varepsilon$ we would get infinitely many bars $[c_\varepsilon,d_\varepsilon)\in \mathbb B_{*,\alpha}(U^*_{g_2} M)$ which all contain the middle $\frac{\ln a+ \ln b}{2}$ of the interval $[\ln a, \ln b).$ This would imply that $S^{\frac{\ln a+ \ln b}{2}}_{*}(U^*_{g_2}M; \alpha)$ is not finite dimensional.

Finally, by Remark \ref{endpoints} we know that there exist closed geodesics $\gamma_1,\gamma_2$ of $g_2$ such that $c_0=\ln \sqrt{2E_{g_2}(\gamma_1)}$ and $d_0=\ln \sqrt{2E_{g_2}(\gamma_2)}$. Then the proof follows.
\end{proof}

In the rest of the section, we focus on proving Theorem \ref{exist-geodesic}. 

\subsection{Lemmas about persistence modules}

In order to prove Theorem \ref{exist-geodesic}, we will use a lemma about general persistence modules, see Lemma \ref{bar} below. We start with an auxiliary statement first.

\begin{lemma} \label{intervals}
Let $\k_{[a,b)}$ and $\k_{[c,d)}$ be two interval modules over field $\k.$ Then a non-zero persistence morphism
$$\f = \{f^t\}_{t \in \R}:\k_{[a,b)} \rightarrow \k_{[c,d)}$$
exists if and only if $c\leq a \leq d \leq b.$ Similarly for $b=d=+\infty$ a non zero persistence morphism
$$\f=\{f^t\}_{t \in \R}:\k_{[a,+\infty)} \rightarrow \k_{[c,+\infty)}$$
exists if and only if $c\leq a.$
\end{lemma}
\begin{proof}
Firstly, note that if $c\leq a \leq d \leq b$ there exists a non-zero persistence morphism $\f$ given by $f^t(1_{\k})=1_{\k}$ for $t\in [a,d)$ and $f^t=0$ otherwise. This proves one direction of the equivalence. For the other direction one readily sees that $b>c$ and $d>a$ since otherwise $[a,b)$ and $[c,d)$ do not intersect. The rest of the proof follows from a case analysis in terms of the order of endpoints $a,b,c,d.$ We will analyze one case, the other cases are treated in the same way. Assume, for example, that $c\leq a \leq b < d$ and let $\varepsilon>0$ be such that $b<b+\varepsilon<d.$ Now, if $f^t(1_{\k})\neq 0$ for some $t\in [a,b),$ on the one hand we have
$$f^{b+\varepsilon}(\iota_{t,b+\varepsilon}(1_{\k}))=f^{b+\varepsilon}(0)=0,$$
while on the other hand
$$f^{b+\varepsilon}(\iota_{t,b+\varepsilon}(1_{\k}))= \iota_{t,b+\varepsilon}(f^t(1_{\k}))=f_t(1_{\k})\neq 0$$
which gives a contradiction. 
\end{proof}

Recall that if $\mathbb V$ is a persistence module, for $A>0,$ shifted module $\mathbb V[A]$ is defined by $\mathbb V[A]^t=\mathbb V^{t+A}$ with comparison maps $\iota^{\mathbb V[A]}_{s,t}=\iota^{\mathbb V}_{s+A,t+A}.$ Barcode $\mathbb{B}(\mathbb V[A])$ is a shift of $\mathbb{B}(\mathbb V)$ by $A$ to the left, i.e. $[x,y) \in \mathbb{B}(\mathbb V[A])$ if and only if $[x+A,y+A) \in \mathbb{B}(\mathbb V).$ The following lemma is the main combinatorial ingredient of the proof of Theorem \ref{exist-geodesic}. 

\begin{lemma} \label{bar}
Let $\mathbb{V}, \mathbb{W}$ be two persistence modules, $A,B\geq0$ non-negative constants and 
$$\f:\mathbb{V} \rightarrow \mathbb{W}[A],~~ \g: \mathbb{W}[A] \rightarrow \mathbb{V}[A+B]$$
persistence morphisms such that the following diagram commutes
\[ \xymatrix{
\mathbb{V} \ar[rr]^{\iota^{\mathbb V}_{t,t+A+B}} \ar[rd]_{\f} && \mathbb{V}[A+B]. \\
& \mathbb W[A] \ar[ru]_{\g} &} \]
If there exists a bar $[a,b)\in \mathbb{B}(\mathbb{V})$ such that $b-a>A+B$ then there exists a bar $[c,d)\in \mathbb{B}(\mathbb{W})$ such that
$$a-B\leq c \leq a+A,~~b-B\leq d \leq b+A.$$\end{lemma}
\begin{proof}
Fix a decomposition of $\mathbb{V}$ given by Theorem \ref{thm-dc},
\begin{equation}\label{decomposition1}
\mathbb{V}= \bigoplus_{I\in \mathbb{B}(\mathbb{V})} \k_I
\end{equation}
and let 
\begin{equation}\label{decomposition2}
\mathbb{V}[A+B]= \bigoplus_{I\in \mathbb{B}(\mathbb{V})} \k_I[A+B]
\end{equation}
be the induced decomposition of $\mathbb{V}[A+B].$ 
Since $[a,b) \in \mathbb{B}(\mathbb{V}),$ $\k_{[a,b)}$ is a summand in (\ref{decomposition1}) and $\k_{[a,b)}[A+B]=\k_{[a-A-B,b-A-B)}$ is a summand in (\ref{decomposition2}). Denote by
$$\pi_{[a-A-B,b-A-B)}:\mathbb{V}[A+B] \rightarrow \k_{[a-A-B,b-A-B)}$$
the projection with respect to (\ref{decomposition2}). By restricting $f$ to  $\k_{[a,b)}$ we obtain the following commutative diagram for all $t\in \R:$
\[ \xymatrix{
\k^t_{[a,b)} \ar[rr]^{\iota_{t,t+A+B}} \ar[rd]_{f^t} && \k^{t}_{[a-A-B,b-A-B)} \\
& \mathbb W[A]^t \ar[ru]_{ \quad (\pi_{[a-A-B,b-A-B)}\circ g)^t} &} \]
Condition $A+B<b-a$ implies that $[a,b) \cap [a-A-B,b-A-B)\neq \emptyset$ and hence 
$$\iota_{t,t+A+B}:\k^t_{[a,b)} \rightarrow \k^t_{[a-A-B,b-A-B)}$$
is non-zero and is given by the obvious map equal to $\mathds{1}_{\k}$ when $t\in [a,b-A-B)$ and zero otherwise. Let us fix $t_0\in (a,b-A-B)$ and $1^{t_0}\in \k^{t_0}_{[a,b)}.$ By previous, It holds $\iota_{t_0,t_0+A+B}(1^{t_0})\neq 0.$

Fix a decomposition of $\mathbb{W}[A],$
\begin{equation}\label{decomposition3}
\mathbb{W}[A]= \bigoplus_{I\in \mathbb{B}(\mathbb{W}[A])} \k_I    
\end{equation}
and assume that 
$$f^{t_0}(1^{t_0})=\sum_{i=1}^N \lambda_i 1^{t_0}_{I_i},$$
where $\lambda_i \in \k,\, \lambda_i \neq 0$ and $1^{t_0}_{I_i}\in \k^{t_0}_{I_i}$ for $I_i\in \mathbb{B}(\mathbb{W}[A]).$ Since $(\pi_{[a-A-B,b-A-B)} \circ g)^{t_0} \circ f^{t_0}=\iota_{t_0,t_0+A+B}$ and $\iota_{t_0,t_0+A+B}(1^{t_0})\neq 0,$ we have that $(\pi_{[a-A-B,b-A-B)} \circ g)^{t_0}(\sum_{i=1}^N \lambda_i 1^{t_0}_{I_i}) \neq 0$ and hence there exists some $i_0 \in \{1, ..., N\}$ such that $(\pi_{[a-A-B,b-A-B)} \circ g)^{t_0}(\lambda_{i_0} 1^{t_0}_{I_{i_0}})\neq 0,$ see Figure \ref{interval_picture} below.
\begin{figure}[ht]
\begin{center}
\includegraphics[scale=0.4]{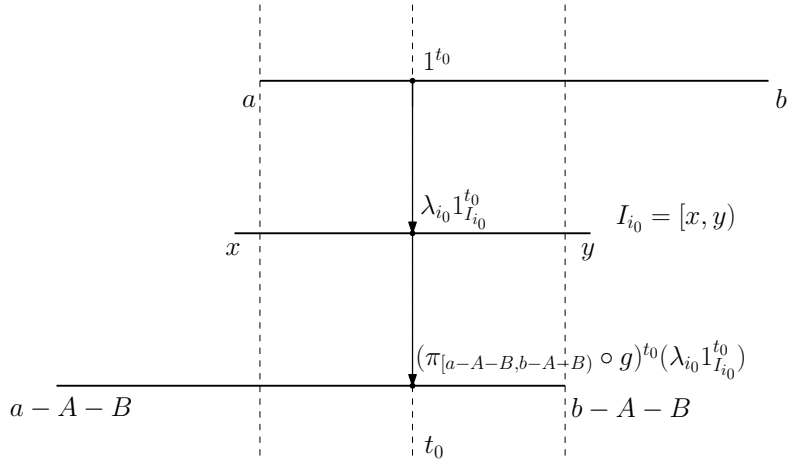}
\caption{Morphisms at a point $t_0$.}\label{interval_picture}
\end{center}
\end{figure}

Let $I_{i_0}=[x,y).$ We claim that $a-A-B\leq x\leq a$ and $b-A-B\leq y\leq b.$ Indeed, denote by
$$\pi_{I_{i_0}}:\mathbb{W}[A] \rightarrow \k_{I_{i_0}}$$
the projection with respect to (\ref{decomposition3}). This projection is a morphism of persistence modules and we have that 
$$\pi_{I_{i_0}}\circ \f:\k_{[a,b)} \rightarrow \k_{I_{i_0}}$$
is a non-zero persistence morphism because $(\pi_{I_{i_0}}\circ f)^{t_0}(1^{t_0})=\lambda_{i_0} 1^{t_0}_{I_{i_0}} \neq 0.$ Thus, Lemma \ref{intervals} imples that $x\leq a$ and $y\leq b.$

Similarly, restricting $g$ to $\k_{I_{i_0}}$ gives
$$\pi_{[a-A-B,b-A-B)}\circ g :\k_{I_{i_0}} \rightarrow \k_{[a-A-B,b-A-B)}.$$
This morphism is non-zero because $(\pi_{[a-A-B,b-A-B)}\circ g)^{t_0}(\lambda_{i_0}1^t_{I_{i_0}})\neq 0$ and hence Lemma \ref{intervals} implies that $a-A-B\leq x$ and $b-A-B\leq y.$

To finish the proof notice that $[x,y)\in \mathbb{B}(\mathbb W[A])$ and hence $[x+A,y+A)\in \mathbb{B}(\mathbb{W}).$ For this bar it holds $a-B\leq x+A \leq a+A$ and $b-B\leq y+A \leq b+A$ and we may take $[c,d)=[x+A,y+A).$
\end{proof}

\subsection{Proof of Theorem \ref{exist-geodesic}}\label{Proof_non_symmetric}

\begin{proof} (Proof of Theorem \ref{exist-geodesic})
We will prove only the case of a finite bar $[x,y)\in \mathbb{B}(\mathbb H_{*, \alpha}(M, g_1)),$ the other case is proved in the same manner. 

It follows from the definition that $U^*_{Cg} M=\sqrt{C}U^*_g M.$ From the assumption $\frac{1}{C_1} g_1 \preceq g_2 \preceq C_2 g_1$, one concludes that
$$U^*_{\frac{1}{C_1} g_1} M \subset U^*_{g_2} M \subset U^*_{C_2 g_1} M,$$
which is equivalent to
\begin{equation}\label{inclusions}
\frac{1}{\sqrt{C_1}}U^*_{ g_1} M \subset U^*_{g_2} M \subset \sqrt{C_2} U^*_{g_1} M.
\end{equation}
Applying contravariant functor $\SH^a_{*}(\cdot ; \alpha)$ to (\ref{inclusions}) gives the following commutative diagram
\[ \xymatrix{
\SH^a_{*}(\sqrt{C_2}U^*_{ g_1} M; \alpha) \ar[rr]^{\h_{inc}} \ar[rd]_{\h_{inc}} &&  \SH^a_{*}(\frac{1}{\sqrt{C_1}}U^*_{ g_1} M; \alpha)\\
&  \SH^a_{*}(U^*_{g_2} M; \alpha) \ar[ru]_{ \h_{inc}} &} \]
where $\h_{inc}$ denote maps induced by the respective inclusions. Applying (2) in Proposition \ref{fp-sh} with $C=\frac{1}{\sqrt{C_1 C_2}}\leq 1$ to the horizontal arrow gives us
\[ \xymatrix{
\SH^a_{*}(\sqrt{C_2}U^*_{ g_1} M; \alpha) \ar[rr]^{\iota_{a,\sqrt{C_1 C_2}a}} \ar[rd]_{\h_{inc}} && \SH^{\sqrt{C_1 C_2}a}_{*}(\sqrt{C_2}U^*_{ g_1} M; \alpha) \\
&  \SH^a_{*}(U^*_{g_2} M; \alpha) \ar[ru]_{\quad (r_{\frac{1}{\sqrt{C_1C_2}}})^{-1} \circ \h_{inc}} &}\]
where $\iota_{a,\sqrt{C_1 C_2}a}$ denotes the persistence comparison map of the symplectic persistence module $\mathbb {SH}_{*, \alpha}(\sqrt{C_2}U_{g_1}^*M)$ and $r_{\frac{1}{\sqrt{C_1C_2}}}$ is the isomorphism given by (2) in Proposition \ref{fp-sh}. In terms of the logarithmic  version of symplectic persistence modules, setting $t = \ln a$ gives us 
\[ \xymatrix{
S^t_{*}(\sqrt{C_2}U^*_{ g_1} M;\alpha) \ar[rr]^-{\iota_{t,t+ \ln \sqrt{C_1} + \ln \sqrt{C_2}}} \ar[rd]_{\h_{inc}} && S^{t+ \ln \sqrt{C_1} + \ln \sqrt{C_2}}_{*}(\sqrt{C_2}U^*_{ g_1} M; \alpha). \\
& S^t_{*}(U^*_{g_2} M; \alpha) \ar[ru]_{\quad (r_{\frac{1}{\sqrt{C_1C_2}}})^{-1} \circ \h_{inc}} &}\]
Since $[x,y) \in  \mathbb{B}(\mathbb H_{*,\alpha}(M,g_1))$, $[\ln \sqrt{2x},\ln \sqrt{2y})\in \mathbb{B}_{*, \alpha}(U_{g_1}^*M)$ by Theorem \ref{sh-loop-thm}, where $\mathbb{B}_{*, \alpha}(U_{g_1}^*M)$ denotes the barcode of persistence module $\mathbb S_{*, \alpha}(U^*_{g_1} M)$. Now, Proposition \ref{fp-sh} implies that $[\ln \sqrt{2x} + \ln \sqrt{C_2}, \ln \sqrt{2y} + \ln \sqrt{C_2} ) \in \mathbb{B}_{*, \alpha}(\sqrt{C_2}U^*_{g_1} M).$ The length of this bar is $\ln \sqrt{\frac{y}{x}} $ and since $\ln \sqrt{\frac{y}{x}} > \ln \sqrt{C_1} + \ln \sqrt{C_2}$ by the assumption, we can apply Lemma \ref{bar} with $A=0$ and $B=\ln \sqrt{C_1}+ \ln \sqrt{C_2}.$ 

It follows that there exists a bar $(c,d] \in \mathbb{B}_{*, \alpha}(U^*_{g_2} M)$ such that 
$$ \ln \sqrt{2x} - \ln \sqrt{C_1} \leq c \leq \ln \sqrt{2x} + \ln \sqrt{C_2},$$
and
\[ \ln \sqrt{2y} - \ln \sqrt{C_1} \leq d \leq \ln \sqrt{2y} + \ln \sqrt{C_2}.\]
By Theorem \ref{sh-loop-thm}, bar $[\frac{1}{2}e^{2c},\frac{1}{2}e^{2d})\in \B(\mathbb H_{*, \alpha}(M,g_2))$ and its endpoints satisfy
$$ \frac{x}{C_1} \leq \frac{1}{2}e^{2c} \leq C_2 x ,\,\,\,\, \frac{y}{C_1} \leq \frac{1}{2}e^{2d} \leq C_2 y.$$
Endpoints of bars in $\mathbb B(\mathbb H_{*, \alpha}(M,g_2))$ correspond to the energies of closed geodesics and thus there exist closed geodesics $\gamma_1,\gamma_2$ of $g_2$ such that 
$$E_{g_2}(\gamma_1)=\frac{1}{2}e^{2c} ,\,\,\,\, E_{g_2}(\gamma_2)=\frac{1}{2}e^{2d}.$$
This finishes the proof. \end{proof}

\section{Appendix}
\subsection{Background on persistence modules and barcodes} \label{app-pm}
We review some rudiments of the theory of persistence modules and barcodes. For a detailed exposition see \cite{EH08,Ghr08,Car09,Wein11,BL14,Edel14,Oud15,CdSGO16,PRSZ17}.

\begin{dfn}
A persistence module $\mathbb V$ (parametrized by $\R$) over a field $\k$ consists of the following data,
\begin{equation} \label{pm}
 \mathbb V = \left\{\{V^t\}_{t \in \R}, \{\iota_{s,t}: V^s \to V^t\}_{s\leq t \in \R}\right\}
 \end{equation}
where each $V^t$ is a finite dimensional $\k$-vector space, $\iota_{s,t}$ are $\k$-linear maps such that $\iota_{t,t}=\mathds{1}_{V^t}$ for all $t\in \R$ and for all $s\leq t \leq r$ we have $$\iota_{s,r} = \iota_{t,r} \circ \iota_{s,t}.$$
Maps $\iota_{s,t}$ are called comparison maps of $\mathbb{V}.$
\end{dfn}

Typical examples of persistence modules are homologies of filtered chain complexes, $V^t=H(C^t_*),$ $t \in \R$, $\iota_{s,t}$ being induced by inclusions $C^s_*\hookrightarrow C^t_*$ for $s\leq t.$ These include (Morse) homologies of sublevel sets $\{f\leq t \}$ of a Morse function $f:M\rightarrow \R$ on a closed manifold $M$, as well as, Hamiltonian Floer homologies of the Hamiltonian Floer chain complex, filtered by symplectic action functional, on a symplectically aspherical manifold, see \cite{PS16}. In this paper we consider a persistence module coming from filtered symplectic homology, see Section \ref{section-SympHom}.

\begin{dfn} \label{p-mor} Given two persistence modules $\mathbb V$ and $\mathbb W$, we call a family of maps $\mathfrak f = \{f^t\}: \mathbb V \to \mathbb W$ a {\it persistence morphism} if each $f^t: V^t \to W^t$ is a $\k$-linear map and $\mathfrak f$ commutes with comparison maps, i.e., for each $s\leq t$,
\begin{equation} \label {pmm}
f^t \circ \iota^{\mathbb V}_{s,t} = \iota^{\mathbb W}_{s,t} \circ f^s.
\end{equation}
Moreover, we call $\mathbb V$ and $\mathbb W$ {\it isomorphic} if there exist persistence morphisms $\mathfrak f: \mathbb V \to \mathbb W$ and $\mathfrak g: \mathbb W \to \mathbb V$ such that for each $t$, $f^t \circ g^t = {\mathds 1}_{W_t}$ and $g^t \circ f^t = {\mathds 1}_{V_t}$.
\end{dfn}

Persistence modules and their morphisms form an Abelian category. Kernels and images are taken pointwise, meaning for each $t\in \R$, while the direct sum of $\mathbb{V}$ and $\mathbb{W}$ is given by
$$((\mathbb{V}\oplus \mathbb{W})^t, \iota_{s,t}^{\mathbb{V}\oplus \mathbb{W}})= (V^t \oplus W^t, \iota_{s,t}^{\mathbb{V}} \oplus \iota_{s,t}^{\mathbb{W}}) ~ \textrm{ for all } s\leq t. $$

An important abstract example is the {\it interval persistence module $\k_I$} associated to a given interval $I \subset \R$ which is defined by setting $\k^t_I = \k$ if $t \in I$ and $0$ otherwise, with the comparison map $\iota_{s,t} = {\mathds 1}_{\k}$ when both $s,t \in I$ and $\iota_{s,t} =0$ otherwise. Interval persistence modules are the building blocks of all persistence modules, i.e. the following theorem holds.

\begin{theorem} \label{thm-dc} (Stucture theorem \cite{CZ05})
For a persistence module $\mathbb V$, there exists a collection of intervals $\{{I_j} \}_{j\in \mathcal{J}}$ with repetitions, unique up to reordering, such that
\[ \mathbb V = \bigoplus_{j\in \mathcal{J} } \k_{I_j}.\]
This collection of intervals is called the barcode of $\mathbb V$ and is denoted by $\mathbb B(\mathbb V),$ while the intervals themselves are called bars. Since $V^t$ are finite dimensional for all $t\in \R,$ each bar appears finitely many times in $\mathbb B(\mathbb V).$ The number of times a bar $I$ appears in $\mathbb B(\mathbb V)$ is called the multiplicity of $I.$
\end{theorem}

\begin{remark}
Theorem \ref{thm-dc} for homology of filtered chain complexes was proven in \cite{Bar94}, using different language, since persistence modules and barcodes were not introduced at the time.
\end{remark}

The set of all endpoints of all bars in $\mathbb B(\mathbb V)$ is called the \textit{spectrum} $\mathbb{V}.$ It follows from the definitions that if $[s,t]$ contains no points of the spectrum, then $\iota_{s,t}$ is an isomorphism. In other words, $V^t$ changes when $t$ passes through a point in the spectrum. In the case of the homology of sublevel sets $V^t=H_*(\{ f\leq t\}),$ of a Morse function $f,$ the spectrum equals the set of all critical values of $f.$ Similarly, in the case of filtered Hamiltonian Floer homology on a symplectically aspherical manifold, points in the spectrum are actions of closed Hamiltonian orbits. The spectrum of a persistence module which we use in this paper, the one coming from filtered symplectic homology, consists of periods of closed Reeb orbits on the boundary of the domain, see Remark \ref{endpoints} for the case of unit codisc bundles.

\begin{remark}
All the persistence modules considered in this paper are defined using conventions which guarantee that all the intervals in the corresponding barcodes have left endpoints closed and right endpoints open. In other words, all the bars are either equal to $(-\infty, +\infty)$ or of the form $[a,b)$ for $a<b\leq +\infty $ with finite $a.$ Moreover, we sometimes use the set of parameters $t\in \R^+=(0,+\infty)$ instead of $t\in \R$ in the definition of our persistence modules. This difference is non-essential because the two sets of parameters can be related by an order-preserving bijection, for example $\ln : \R^+\rightarrow \R.$
\end{remark}

Theorem \ref{thm-dc} translates an algebraic structure into combinatorial information via correspondence $\mathbb V \leftrightarrow \mathbb B(\mathbb V).$ Our goal is to use this correspondence to quantitatively compare different persistence modules by comparing their corresponding barcodes. In order to achieve this, we need to introduce the following two distances: {\it interleaving distance $d_{inter}$} between persistence modules and {\it bottleneck distance $d_{bottle}$} between multisets of intervals.

\begin{dfn} \label{int-bottle} We say that two persistence modules $\mathbb V$ and $\mathbb W$ are {\it $\delta$-interleaved} if there exist persistence morphisms $\Phi$ and $\Psi$,
\[ \xymatrix{
\mathbb V \ar[r]^{\Phi} & \mathbb W[\delta] \ar[r]^{\Psi[\delta]} & \mathbb V[2\delta]} \,\,\,\,\,\,\mbox{s.t.}\,\,\,\, \,\,\, \Psi[\delta] \circ \Phi = \{\iota^{\mathbb V}_{t ,t+ 2\delta}\}  \]
and
\[ \xymatrix{
\mathbb W \ar[r]^{\Psi} & \mathbb V[\delta] \ar[r]^{\Phi[\delta]} & \mathbb W[2\delta]} \,\,\,\,\,\,\mbox{s.t.}\,\,\,\, \,\,\, \Phi[\delta] \circ \Psi = \{\iota^{\mathbb W}_{t ,t+ 2\delta}\}. \]
Here $\mathbb V[\delta]$ and $\Phi[\delta]$ are $\delta$-shifts of $\mathbb V$ and $\Phi,$ given by $\mathbb V[\delta]^t = V^{t +\delta},$ $\iota^{\mathbb V[\delta]}_{s,t}=\iota^{\mathbb V}_{s+\delta,t +\delta}$ and $(\Phi[\delta])^t=\Phi^{t+\delta},$ and the same holds for $\mathbb W[\delta]$, $\Psi[\delta].$ Now one defines,  
\[ d_{inter}(\mathbb V, \mathbb W) := \inf\{\delta \geq 0 \,| \, \mbox{$\mathbb V$ and $\mathbb W$ are $\delta$-interleaved}\}.\]
\end{dfn}

On the other hand we define $d_{bottle}$ in the following way.

\begin{dfn}
Let $\mathbb B_1$ and $\mathbb B_2$ be two multisets of intervals with finite multiplicities. We call $\sigma: \B_1 \to \B_2$ an {\it $\varepsilon$-matching} if there exist subsets $(\B_1)_{short} \subset \B_1$, $(\B_2)_{short} \subset \B_2$, consisting of intervals with lengths not greater than $2 \varepsilon$, such that
\[ \sigma: \B_1 \backslash (\B_1)_{short}  \to \B_2 \backslash (\B_2)_{short} \,\,\,\,\mbox{is a bijection}\]
and for any two intervals $[a, b),[c,d),$ for which $\sigma([a,b))=[c, d),$ it holds $|a-c| \leq \varepsilon$ and $|b-d| \leq \varepsilon$. We say that bars in $(\B_1)_{short}$ and $(\B_2)_{short}$ are {\it erased}, while bars in $\B_1 \backslash (\B_1)_{short}$ and $\B_2 \backslash (\B_2)_{short}$ are {\it matched}. The bottleneck distance is now given by
\[ d_{bottle}(\B_1, \B_2) := \inf\{\varepsilon \geq 0 \,| \, \exists\, \mbox{$\varepsilon$-matching between $\B_1$ and $\B_2$}\}. \]

\end{dfn}

\begin{remark}
In the above definition of $d_{bottle}$ we assumed that matched intervals $[a,b)$ and $[c,d)$ have left endpoints closed and right endpoints open in accordance to the conventions we use in the paper. The same definition applies if we make no assumption on the endpoints of bars and Theorem \ref{iso-thm} below remains true. 
\end{remark}

The following theorem is one of the most important results in the theory of persistence modules.

\begin{theorem} \label{iso-thm} (Isometry Theorem \cite{Chazal07, Chazal09, Les15})
Let $\mathbb V$ and $\mathbb W$ be two persistence modules and denote their corresponding barcodes by $\B(\mathbb V)$ and $\B(\mathbb W).$ Then
\[ d_{inter}(\mathbb V, \mathbb W) = d_{bottle}(\B(\mathbb V), \B(\mathbb W)). \]
\end{theorem}

\subsection{Precise parameterizations}\label{precise-par}
In this subsection, we will give precise parameterizations of bulked spheres and multi-bulked surfaces announced in Sections \ref{Proofs} and \ref{sec-bulk}. The metrics which we are going to define will satisfy all the properties that we used in these sections, namely we will prove (1) and (3) in Proposition \ref{bsm-p}, (1) and (4) in Proposition \ref{mb-p} as well as Proposition \ref{Prop-upper}.

\subsubsection{Parameterizations of bulked spheres} Let $S$ be a union of two spheres with radius $A = \sqrt{\frac{1}{8\pi}}$ touching at point. The area of $S$ is equal to 1 and $S$ can be obtained as a (singular) surface of revolution. The graph of the profile function $r$ which defines $S$ is the union of two semicircles of radius $A$ centered at $-A$ and $A,$ see Figure \ref{kisssphere}. 
\begin{figure}[ht]
\includegraphics[scale=0.6]{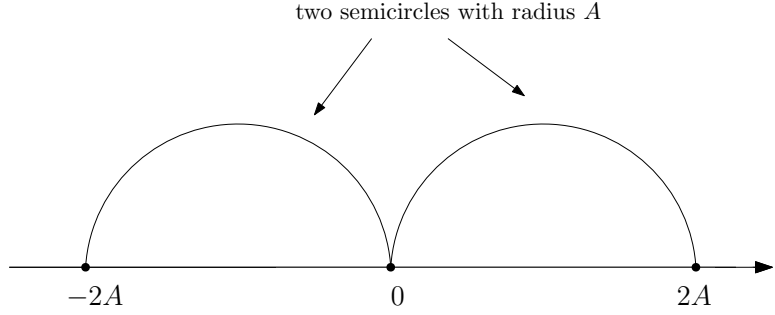}
\caption{Profile function $r$}
\label{kisssphere}
\end{figure}

Let $n \in \N$ and $B = 10^{-n} A.$ We consider $n$ to be a free parameter which will eventually be chosen large. The profile functions $r_x$, $x\in [0,\infty),$ which define bulked sphere metrics $g_x,$ will all be even on $[-2A, 2A]$ and they will coincide with $r$ on $[-2A,-B]\cup [B,2A].$ On $[-B,B],$ each $r_x$ will interpolate between two semicircles and will have a local minimum at $0.$ Let
\begin{equation} \label{delta-2}
\delta_0 = \sqrt{\frac{\pi}{8}} \cdot \frac{1}{2 \cdot 10^n + 1} \cdot \frac{3 - 10^{-2n}}{\sqrt{2 \cdot 10^n - 1}}. 
\end{equation}
Since $\delta_0 = O(10^{\frac{-3n}{2}}),$ by picking large enough $n$, we can make $\delta_0$ arbitrarily small. The following proposition holds.

\begin{prop} \label{geo-prop} Given any sufficiently small $\varepsilon>0$, for all sufficiently large $n$ and $A,B,\delta_0$ as above, there exists a family of profile functions $r_x,$ $x\in [0,\infty),$ each of which defines a bulked sphere metric $g_x$ such that
\begin{itemize}
    \item[(1)] $|\Vol_{g_x}(S^2)-1|\leq \varepsilon$ and ${\rm diam}(S^2,g_x)\leq 100 \sqrt{1- \varepsilon}.$
    \item[(2)] $r_x$ coinicide outside $[-B,B]$ for all $x\in [0,\infty).$ 
    \item[(3)] $r_x(0)=\frac{\delta_0}{2\pi}e^{-x}.$
    \item[(4)] For any $x, y \in [0, \infty)$, \[ \max_{l \in [-2A, 2A]} \max\left\{ \frac{r_x(l)}{r_y(l)}, \frac{r_y(l)}{r_x(l)}, \frac{r'_x(l)}{r'_y(l)}, \frac{r'_y(l)}{r'_x(l)}\right\} = e^{|x-y|}. \]
\end{itemize}
In (4), as before, we use convention that $\frac{0}{0}=1.$ 
\end{prop}
\begin{proof}
We ask for $r_x$ to be even on $[-2A, 2A]$, and hence only give their definitions on $[0,2A]$. Let $B' = 10^{-2n}A$, $h = \frac{3\delta_0}{2\pi}$, $h_x = \frac{\delta_0}{2 \pi} e^{-x}$, and define
\[ r_x(l) = \frac{h - h_x}{(B')^2} l^2 + h_x, \,\,\,\,\mbox{for}\,\,\,\,l \in [0, B']. \]
It immediatelly follows that $r_x(0)=h_x=\frac{\delta_0}{2 \pi} e^{-x}$ and hence (3) is satisfied. On the other hand, a simple computation shows that for all $x,y\in [0,\infty)$
\begin{equation}\label{first_part}
\max_{l \in [0, B']} \max\left\{ \frac{r_x(l)}{r_y(l)}, \frac{r_y(l)}{r_x(l)}, \frac{r'_x(l)}{r'_y(l)}, \frac{r'_y(l)}{r'_x(l)}\right\} = e^{|x-y|}.
\end{equation}
Notice also that $r_x(B')=h$ for all $x\in [0,\infty)$, i.e., the graphs of all $r_x$ meet at the point $(B',h),$ see Figure \ref{abpic} (a).
\begin{figure}[ht]
\includegraphics[scale=0.5]{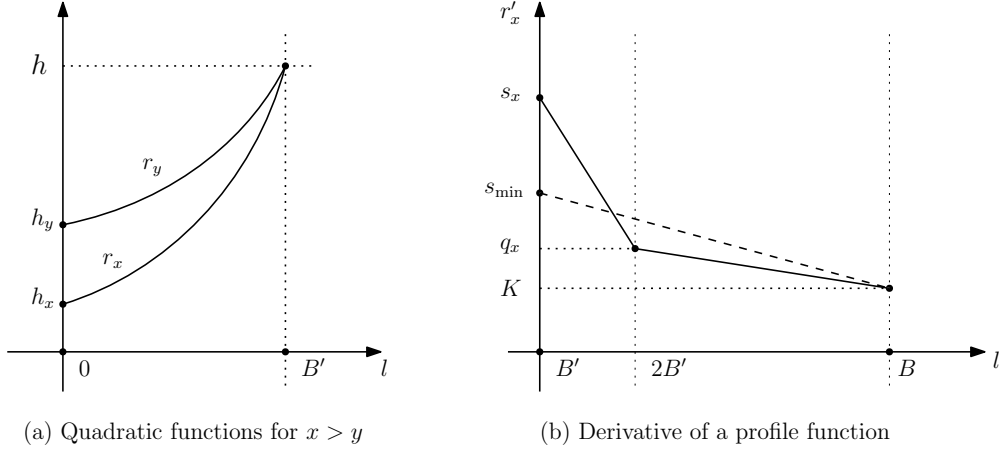}
\caption{Parameterization of $r_x$ in the region $[0, B]$}
\label{abpic}
\end{figure}

Graphs of profile functions $r_x$ will connect $(B',h)$ to $(B,\sqrt{A^2-(A-B)^2})$ on a semicircle. To this end, let 
\[ s_x := r'_x(B') = \frac{2(h-h_x)}{B'} \,\,\,\,\mbox{and}\,\,\,\, K = \frac{A - B}{\sqrt{2AB - B^2}}.\]
For $x\in [0,\infty)$ we have $s_x\in [\frac{2\delta_0}{\pi B'}, \frac{3\delta_0}{\pi B'})$ and we denote $s_{\min} = s_{0} = \frac{2\delta_0}{\pi B'}.$ On the other hand $K$ is the derivative at $B$ of the function $y=\sqrt{A^2-(A-l)^2},$ which defines a semicircle.

We now define $r_x$ by giving its derivative on $[B',B].$ Let
\[q_x=\frac{s_{\min}(B - B') + KB'-s_xB' }{B-B'}.\]
On $[B',2B']$, the derivative $r'_x$ is by definition equal to a linear function whose graph connects $(B',s_x)$ and $(2B',q_x).$ On $[2B',B]$, $r'_x$ is equal to another linear function, whose graph connects $(2B',q_x)$ and $(B,K).$ It is easy to check that $r'_0$ is linear on $[B',B],$ i.e. $(B',s_{\min}),$ $(2B',q_0)$ and $(B,K)$ are on the same line, as well as that $K<q_x\leq q_0$ for all $x\in [0,\infty)$, see Figure \ref{abpic} (b). Explicitly $r'_x$ is given by
\begin{equation} \label{exp-slope}
\displaystyle r'_x(l) = \left\{ \begin{array}{cc} \frac{q_x - s_x}{B'} (l - B') + s_x & l \in [B', 2B'] \\ \frac{K-q_x}{B-2B'}(l-B) + K & l \in [2B', B] \end{array} \right..
\end{equation}
Straightforward calculation shows that
$$r_x(B)=h+\int_{B'}^B r'_x(l) ~d l= \sqrt{A^2-(A-B)^2},$$
and thus by setting $r_x(l)=\sqrt{A^2-(A-l)^2}$ for $l\in [B,2A],$ we obtain a $C^1$-smooth function $r_x:[0,2A]\rightarrow [0,\infty).$ Another straightforward calculation shows that for $x,y\in [0,\infty)$
\begin{equation}\label{second_part}
\max_{l \in [B', 2A]} \max\left\{ \frac{r_x(l)}{r_y(l)}, \frac{r_y(l)}{r_x(l)}, \frac{r'_x(l)}{r'_y(l)}, \frac{r'_y(l)}{r'_x(l)}\right\} \leq e^{|x-y|}. 
\end{equation}
Moreover, by making a $C^1$-small perturbation near the points $B',2B'$ and $B$, we can make sure that $r_x$ are all smooth while (\ref{second_part}) remains valid. Finally, we extend $r_x$ to $[-2A,2A]$ by setting $r_x(l)=r_x(-l).$ It is clear from the construction that property (2) holds. 

Combining (\ref{first_part}) and (\ref{second_part}) proves property (4). By taking large enough $n$ we can guarantee that property (1) holds, which finishes the proof.
\end{proof}

We can now give a proof of (1) and (3) in Proposition \ref{bsm-p}. 
\begin{proof} (Proof of (1) and (3) in Proposition \ref{bsm-p}) 
Denote by $g_x$ the metric induced from profile function $r_x$ given by Proposition \ref{geo-prop}. By (2) in Proposition \ref{geo-prop} we get that $L_{g_x}(\gamma_0)=2\pi r_x(0)=\delta_0 e^{-x}.$ Since $\gamma_0$ has constant speed
$$E_{g_x}(\gamma_0) = \frac{L_{g_x}^2(\gamma_0)}{2} =\frac{\delta_0^2}{2} e^{-2x},$$
which proves (1) in Proposition \ref{bsm-p}. 

To prove (3) in Proposition \ref{bsm-p}, let $R_x= \frac{1}{\sqrt{{\rm Vol}_{g_x} S^2}}.$ Now ${\rm Vol}_{R_x \cdot g_x} S^2 =1$ and from (1) in Proposition \ref{geo-prop} it follows that
\[ {\rm diam}(S^2, R_x \cdot g_x) = R_x \cdot {\rm diam}(S^2, g_x) \leq \sqrt{\frac{1}{1-\varepsilon}} \cdot (100 \sqrt{ 1- \varepsilon}) \leq 100, \]
as well as that
\[ R_x^2 (1- \varepsilon) \leq 1 \leq R_x^2 (1+ \varepsilon). \]

Thus $R_x \cdot g_x\in \bar{\mathcal{G}}_{S^2}$, and taking small enough $\varepsilon$ finishes the proof.
\end{proof}

\subsubsection{Parameterizations of multi-bulked surfaces} Recall that a cylindrical segment is a surface of revolution with constant profile function $r: I \to [0,\infty)$ on an open interval $I.$ Let $\Sigma$ be a closed, orientable surface of genus at least one and fix $N\in \N.$ Denote by $g_{std}$ the standard Riemannian metric on $\R^3$ and let $0<\tau <<1$ be a small number. We fix an embedding $\phi:\Sigma \rightarrow \R^3$ such that $\Vol_{\phi^*g_{std}}(\Sigma)=1$, ${\rm diam}(\Sigma,\phi^*g_{std}) \leq 99$ and $\im \phi$ contains a cylindrical segment $C$ given by a profile function $r_{seg}:(L_{-},L_{+})\rightarrow \R,$ $r_{seq}\equiv \tau$ with $L_{+}-L_{-}=2\tau.$ 

We construct our $N$-bulked surface by replacing $C$ with an open chain of $N-1$ spheres denoted by $O(N).$ Let $A_N=\frac{\tau}{N}$ and take a profile function $r:(L_{-},L_{+})\rightarrow [0,\infty)$ whose graph consists of $N-1$ semicircles of radius $A_N$ and two connecting ends. More precisely, on $[L_{-},L_{-}+A_N]$, $r$ is strictly decreasing, and moreover on $[L_{-}+\frac{A_N}{2},L_{-}+A_N]$ its graph coincides with a part of a semicircle with radius $A_N$ centered at $L_{-}.$ Similarly, $r$ is strictly increasing on $[L_{+}-A_N,L_{+}]$ and on $[L_{+}-A_N,L_{+}-\frac{A_N}{2}]$ its graph coincides with a semicircle with radius $A_N$ centered at $L_+$, see Figure \ref{multipara}. 
\begin{figure}[ht]
\includegraphics[scale=0.5]{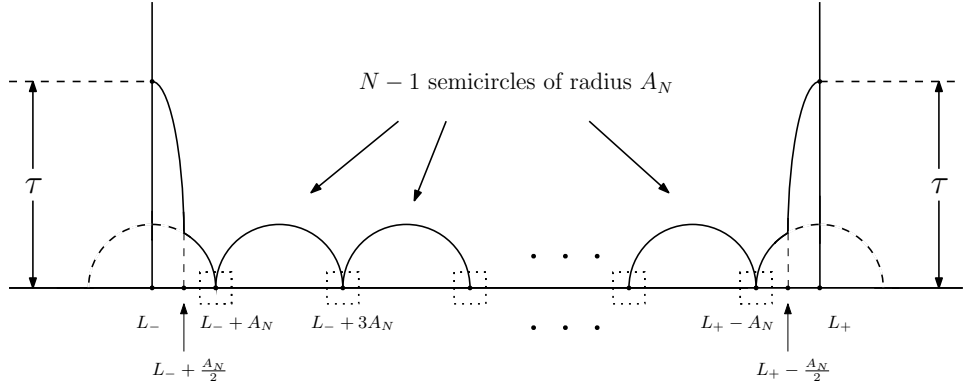}
\caption{Profile function $r$ of the chain of $N-1$ spheres}
\label{multipara}
\end{figure}

Let $g^r$ be the (singular) metric on $\Sigma$ obtained from the standard metric on $\R^3$ after replacing $C$ with the $O(N)$ given by the above profile function $r.$ The area of $O(N)$ is of order $\tau^2$, while its diameter is of order $\tau.$ Thus, for any given $0<\varepsilon<1$, by taking $\tau$ small enough, we have that $|\Vol_{g^r}(\Sigma)-1|\leq \varepsilon$ and ${\rm diam}(\Sigma,g^r)\leq 100 \sqrt{1-\varepsilon}.$ Now, let $B_N=10^{-n} A_N$,  $\delta_0(N) = \delta_0 \cdot \frac{\varepsilon  \sqrt{8 \pi}}{N}$, $\delta_0$ being given by (\ref{delta-2}). Notice also that $B_N$ and  $\delta_0(N)$ depend on a parameter $n.$ By carrying out the same construction as in the bulked sphere case near each of the touching points $L_{-}+A_N,L_{-}+3A_N,\ldots , L_{+}-3A_N,L_{+}-A_N$, we obtain the following proposition.  

\begin{prop} \label{geo-prop-2} Given any sufficiently small $\varepsilon>0$, for all sufficiently large $n$ and $A_N,B_N,\delta_0(N)$ as above, there exists a family of profile functions $r_{\vec{x}},$ $\vec{x} \in \mathcal T(N),$ each of which defines an $N$-bulked metric $g_{\vec{x}}$ such that
\begin{itemize}
    \item[(1)] $|\Vol_{g_{\vec{x}}}(\Sigma)-1|\leq \varepsilon$ and ${\rm diam}(\Sigma,g_{\vec{x}})\leq 100\sqrt{1-\varepsilon}.$
    \item[(2)] For different $\vec{x}\in\mathcal T(N)$, $r_{\vec{x}}$ coinicide outside of $B_N$-neighbourhoods of  $L_{-}+A_N,L_{-}+3A_N,\ldots , L_{+}-3A_N,L_{+}-A_N.$
    \item[(3)] For $k=1,\ldots , N$, $r_{\vec{x}}(L_{-}+(2k-1)A_N)=\frac{\delta_0}{2\pi}e^{-x_k},$ where $\vec{x}=(x_1,\ldots,x_N).$
    \item[(4)] For any $\vec{x}, \vec{y} \in \mathcal T(N)$, \[ \max_{l \in [L_{-}, L_{+}]} \max\left\{ \frac{r_{\vec{x}}(l)}{r_{\vec{y}}(l)}, \frac{r_{\vec{y}}(l)}{r_{\vec{x}}(l)}, \frac{r'_{\vec{x}}(l)}{r'_{\vec{y}}(l)}, \frac{r'_{\vec{y}}(l)}{r'_{\vec{x}}(l)}\right\} = e^{|\vec{x}-\vec{y}|_\infty}. \]
\end{itemize}
In (4), as before, we use convention that $\frac{0}{0}=1.$ 
\end{prop}

We can now give a proof of (1) and (4) in Proposition \ref{mb-p}. 

\begin{proof} (Proof of (1) and (4) in Proposition \ref{mb-p}) Let $g_{\vec{x}}$ be the $N$-bulked metric from Proposition \ref{geo-prop-2}. By (2) in Proposition \ref{geo-prop} it follows
$$L_{g_{\vec{x}}}(\gamma_1)=\delta_0(N) e^{-x_1}, \ldots , L_{g_{\vec{x}}}(\gamma_N)=\delta_0(N) e^{-x_N}.$$
Since all $\gamma_i$ have constant speed, we have
$$E_{g_{\vec{x}}}(\gamma_i) = \frac{L_{g_{\vec{x}}}^2(\gamma_i)}{2} =\frac{(\delta_0(N))^2}{2} e^{-2x_i}, ~~{\rm for }~~ i=1,\ldots, N,$$
which proves (1) in Proposition \ref{mb-p}. The proof of (4) in Proposition \ref{mb-p} is exactly the same as the proof of (3) in Proposition \ref{bsm-p} above. 
\end{proof}

Finally, the proof of Proposition \ref{Prop-upper} follows directly  from (4) in Proposition \ref{geo-prop} and (4) in Proposition \ref{geo-prop-2}.

\subsection{Reduction of parameterization space} \label{red-emb-space}
Recall that $\mathcal T(2N)$ is defined as
\[ \mathcal T(2N) = \left\{\vec{x} = (x_1, ..., x_{2N}) \in [0, \infty)^{2N}\,| \, x_1 \leq  x_2 \leq ... \leq x_{2N} \right\}.\]
In this subsection, we will prove Lemma \ref{red-par}. It claims that for every $N \in \N$ there exists a quasi-isometric embedding $Q:(\R^N, |\cdot|_{\infty}) \rightarrow (\mathcal T(2N), |\cdot|_{\infty}).$ We construct $Q$ as a composition of two quasi-isometric embeddings $A$ and ${\bf L}$ as follows
\[ (\R^N, |\cdot|_{\infty}) \xrightarrow{{\bf L}} ([0, \infty)^{2N}, |\cdot|_{\infty}) \xrightarrow{A} (\mathcal T(2N), |\cdot|_{\infty}). \]

\subsubsection{Construction of ${\bf L}$} Consider a map $L: \R \to [0, \infty)^2$, given by
\[ L(x) = \left\{ \begin{array}{cc} (1, -x+1) \,\,\,\,\mbox{when $x <0$} \\ (1+x, 1) \,\,\,\,\mbox{when $x \geq 0$} \end{array} \right.. \]
If we realize $[0, \infty)^2$ as the first quadrant of $\R^2$, then map $L$ gives an ``L-shaped'' embedding of $\R$ with corner at $(1,1)$. Now define a multi-dimensional version of $L$, that is ${\bf L}: \R^N \to [0, \infty)^{2N}$ by
\begin{equation*}
{\bf L}(\vec{x}) = (L(x_1), ..., L(x_N)).
\end{equation*}

We claim the following.

\begin{lemma} \label{lip-emb}
For any $N \in \N$ and $\vec{x}, \vec{y} \in \R^{N}$, it holds
\[ \frac{1}{2} |\vec{x} - \vec{y} |_{\infty} \leq |{\bf L}(\vec{x}) - {\bf L}(\vec{y})|_{\infty} \leq |\vec{x}- \vec{y}|_{\infty}. \]
\end{lemma}
\begin{proof} First consider the case when $N=1$. When both $x, y$ are negative or both $x,y$ are non-negative, it is easy to see $|x-y| = |L(x) - L(y)|_{\infty}$. When $x<0$ and $y \geq 0$,
\begin{align*}
|L(x) - L(y)|_{\infty} & = |(1, - x+1) - (1+y, 1)|_{\infty} \\
& = |(-y, -x)|_{\infty} = \max\{|x|, |y|\} \\
& \leq |x| + y = |x-y|.
\end{align*}
On the other hand,
\begin{align*}
2|L(x) - L(y)|_{\infty} &= 2 \max\{|x|, |y|\} \\
& \geq |x| + |y| \\
& = |x| + y = |x-y|.
\end{align*}
The same argument works for $x \geq 0$ and $y<0$. Therefore, we get a bi-Lipschitz relation
\begin{equation} \label{bi-lip-re}
|L(x)- L(y)|_{\infty} \leq |x-y| \leq 2|L(x) - L(y)|_{\infty}
\end{equation}
Then
\begin{align*}
|{\bf L}(\vec{x}) - {\bf L}(\vec{y})|_{\infty} & = \max\{|L(x_1) - L(y_1)|_{\infty}, ..., |L(x_N) - L(y_N)|_{\infty}\}\\
&\leq \max\{|x_1 - y_1|, ..., |x_N - y_N|\} = |\vec{x} - \vec{y}|_{\infty}.
\end{align*}
and
\begin{align*}
2|{\bf L}(\vec{x}) - {\bf L}(\vec{y})|_{\infty} & = \max\{2|L(x_1) - L(y_1)|_{\infty}, ..., 2|L(x_N) - L(y_N)|_{\infty}\} \\
& \geq \max\{|x_1 - y_1|, ..., |x_N - y_N|\} = |\vec{x} - \vec{y}|_{\infty}.
\end{align*}
Thus we get the conclusion.
\end{proof}
\subsubsection{Construction of $A$} Consider the following map $A: [0, \infty)^{2N} \to \mathcal T(2N)$,
\[ A(\vec{x}) = A(x_1, ..., x_{2N}) = \left(x_1, x_1 + x_2, ..., x_1 + ... + x_{2N} \right).\]
We have
\begin{lemma} \label{lip-emb-2} For every $N \in \N$ and $\vec{x}, \vec{y} \in [0, \infty)^{2N}$, it holds
\[ \frac{1}{2} |\vec{x} - \vec{y}|_{\infty} \leq |A(\vec{x}) - A(\vec{y})|_{\infty} \leq (2N)\cdot |\vec{x} - \vec{y}|_{\infty}. \] \end{lemma}

\begin{proof} Conclusion of Lemma \ref{lip-emb-2} immediately follows from the following inequalities. For $a_1, ..., a_n \in \R$,
\[ \frac{1}{2} \max\{|a_1|, ..., |a_n|\} \leq \max\{|a_1|, |a_1 + a_2|, ..., |a_1 + ... + a_n| \} \leq n \cdot \max\{|a_1|, ..., |a_n|\}. \]
The second inequality comes from the fact that for any $k \in \{1, ..., n\}$,
\[ |a_1 + ... + a_k| \leq |a_1|+ ... + |a_k| \leq k \cdot \max\{|a_1|, ..., |a_k|\}\leq n  \cdot \max\{|a_1|, ..., |a_n|\}.\]
For the first inequality, consider the two-term case first, that is 
\begin{equation} \label{2-t}
\max\{|a_1|, |a_1+ a_2|\} \geq \frac{1}{2} \max\{|a_1|, |a_2|\}.
\end{equation}
If $|a_1| \geq |a_2|$, the inequality is obvious. If on the other hand, $|a_1| \leq |a_2|$, then
\begin{align*}
2\max\{|a_1|, |a_1 + a_2|\} & = 2 \max\{|a_1|, |a_1 - (-a_2)|\}\\
&\geq  2 \max\{|a_1|, ||a_1| - |a_2||\}\\
& = 2 \max\{|a_1|, |a_2| - |a_1|\} \\
& \geq |a_1| + |a_2| - |a_1| \\
& = |a_2| = \max\{|a_1|, |a_2|\}.
\end{align*}
This proves (\ref{2-t}).

For the general case, assume that $\max\{|a_1|, |a_2|, ..., |a_n|\}=|a_k|.$ If $k=1,$ the inequality if obvious. If $k\geq 2$ then (\ref{2-t}) implies that
$$\max \{ |a_1+...+a_{k-1}|,|a_1+...+a_k| \} \geq \frac{1}{2} \max \{ |a_1+...+a_{k-1}|,|a_k| \}\geq \frac{1}{2}|a_k|,$$
and the claim follows. \end{proof}

\begin{proof} (Proof of Lemma \ref{red-par}) Set $Q = A \circ \bf L$ and we get the conclusion. \end{proof}

\subsection{Geodesic flow on a torus of revolution} \label{geodesic-flow}

We give a detailed analysis of the geodesic flow of the metric of revolution on $\mathbb{T}^2$ and in particular prove Lemma \ref{No-geodesics}. 

Recall from Example \ref{Torus_of_revolution} that $f:[-A,A]\rightarrow (0,+\infty)$ was a smooth, even function, which extends $2A$-periodically to a smooth function on $\R.$ Moreover, $f$ was strictly increasing on $[-A,0]$ and hence strictly decreasing on $[0,A]$ with a unique maximum at 0 and two minima at $\pm A.$ Using $f$ as a profile function, we defined a metric of revolution $g$ on $\mathbb{T}^2= \R/ 2A \Z \times \R/ 2\pi\Z.$ In other words, $g$ is a pull back of the Euclidean metric on $\R^3$ via the embedding
$$(x,\theta)\rightarrow (x,f(x)\cos \theta, f(x) \sin \theta).$$
Recall also that we used a change of variable $X(x)=\int_{0}^{x}\sqrt{1+(f'(t))^2} dt, x\in [-A,A].$ The new variable satisfies $X\in[-T,T]$ for $T=\int_{0}^{A}\sqrt{1+(f'(t))^2}$ and we denoted $F(X)=f(x(X)).$ A direct computation shows that in $(X,\theta)$ coordinates metric has the following form:
\begin{equation} \label{metric-tensor}
g_{(X, \theta)} = \begin{pmatrix} 1 & 0 \\ 0 & F^2(X) \end{pmatrix}.
\end{equation}
The Lagrangian of the geodesic flow of $g$ is given by
\[ L(X, \theta, v_{X}, v_{\theta}) = \frac{1}{2} \left(v_X^2  + F^2(X) v_{\theta}^2 \right). \]
while momenta are
\[ p_X = \frac{\partial L}{\partial v_X} =  v_X \,\,\,\,\mbox{and}\,\,\,\, p_{\theta} = \frac{\partial L}{\partial v_\theta} = F^2(X)v_{\theta}. \]
We compute the Hamiltonian as a Legendre transform
\begin{equation} \label{Lege}
H(X, \theta, p_X, p_{\theta}) = p_X v_X + p_\theta v_\theta - \frac{1}{2} \left( v^2_X + F^2(X) v^2_{\theta}\right)= \frac{1}{2} \left( p^2_X + \frac{p^2_{\theta}}{F^2(X)}\right).
\end{equation}
Hamiltonian equations are
\begin{equation} \label{ham-equ}
\left\{ \begin{array}{l}
{\displaystyle \dot X  = \frac{\partial H}{\partial p_X} = p_X} \\
\\
{\displaystyle\dot {\theta} = \frac{\partial H}{\partial p_{\theta}} = \frac{p_{\theta}}{F^2(X)} }\\
\\
{\displaystyle\dot {p}_X  = - \frac{\partial H}{\partial X} = \frac{F'(X)}{F^3(X)} p_{\theta}^2}\\
\\
{\displaystyle\dot {p}_{\theta} = - \frac{\partial H}{\partial \theta} = 0}.
\end{array} \right.
\end{equation}
The above system is integrable with two integrals given by $H$ and $p_\theta.$ Let us analyze the system on the energy level $H=\frac{1}{2}$ (this corresponds to unit speed geodesics) and let us assume that $p_\theta = \sqrt{C}\geq 0.$ The case $p_\theta<0$ is treated similarly (note that $(p_X,p_\theta)\rightarrow (-p_X,-p_\theta)$ corresponds to changing the direction of a geodesic). Now, (\ref{Lege}) translates to
\begin{equation}\label{H_integral}
p_X^2+\frac{C}{F^2(X)}=1,
\end{equation}
while Hamiltonian equations become

\begin{equation} \label{H_equation}
\left\{ \begin{array}{l}
{\displaystyle \dot X = p_X} \\
\\
{\displaystyle\dot {\theta} = \frac{\sqrt{C}}{F^2(X)} }\\
\\
{\displaystyle\dot {p}_X  = C \frac{F'(X)}{F^3(X)}} \\
\end{array} \right.
\end{equation}

If $C=0,$ (\ref{H_integral}) and (\ref{H_equation}) imply that $\dot \theta=0,~ \dot X =p_X=\pm 1.$ Hence, in this case geodesics are given by $\theta(t)=const,~ X(t)=X(0)\pm t.$

If $C>0$, (\ref{H_integral}) implies that $\sqrt{C}\leq \max F$ and we distinguish four cases.

\medskip

\noindent \underline{$1^\circ ~~ \sqrt{C}=\max F$:}
\\
In this case (\ref{H_integral}) implies that $X=0$ and thus $p_X=0,~ F'(X)=0.$ Now, (\ref{H_equation}) gives $\dot X=0,~\dot p_X=0,~ \dot \theta = \frac{1}{\max F}$ and thus $X(t)=0,~ p_X(t)=0,~\theta(t)=\theta(0)+\frac{t}{\max F}.$ This solution describes a closed geodesic $\hat{\gamma}$, i.e. the parallel circle of radius $\max F$ at $X=0$, and it's iterations.
\\
\\
\underline{$2^\circ ~~ \min F <\sqrt{C}<\max F$:}
\\
In this case the dynamics is constrained to the interval where $\sqrt{C}\leq F(X)$, i.e. on $[-\lambda_C(F),\lambda_C(F)]$ for $F^{-1}(\sqrt{C})=\{-\lambda_C(F),\lambda_C(F) \}.$ Moreover, on this interval it holds $p_X=\pm \sqrt{1-\frac{C}{F^2(X)}}$ and the portrait of the system in $(X,p_X)$-plane looks as follows:

\begin{figure}[ht]
\includegraphics[scale=0.55]{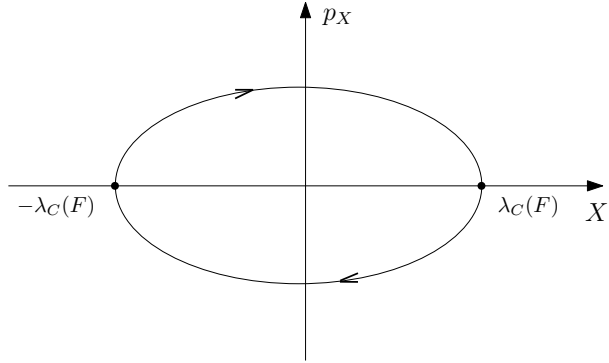}
\caption{$(X,p_X)$-portrait when $\min F<\sqrt{C}<\max F$}
\end{figure}

\noindent\underline{$3^\circ ~~ \sqrt{C}=\min F$:}
\\
In this case $\lambda_C(F)=T$ and $p_X=\pm \sqrt{1-\frac{C}{F^2(X)}}.$ The behaviour of the flow at $X=\pm T$ differs from the behaviour when $X\in (-T,T).$ Indeed, if $X=\pm T$, we have $p_X=F'(X)=0$ and (\ref{H_equation}) becomes $\dot X=0,~\dot p_X=0,~ \dot \theta = \frac{1}{\min F}.$ Thus, we obtain a solution $X(t)=\pm T,~ p_X(t)=0,~\theta(t)=\theta(0)+\frac{t}{\min F},$ which describes a closed geodesic $\check{\gamma}$, i.e. the parallel circle of radius $\min F$ at $X=\pm T$, and it's iterations. In $(X,p_X)$-plane solutions with $X\in (-T,T)$ trace two curves which connect points $-T$ and $T$ and the portrait looks as follows:

\begin{figure}[ht]
\includegraphics[scale=0.45]{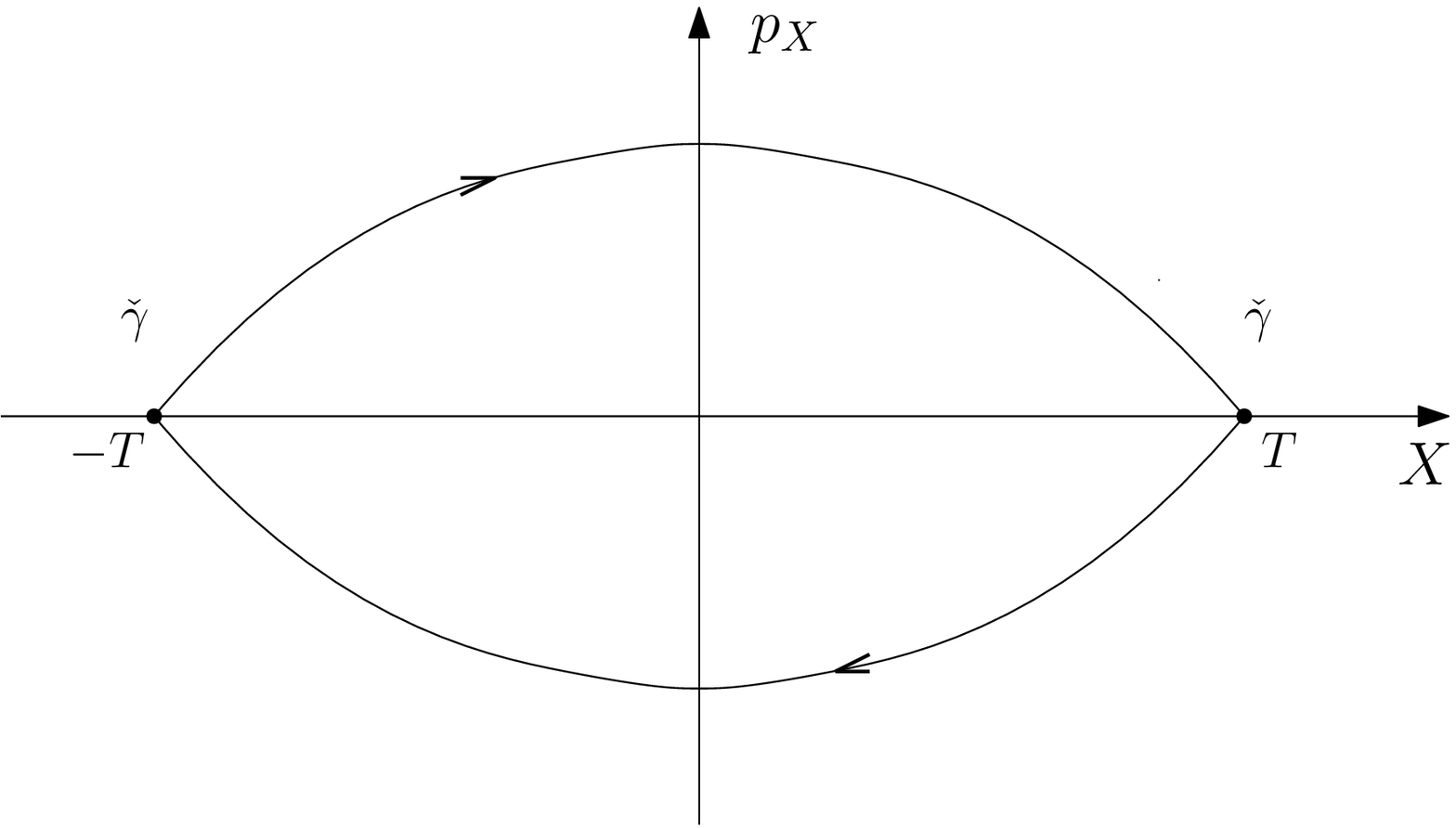}
\caption{$(X,p_X)$-portrait when $\sqrt{C}=\min F$}
\end{figure}

\noindent\underline{$4^\circ ~~ \sqrt{C}<\min F$:}
\\
In this case $1-\frac{C}{F^2(X)}>0$ for all $X\in [-T,T]$ and $p_X=\pm \sqrt{1-\frac{C}{F^2(X)}}.$ The portrait looks as follows:

\begin{figure}[ht]
\includegraphics[scale=0.45]{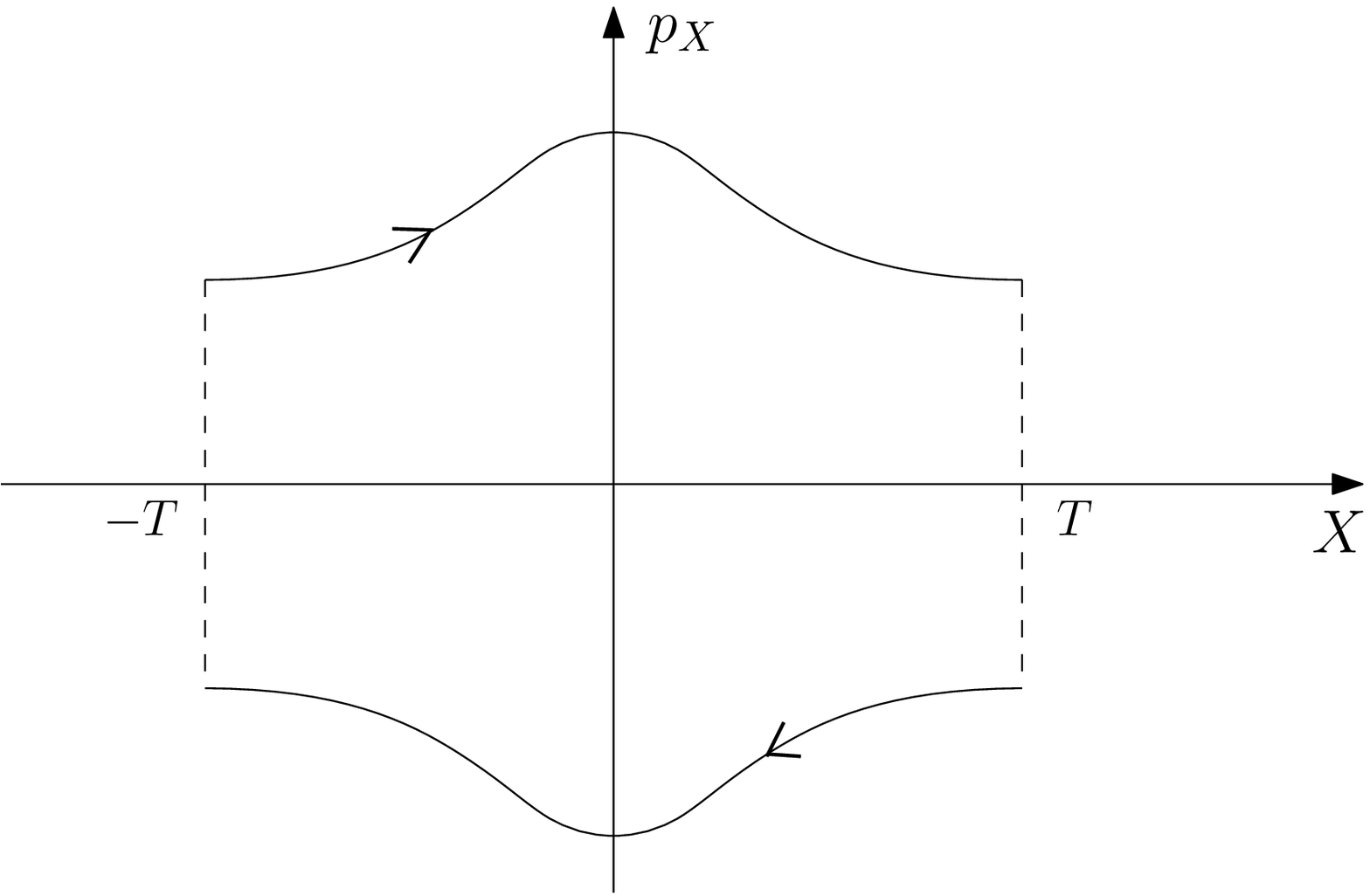}
\caption{$(X,p_X)$-portrait when $\sqrt{C}<\min F$}
\end{figure}

Recall that $\alpha$ denotes the homotopy class of loops represented by a loop $\theta(t)=t, ~t\in[0,2\pi], ~X=const.$ Let us now focus on closed geodesics in homotopy class $\alpha.$ As we saw above, there are always two closed geodesics in this class, namely $\check{\gamma}$ and $\hat{\gamma}.$ Their properties are summarized by the following lemma.

\begin{lemma} \label{Critical_points}
Assume that $F''(T)>0$ as well as that $0<-F(0)F''(0)<1.$ Then $\check{\gamma}$ and $\hat{\gamma}$ are non-degenerate closed geodesics and $\ind \check{\gamma}=0, \ind \hat{\gamma}=1.$
\end{lemma}

\begin{proof}

First notice that $F(0)=f(0), F(\pm T)=f(\pm A), F'(0)=f'(0)=0,F'(\pm T)=f'(\pm A)=0$ and $F''(0)=f''(0),F''(\pm T)=f''(\pm A).$ This immediately follows after differentiating (twice) the expression $f(x)=F(X(x)),$ using that $X'(x)=\sqrt{1+(f'(x))^2}$ as well as that $f'(0)=f'(\pm A)=0.$ Hence, $F''(\pm T)>0$ implies $f''(\pm A)>0$ and thus $\check{\gamma}$ is non-degenerate and $\ind \check{\gamma}=0,$ as show in Subsection \ref{Geo_index} (proof of Lemmas \ref{lm-short-geo} and \ref{lm-short-geo-2}).

By Lemma \ref{Lemma_Nulity}, $\hat{\gamma}$ is non-degenerate if and only if there are no periodic Jacobi fields along $\hat{\gamma},$ orthogonal to $\dot{\hat{\gamma}}.$ As in the case of $\check{\gamma},$ Jacobi fields are computed using (\ref{jac2}), however in this case $K=-4\pi^2f(0)f''(0)=-4\pi^2F(0)F''(0)>0.$ Orthogonal Jacobi fields are of the form $J(t)=(J_1(t),0,0)$ and (\ref{jac2}) translates to
\begin{equation} \label{jac3}
{\ddot J}_1(t) + K \cdot J_1(t) = 0.
\end{equation}
In other words, the space of orthogonal Jacobi fields is spanned by
$$J_{\sin}(t)=(\sin(\sqrt{K}t),0,0)\,\,\,\, \mbox{and}\,\,\,\, J_{\cos}(t)=(\cos(\sqrt{K}t),0,0).$$ 
By the assumption, $0<-F(0)F''(0)<1$ and hence $0<\sqrt{K}<2\pi,$ which impies that no orthogonal Jacobi field is periodic, i.e. $\hat{\gamma}$ is non-degenerate.

Using that 
$$(J_{\sin}(t), \dot{J}_{\sin}(t)) = ((\sin(\sqrt{K}t),0,0), (\sqrt{K}\cos(\sqrt{K}t),0,0))$$
and
$$(J_{\cos}(t), \dot{J}_{\cos}(t)) = ((\cos(\sqrt{K}t),0,0), (-\sqrt{K}\sin(\sqrt{K}t),0,0)),$$
we may express the linearized Poincare map $P:(T\hat{\gamma}(0))^{\perp}\oplus(T\hat{\gamma}(0))^{\perp}\rightarrow (T\hat{\gamma}(0))^{\perp}\oplus(T\hat{\gamma}(0))^{\perp} $ with respect to basis $((1,0,0),(0,0,0))^T, ((0,0,0),(\sqrt{K},0,0))^T$ by the following matrix
$$
P = \begin{pmatrix} \cos \sqrt{K} & \sin \sqrt{K} \\ -\sin \sqrt{K} & \cos \sqrt{K} \end{pmatrix}.
$$
The eigenvalues of $P$ are $e^{\pm i \sqrt{K}}$ and hence $\hat{\gamma}$ is not hyperbolic. Finally, in order to compute $\ind \hat{\gamma}$ we use Lemma 3.4.2. from \cite{Kli95}. This lemma states that if a closed geodesic $\gamma$ on an orientable surface is non-degenerate and not hyperbolic, then it's index is an odd number equal to either $m$ or $m+1,$ where $m$ denotes the number of points $\gamma(t_*),0<t_*<1,$ conjugate\footnote{Recall that a point $\gamma(t_*)$ is called conjugate to $\gamma(0)$ along $\gamma$ if there exists a Jacobi field $J$ along $\gamma$ such that $J(0)=J(t_*)=0.$ Since the space of Jacobi fields along $\gamma$ is spanned by orthogonal Jacobi fields and $\dot{\gamma}$ and $t\dot{\gamma},$ one readily sees that $\gamma(t_*)$ is conjugate to $\gamma(0)$ if and only if there exists an orthogonal Jacobi field $J$ along $\gamma$ such that $J(0)=J(t_*)=0.$} to $\gamma(0)$ along $\gamma$. Since every orthogonal Jacobi field has the form 
$$J(t)=(A\cos (\sqrt{K} t) + B\sin (\sqrt{K} t), 0, 0),$$
$J(0)=(0,0,0)$ is equivalent to $A=0$ and since $0< \sqrt{K}<2\pi$ there can be at most one point conjugate to $\hat{\gamma}(0),$ namely $\hat{\gamma}(\frac{\pi}{\sqrt{K}}).$ This means that $m=0$ or $m=1$ and thus $\ind \hat{\gamma}=1.$
\end{proof}

The above analysis of the portrait in $(X,p_X)$-plane shows that closed geodesics in class $\alpha$ other than $\check{\gamma}$ and $\hat{\gamma}$ can only appear when $\min F<\sqrt{C}<\max F.$ In this case, for a fixed $C$, the flow is periodic in $(X,p_X)$-plane. Denote by $\Theta_F(C)$ the shift in $\theta$-coordinate made by a flow line $\tilde{\gamma}$ by the time it makes a single turn from $(-\lambda_C,0)$ back to $(-\lambda_C,0)$ (we abbreviate $\lambda_C=\lambda_C(F)$). 

Formally, let $\tilde{\gamma}(t)=(X(t),\theta(t),p_X(t),\sqrt{C})$ be a flow line of the Hamiltonian system (\ref{H_equation}), assume without lost of generality that $X(0)=-\lambda_C,p_X(0)=0$ and let $t_0>0$ be the smallest time when $X(t_0)=-\lambda_C,p_X(t_0)=0$ again. Define
$$\Theta_F(C)=\theta(t_0)-\theta(0).$$
As notation suggests, $\Theta_F(C)$ only depends on $F$ and $C.$ Indeed, using (\ref{H_integral}), (\ref{H_equation}) and the symmetry of $F$ we calculate\footnote{Compare to Proposition 2 in \cite{Alex06}.}
\begin{equation}\label{Theta}
\begin{split}
\Theta_F(C) & =\int_0^{t_0} \dot{\theta}(t)dt = \int_{0}^{t_0} \frac{\sqrt{C}}{F^2(X(t))}dt=2 \int_{-\lambda_C}^{\lambda_C} \frac{\sqrt{C}}{F^2(X)}\frac{dX}{\dot{X}}\\
& =2 \int_{-\lambda_C}^{\lambda_C} \frac{\sqrt{C}}{F^2(X)}\frac{dX}{p_X}=2\sqrt{C}\int_{-\lambda_C}^{\lambda_C} \frac{dX}{F(X)\sqrt{F^2(X)-C}}.
\end{split}
\end{equation}
We will define $F_\varepsilon$, described in Lemma \ref{No-geodesics}, for which $\Theta_{F_\varepsilon}(C)>2\pi$ for all $\min F_\varepsilon<\sqrt{C}<\max F_\varepsilon.$ Since from (\ref{H_equation}) we have that $\dot{\theta}>0$, $\theta(t)$ is increasing and hence $\Theta_{F_\varepsilon}(C)>2\pi$ implies that any closed geodesic $\gamma$ must make at least two full turns in $\theta$-direction, i.e. it can not lie in the homotopy class $\alpha.$

Everything we have done so far applies to any $F$ satisfying the necessary conditions. Let us now focus on concrete examples and prove Lemma \ref{No-geodesics}.
\begin{proof}(Proof of Lemma \ref{No-geodesics}) First, we note that $F$ is implicitly defined by $f$ and hence, it is not a priori clear that we may choose $F$ freely. However, one can show that if $F:[-T,T]\rightarrow (0,+\infty)$ satisfies $|F'(X)|<1$ for all $X\in [-T,T]$ then there exists $f:[-A,A]\rightarrow (0,+\infty)$, for some $A$, such that $F(X)=f(x(X)).$ Indeed, by setting $x(X)=\int_0^X \sqrt{1-(F'(\tau))^2}d\tau,~A=\int_0^T \sqrt{1-(F'(\tau))^2}d\tau$ and $f(x(X))=F(X),$ one checks by a direct computation that $f$ defines $F.$ Moreover, since $|F'(X)|<1$ for all $X\in [-T,T],$ $x(X)$ is a smooth function and $\frac{dx}{dX}>0$ on $[-T,T].$ Thus, $f$ is smooth if and only if $F$ is smooth.

Fix $0<\sqrt{k}<m$ and let us take $T=1,F_0:[-1,1]\rightarrow(0,+\infty)$ given by $F_0(X)=\frac{1}{\sqrt{kX^2+m}}.$ One readily checks that $|F'(X)|< 1$ for all $X\in [-1,1].$ For small enough $\varepsilon>0,$ $F_\varepsilon$ will be a smoothing of $F_0$ near the points $\pm 1.$ We start by analysing $F_0.$

Denote by $\lambda^0_C=\lambda_C(F_0)=\sqrt{\frac{\frac{1}{C}-m}{k}}$ and let $\tilde{\gamma}(t)=(X(t),\theta(t),p_X(t),\sqrt{C}),~ X(0)=-\lambda^0_C,~ p_X(0)=0$ be a solution of the Hamiltonian system (\ref{H_integral}), (\ref{H_equation}) associated to $F_0.$ From (\ref{H_equation}) we have
$$\ddot{X}=\dot{p_X}=\frac{F'(X)}{F^3(X)}C=-CkX,$$
and hence $X(t)=a\cos(\sqrt{CK}t)+b\sin(\sqrt{CK}t).$ Initial conditions $X(0)=-\lambda^0_C$ and $\dot{X}(0)=p_X(0)=0$ give us that $a=-\lambda^0_C$ and $b=0,$ i.e.
\begin{equation}\label{Solution_X}
X(t)=-\sqrt{\frac{\frac{1}{C}-m}{k}} \cos(\sqrt{CK}t).
\end{equation}
Using (\ref{H_equation}) and (\ref{Solution_X}), a direct computation shows that
\begin{align*}
\Theta_{F_0}(C) & =\int_0^{\frac{2\pi}{\sqrt{Ck}}} \dot{\theta}(t)dt=\int_0^{\frac{2\pi}{\sqrt{Ck}}} \frac{\sqrt{C}}{F^2(X(t))} dt = \int_0^{\frac{2\pi}{\sqrt{Ck}}} \frac{\sqrt{C}}{\frac{1}{kX^2(t)+m}} dt \\
& =\sqrt{C}\int_0^{\frac{2\pi}{\sqrt{Ck}}} \left( \left( \frac{1}{C}-m \right) \cos^2(\sqrt{Ck}t)+m \right )dt=\frac{\pi}{\sqrt{k}} \left( \frac{1}{C}+m \right).
\end{align*}
From $\sqrt{C}<\max F_0 = \frac{1}{\sqrt{m}}$ we have that $\frac{1}{C}>m$ and thus $\Theta_{F_0}(C)>2\pi \frac{m}{\sqrt{k}}.$ Since $m>\sqrt{k}$ it follows that $\Theta_{F_0}(C)>2\pi$ for all $\min F_0<\sqrt{C}<\max F_0.$

Finally, let us show that for $\varepsilon>0$ we may smoothen $F_0$ on intervals $[-1,-1+\varepsilon]$ and $[1-\varepsilon,1]$ in such a way that newly obtained $F_\varepsilon$ also satisfies $\Theta_{F_\varepsilon
}(C)>2\pi$ for all $\min F_\varepsilon<\sqrt{C}<\max F_\varepsilon.$ To this end, let $F_\varepsilon$ be such that $F_\varepsilon|_{[-1+\varepsilon,1-\varepsilon]} = F_0|_{[-1+\varepsilon,1-\varepsilon]},$ $F_\varepsilon\geq F_0$ elsewhere, $|F_\varepsilon'(X)|\leq |F_0'(X)|< 1$ for all $X\in[-1,1],$ $F_\varepsilon$ extends 2-periodically to a smooth function on $\R$ and $F_\varepsilon\xrightarrow{C^0}F_0$ as $\varepsilon \rightarrow 0.$ Denote $F_\varepsilon^{-1}(\sqrt{C})=\{-\lambda_C,\lambda_C \},~\lambda_C>0.$ Since $F_\varepsilon\geq F_0$ we have that $\lambda_C^0\leq \lambda_C.$ Now, note that if $\lambda^0_C \leq 1-\varepsilon,$ it holds $\lambda_C=\lambda^0_C$ as well as $\Theta_{F_\varepsilon}(C)=\Theta_{F_0}(C)>2\pi,$ because two function coincide on $[-1+\varepsilon, 1-\varepsilon].$ If however $1-\varepsilon<\lambda^0_C\leq \lambda_C<1,$ from (\ref{Theta}) we obtain
$$\Theta_{F_\varepsilon}(C)  =2\sqrt{C}\int_{-\lambda_C}^{\lambda_C}\frac{dX}{F_\varepsilon(X)\sqrt{F^2_\varepsilon(X)-C}}>2\sqrt{C}\int_{-\lambda_C+\varepsilon}^{\lambda_C-\varepsilon} \frac{dX}{F_\varepsilon(X)\sqrt{F^2_\varepsilon(X)-C}}.$$
Since $\lambda^0_C\leq \lambda_C$ we have
\begin{align*}
2\sqrt{C}\int_{-\lambda_C+\varepsilon}^{\lambda_C-\varepsilon} \frac{dX}{F_\varepsilon(X)\sqrt{F^2_\varepsilon(X)-C}} & \geq 2\sqrt{C}\int_{-\lambda^0_C+\varepsilon}^{\lambda^0_C-\varepsilon} \frac{dX}{F_\varepsilon
(X)\sqrt{F^2_\varepsilon(X)-C}}=\\
& =2\sqrt{C}\int_{-\lambda^0_C+\varepsilon}^{\lambda^0_C-\varepsilon} \frac{dX}{F_0(X)\sqrt{F^2_0(X)-C}}.
\end{align*}
Same change of variables used in (\ref{Theta}) gives us
\begin{align*}
\Theta_{F_\varepsilon}(C) & > 2\sqrt{C}\int_{-\lambda^0_C+\varepsilon}^{\lambda^0_C-\varepsilon} \frac{dX}{F_0(X)\sqrt{F^2_0(X)-C}} = 2\sqrt{C}\int_{X^{-1}(-\lambda^0_C+\varepsilon)}^{X^{-1}(\lambda^0_C-\varepsilon)} \frac{dt}{F^2_0(X(t))} \\
&=2 \sqrt{C}\int\limits_{\frac{\arccos\left(1-\frac{\varepsilon}{\lambda^0_C}\right)}{\sqrt{Ck}}}^{\frac{\arccos\left(-1+\frac{\varepsilon}{\lambda^0_C}\right)}{\sqrt{Ck}}} (kX^2(t)+m)dt > \frac{2m}{\sqrt{k}} \left( \arccos\left(-1+\frac{\varepsilon}{\lambda^0_C}\right) - \arccos\left(1-\frac{\varepsilon}{\lambda^0_C}\right)\right).    
\end{align*}
Since $\lambda^0_C\geq 1-\varepsilon$, it follows that $\frac{\varepsilon}{\lambda^0_C}\leq \frac{\varepsilon}{1-\varepsilon}\rightarrow 0$ when $\varepsilon\rightarrow 0,$ independently of $C.$ Moreover, by the assumption $\sqrt{k}<m,$ so for small enough $\varepsilon$ we have that $\Theta_{F_\varepsilon}(C)>2\pi$ for all $\min F_\varepsilon < \sqrt{C}<\max F_\varepsilon.$ As explained above, this implies that the only closed geodesics in class $\alpha$ of a metric induced by $F_\varepsilon$ are $\check{\gamma}$ and $\hat{\gamma}$ which together with Lemma \ref{Critical_points} concludes the proof. \end{proof}

\end{document}